\documentclass{amsart}
\usepackage{amsmath,amssymb,amsthm,bm,bbm}
\usepackage{amscd}
\usepackage{graphicx,enumerate}
\usepackage{layout}
\usepackage{longtable}
\usepackage[all,cmtip]{xy}

\theoremstyle{plain}
\newtheorem{theorem}{Theorem}[section]
\newtheorem*{theorem-nn}{Theorem}
\newtheorem{lemma}[theorem]{Lemma}
\newtheorem{proposition}[theorem]{Proposition}
\newtheorem*{proposition-nn}{Proposition}

\theoremstyle{definition}
\newtheorem{definition}[theorem]{Definition}

\theoremstyle{remark}

\newcommand{\bZ}{\mathbbm{Z}}\newcommand{\bQ}{\mathbbm{Q}}
\newcommand{\bE}{\mathbbm{E}}
\newcommand{\bC}{\mathbbm{C}}
\newcommand{\bF}{\mathbbm{F}}

\newcommand{\bk}{\mathbbm{k}}
\newcounter{sub}
{\begin{list}{(\arabic{sub})}{\usecounter{sub}%
\setlength{\leftmargin}{2em}}}{\end{list}}

\def\fn#1{\operatorname{#1}} 

\title[Degree three unramified cohomology groups]
{Degree three unramified cohomology groups
and Noether's problem for groups of order $243$}

\author[A. Hoshi]{Akinari Hoshi}
\address{Department of Mathematics, Niigata University, Niigata 950-2181,
Japan}
\email{hoshi@math.sc.niigata-u.ac.jp}

\author[M. Kang]{Ming-chang Kang}
\address{Department of Mathematics, National Taiwan University, Taipei, Taiwan}
\email{kang@math.ntu.edu.tw}

\author[A. Yamasaki]{Aiichi Yamasaki}
\address{Department of Mathematics, Kyoto University, Kyoto 606-8502, Japan}
\email{aiichi.yamasaki@gmail.com}

\thanks{
This work was partially supported by JSPS KAKENHI Grant Numbers
 25400027, 16K05059, 19K03418.
Parts of the work were finished when the
first-named author and the third-named author were visiting the
National Center for Theoretic Sciences (Taipei), whose
support is gratefully acknowledged.}


\begin{document}

\begin{abstract}
Let $k$ be a field and $G$ be a finite group acting on the rational
function field $k(x_g : g\in G)$ by $k$-automorphisms defined as
$h(x_g)=x_{hg}$ for any $g,h\in G$.
We denote the fixed field $k(x_g : g\in G)^G$ by $k(G)$.
Noether's problem asks whether $k(G)$ is rational (= purely transcendental)
over $k$.
It is well-known that if $\bC(G)$ is stably rational over $\bC$, then all
the unramified cohomology groups $H_{\rm nr}^i(\bC(G),\bQ/\bZ)=0$ for
$i \ge 2$.
Hoshi, Kang and Kunyavskii \cite{HKK} showed that,
for a $p$-group of order $p^5$ ($p$: an odd prime number),
$H_{\rm nr}^2(\bC(G),\bQ/\bZ)\neq 0$ if and only if $G$ belongs to the
isoclinism family $\Phi_{10}$. When $p$ is an odd prime number, Peyre \cite{Pe3} and Hoshi, Kang and Yamasaki \cite{HKY1} exhibit some $p$-groups $G$ which are of the form of a central extension of certain elementary abelian $p$-group by another one with $H_{\rm nr}^2(\bC(G),\bQ/\bZ)= 0$ and $H_{\rm nr}^3(\bC(G),\bQ/\bZ) \neq 0$. However, it is difficult to tell whether $H_{\rm nr}^3(\bC(G),\bQ/\bZ)$ is non-trivial if $G$ is an arbitrary finite group.
In this paper, we are able to determine $H_{\rm nr}^3(\bC(G),\bQ/\bZ)$
where $G$ is any group of order $p^5$ with $p=3, 5, 7$.
Theorem 1.
Let $G$ be a group of order $3^5$.
Then $H_{\rm nr}^3(\bC(G),\bQ/\bZ)\neq 0$ if and only if
$G$ belongs to the isoclinism family $\Phi_7$.
Theorem 2. If $G$ is a group of order $3^5$, then the fixed field $\bC(G)$ is rational if and only if $G$ does not belong to the isoclinism families $\Phi_{7}$ and $\Phi_{10}$.
Theorem 3. Let $G$ be a group of order $5^5$ or $7^5$.
Then $H_{\rm nr}^3(\bC(G),\bQ/\bZ)\neq 0$ if and only if
$G$ belongs to the isoclinism families $\Phi_6$, $\Phi_7$ or $\Phi_{10}$.
Theorem 4. If $G$ is the alternating group $A_n$, the Mathieu group $M_{11}$, $M_{12}$, the Janko group $J_1$ or the group $PSL_2(\bF_q)$, $SL_2(\bF_q)$, $PGL_2(\bF_q)$ (where $q$ is a prime power), then
$H_{\rm nr}^d(\bC(G),\bQ/\bZ)= 0$ for any $d \ge 2$.
Besides the degree three unramified cohomology groups,
we compute also the stable cohomology groups.
\end{abstract}

%
\maketitle

\tableofcontents

%
\section{Introduction}\label{s1}

Let $k$ be a field and $k[x_1,\ldots,x_n]$ be the polynomial ring in $n$ variables over $k$. A rational function field over $k$, denoted by $k(x_1,\ldots,x_n)$, is a field which is $k$-isomorphic to the quotient field of $k[x_1,\ldots,x_n]$, the polynomial ring in $n$ variables over $k$ for some non-negative integer $n$.

Suppose that $L$ is a finitely generated field extension of $k$. Then
$L$ is called $k$-{\it rational} (or {\it rational over $k$}) if $L$ is purely transcendental over $k$,
i.e.\ $L$ is isomorphic to a rational function field over $k$. $L$ is called {\it stably $k$-rational} if $L(y_1,\ldots,y_m)$ is $k$-rational for some $y_1,\ldots, y_m$ which are algebraically independent over $L$.
If $k$ is an infinite field, then the field extension $L$ of $k$ is said to be
{\it retract $k$-rational} if there exists a $k$-algebra
$A$ contained in $L$ such that (i) $L$ is the quotient field of
$A$, (ii) for some non-negative integer $n$, there exist a non-zero polynomial $f\in
k[x_1,\ldots,x_n]$ and $k$-algebra homomorphisms $\varphi\colon A\to
k[x_1,\ldots,x_n][1/f]$ and $\psi\colon k[x_1,\ldots,x_n][1/f]\to
A$ satisfying $\psi\circ\varphi =1_A$ (see \cite{Sa2} and \cite{Ka} for details).
$L$ is called {\it $k$-unirational} if $L$ is $k$-isomorphic to a subfield of some $k$-rational field extension of $k$.

Two finitely generated field extensions $L_1, L_2$ of $k$ are {\it stably $k$-isomorphic} if $L_1(y_1,\ldots,y_m)$ is $k$-isomorphic to $L_2(z_1,\ldots,z_n)$ where $y_1,\ldots, y_m$ (resp. $z_1,\ldots,z_n$) are algebraically independent over $L_1$ (resp. $L_2$).

For an infinite field $k$, it is easy to see that ``$k$-rational"
$\Rightarrow$ ``stably $k$-rational"
$\Rightarrow$ ``retract $k$-rational"
$\Rightarrow$ ``$k$-unirational".

\begin{definition}\label{d1.1}
Let $G$ be a finite group acting on the rational function field
$k(x_g : g\in G)$ by $k$-automorphisms defined by $h(x_g)=x_{hg}$ for any $g,h\in G$.
We denote by $k(G)$ the fixed subfield $k(x_g : g\in G)^G$ of $k(x_g : g\in G)$.

Emmy Noether \cite{No} asked
whether $k(G)$ is rational over $k$.
This is called Noether's problem for $G$ over $k$. It is a special form of the famous L\"uroth problem in algebraic geometry. On the other hand, Noether's problem is related to the inverse Galois problem,
to the existence of generic $G$-Galois extensions over $k$, and
to the existence of versal $G$-torsors over $k$-rational field extensions. For details, see the survey papers of Swan \cite{Sw2}, Manin and Tsfasman \cite{MT}, Colliot-Th\'el\`ene and Sansuc \cite{CTS} and Serre's expository lectures \cite[pages 86--92]{GMS}.
\end{definition}

First of all, we recall some known results of Noether's problem.

\begin{theorem}[Fischer {\cite{Fi}, see also \cite[Theorem 6.1]{Sw2}}]\label{t1.2}
Let $G$ be a finite abelian group of exponent $e$ and
$k$ be a field containing a primitive $e$-th root of unity.
Then $k(G)$ is $k$-rational.
\end{theorem}

\begin{theorem}[Kuniyoshi \cite{Ku}, Gasch\"utz \cite{Ga}] \label{t1.3}
Let $k$ be a field with $\fn{char}k=p>0$ and $G$ be a finite $p$-group.
Then $k(G)$ is $k$-rational.
\end{theorem}

\begin{theorem}[Chu and Kang \cite{CK}] \label{t1.4}
Let $p$ be any prime number and $G$ be a $p$-group of order $\le p^4$
and of exponent $e$. If $k$ is a field containing a primitive $e$-th root
of unity, then $k(G)$ is $k$-rational.
\end{theorem}

\begin{theorem}[Chu, Hu, Kang and Prokhorov \cite{CHKP}]\label{t1.5}
Let $G$ be a group of order $32$ and of exponent $e$.
If $k$ is a field containing a primitive $e$-th root of unity,
then $k(G)$ is $k$-rational.
\end{theorem}

It is Swan who produced the first counter-example to Noether's problem: $\bQ(C_{47})$ is not $\bQ$-rational \cite{Sw1} where $C_{47}$ is the cyclic group of order $47$. Noether's problem for finite abelian groups over any non-closed field was studied subsequently by Voskresenskii, Endo and Miyata, Lenstra, etc. in 1970's (see \cite{Sw2}). However, it was still unknown at that time whether there was a non-abelian $p$-group $G$ such that $\bC(G)$ was not $\bC$-rational.

In 1984, Saltman found that, for any prime number $p$, there was a $p$-group $G$ with order $p^9$ such that $\bC(G)$ was not stably $\bC$-rational \cite{Sa1}. The idea of Saltman is to compute the unramified Brauer group of $\bC(G)$ over $\bC$, denoted by ${\rm Br}_{\rm nr}(\bC(G)/\bC)$, which is an obstruction to the rationality problem. It is known that, if $k$ is an infinite field and $L$ is retract $k$-rational, it is necessary that ${\rm Br}_{\rm nr}(L/k)=0$. Consequently, if ${\rm Br}_{\rm nr}(\bC(G)/\bC) \neq 0$, then $\bC(G)$ is not retract $\bC$-rational (and is not $\bC$-rational a priori). For the definition of the unramified Brauer group and its basic property, see Definition \ref{d2.1} and Proposition \ref{p2.2}.

The notion of the unramified Brauer group is generalized by
Colliot-Th\'el\`ene and Ojanguren \cite{CTO} to the higher degree unramified cohomology groups. If $K$ is a finitely generated  field extension of $\bC$, the unramified cohomology group of $K$ over $\bC$ of degree $i$ ($ i \ge 2$) with coefficients $\bQ/\bZ$ is denoted by $H_{\rm nr}^i(K/\bC, \bQ/\bZ)$; see Definition \ref{d2.6}. Note that the degree two unramified cohomology group $H_{\rm nr}^2(K/\bC, \bQ/\bZ)$ is isomorphic to ${\rm Br}_{\rm nr}(K/\bC)$, the unramified Brauer group. Moreover, if $K$ is retract $\bC$-rational, then $H_{\rm nr}^i(K/\bC, \bQ/\bZ)=0$. We emphasize that the vanishing of $H_{\rm nr}^i(K/\bC, \bQ/\bZ)$ (where $i \ge 2$) is just a necessary condition for $K$ to be retract $\bC$-rational. Thus the condition that $H_{\rm nr}^2(K/\bC, \bQ/\bZ)=H_{\rm nr}^3(K/\bC, \bQ/\bZ)=0$ is not sufficient to guarantee that $K$ is retract $\bC$-rational. See Theorem \ref{t2.9}.

The vanishing of the degree two unramified cohomology group does not imply the same result for the degree three one. In fact, Colliot-Th\'el\`ene and Ojanguren were able to prove the following result.

\begin{theorem}[{Colliot-Th\'el\`ene and Ojanguren \cite[Section 3]{CTO}}]\label{t1.6}
There is a field extension $K$ over $\bC$ with ${\rm trdeg}_{\bC}K=6$ satisfying that $H_{\rm nr}^2(K/\bC, \bQ/\bZ)=0$, but $H_{\rm nr}^3(K/\bC, \bQ/\bZ) \neq 0$. Thus $K$ is not retract $\bC$-rational.
\end{theorem}

Thus, to prove that $\bC(G)$ is not rational, it suffices to find some integer $i$ such that $H_{\rm nr}^i(K/\bC, \bQ/\bZ)$ is non-trivial. However, we cannot conclude that $\bC(G)$ is rational even if $H_{\rm nr}^i(K/\bC, \bQ/\bZ)$ vanishes for all $i \ge 2$.

We will use the degree three unramified cohomology group to study $\bC(G)$ where $G$ is a non-abelian finite group. In case of the degree two unramified cohomology group, Bogomolov found a powerful formula for computing $H_{\rm nr}^2(\bC(G)/\bC, \bQ/\bZ)$ (the Bogomolov multiplier, see Theorem \ref{t2.3}). Saltman \cite{Sa5} and Peyre \cite{Pe1} endeavored to search for a similar formula of the degree three unramified cohomology group.

Indeed, such a formula was found by Peyre in terms of group cohomology \cite[Theorem 1, page 198]{Pe3} (also see Definition \ref{d2.11} and Theorem \ref{t2.14}). Unfortunately, this formula is so complicated that it requires more efforts to apply it to compute effectively $H_{\rm nr}^3(\bC(G)/\bC, \bQ/\bZ)$  for a specific group $G$.

Using this formula (and the ideas in deriving this formula), Peyre succeeded in constructing a ``special" subgroup of $H_{\rm nr}^3(\bC(G)/\bC, \bQ/\bZ)$ in the case when $G$ is a central extension of a vector group by another one, i.e. $0 \to V \to G \to U \to 0$ is a short exact sequence where $U$ and $V$ are elementary abelian $p$-groups, i.e. isomorphic to vector spaces over the finite field $\bF_p$ (where $p$ is an odd prime number) \cite[Theorem 2, page 210]{Pe3}. If this ``special" subgroup is non-zero, then $H_{\rm nr}^3(\bC(G)/\bC, \bQ/\bZ)$ is non-zero automatically. In this way, Peyre found a group $G$ of order $p^{12}$ such that $H_{\rm nr}^3(\bC(G)/\bC,\bQ/\bZ)\ne 0$. For emphasis, we record Peyre's theorem as follows:

\begin{theorem}[Peyre {\cite[Theorem 2]{Pe3}}]\label{t1.7}
Let $p$ be any odd prime number. Then there exists a $p$-group $G$ of order $p^{12}$ satisfying that {\rm (i)} there is a short exact sequence $0 \to V \to G \to U \to 0$ where $V$ is contained in the center of $G$, both $U$ and $V$ are elementary abelian $p$-groups with ${\rm dim}_{\bF_p}U={\rm dim}_{\bF_p}V=6$, and {\rm (ii)} $H_{\rm nr}^2(\bC(G)/\bC,\bQ/\bZ)=0$ and $H_{\rm nr}^3(\bC(G)/\bC,\bQ/\bZ)\ne 0$.

Consequently $\bC(G)$ is not retract $\bC$-rational and therefore it is not $\bC$-rational.
\end{theorem}

After a careful examination of Peyre's method, we may decrease the group order in Theorem \ref{t1.7} to $p^9$ by decreasing the order of $V$ to $p^3$.

\begin{theorem}[Hoshi, Kang and Yamasaki {\cite[Theorem 1.4]{HKY1}}]\label{t1.8}
Let $p$ be any odd prime number. Then there exists a $p$-group $G$ of order $p^{9}$ of the same form as in Theorem \ref{t1.7} with ${\rm dim}_{\bF_p}U=6$ and ${\rm dim}_{\bF_p}V=3$ such that
$H_{\rm nr}^2(\bC(G)/\bC,\bQ/\bZ)=0$ and $H_{\rm nr}^3(\bC(G)/\bC,\bQ/\bZ)\ne 0$. In particular, $\bC(G)$ is not retract  $\bC$-rational.
\end{theorem}

So far as we know, it was unknown how to compute $H_{\rm nr}^3(\bC(G)/\bC, \bQ/\bZ)$ by applying Peyre's formula \cite[Theorem 1, page 198]{Pe3} to groups $G$ not of the form as that in Theorem \ref{t1.7}.

Before stating the main results of this paper, we digress to recall the definition of the isoclinism family of finite groups, a notion due to P. Hall \cite[page 133]{Ha}.

\begin{definition} \label{d1.9}
Let $G$ be a finite group, $Z(G)$ be the center of $G$ and
$[G,G]$ be the commutator subgroup of $G$.
Two finite groups $G_1$ and $G_2$ are called {\it isoclinic} if there exist
group isomorphisms $\theta\colon G_1/Z(G_1) \to G_2/Z(G_2)$ and
$\phi\colon [G_1,G_1]\to [G_2,G_2]$ such that $\phi([g,h])$
$=[g',h']$ for any $g,h\in G_1$ with $g'\in \theta(gZ(G_1))$, $h'\in
\theta(hZ(G_1))$:
\[\xymatrix{
G_1/Z_1\times G_1/Z_1 \ar[d]_{[\cdot,\cdot]} \ar[r]^{(\theta,\theta)}
\ar@{}[dr]| \circlearrowleft
& G_2/Z_2\times G_2/Z_2 \ar[d]_{[\cdot,\cdot]} \\
[G_1, G_1] \ar[r]^\phi & [G_2, G_2].\\
}\]
Denote by $G_n(p)$ the set of all the non-isomorphic $p$-groups of order $p^n$.
In $G_n(p)$, consider an equivalence relation: two groups $G_1$ and $G_2$ are
equivalent if and only if they are isoclinic.
Each equivalence class of $G_n(p)$ is called an {\it isoclinism family}.
\end{definition}

For examples, by M. Hall and Senior \cite{HS}, there are $267$ groups of order $64$ in total, while there are only $27$ isoclinism families $\Phi_1,\ldots,\Phi_{27}$ for these groups
(see \cite[Table I]{JNO}).

It turns out that the unramified cohomology group of a finite group is an invariant of the isoclinism family by the following theorem of Bogomolov and B\"ohning which is the answer to \cite[Question 1.11]{HKK}.

\begin{theorem}[{Bogomolov and B\"ohning \cite[Theorem 6]{BB1}}]\label{t1.10}
If $G_1$ and $G_2$ are isoclinic finite groups,
then $\bC(G_1)$ and $\bC(G_2)$ are stably $\bC$-isomorphic.
In particular, $H_{\rm nr}^i(\bC(G_1)/\bC, \bQ/\bZ)$
$\overset{\sim}{\longrightarrow}$ $H_{\rm nr}^i(\bC(G_2)/\bC,\bQ/\bZ)$.
\end{theorem}

Now we return to Noether's problem. By Theorem \ref{t1.4} and Theorem \ref{t1.5}, we will concentrate on the rationality problem of $\bC(G)$ where $G$ is a group of order $2^n$ with $n \ge 6$ or a group of order $p^m$ where $p$ is an odd prime number with $m \ge 5$.

Here is the result for groups of order $2^6$.

\begin{theorem}[{Chu, Hu, Kang and Kunyavskii \cite{CHKK}}]\label{t1.11}
Let $G=G(2^6,i)$, $1\leq i\leq 267$, be the $i$-th group of order $64$
in the GAP database {\rm \cite{GAP}}.\\
{\rm (1) (\cite[Theorem 1.8]{CHKK})}
$H_{\rm nr}^2(\bC(G)/\bC, \bQ/\bZ)\ne 0$ if and only if $G$ belongs to the isoclinism family
$\Phi_{16}$, i.e. $G=G(2^6,i)$ with $149\le i\le 151$,
$170\le i\le 172$, $177\le i\le 178$ or $i=182$.
Moreover, if $H_{\rm nr}^2(\bC(G)/\bC, \bQ/\bZ)\neq 0$, then $H_{\rm nr}^2(\bC(G)/\bC, \bQ/\bZ)\simeq \bZ/2\bZ$
$($see \cite[Remark, page 424]{Ka} for this statement$)$;\\
{\rm (2) (\cite[Theorem 1.10]{CHKK})}
If $H_{\rm nr}^2(\bC(G)/\bC, \bQ/\bZ)= 0$,
then $\bC(G)$ is $\bC$-rational except
possibly for five groups which belong to $\Phi_{13}$, i.e.
$G=G(2^6,i)$ with $241\le i\le 245$.
\end{theorem}

In the above theorem, it is still unknown whether $\bC(G)$ is $\bC$-rational or not if $G$ belongs to $\Phi_{13}$.

\bigskip
Now consider $p$-groups of order $p^5$ where $p$ is an odd prime number.

For $p\geq 5$ (resp. $p=3$), there exist $2p+61+\gcd\{4,p-1\}+2\gcd\{3,p-1\}$
(resp. $67$) groups $G$ of order $p^5$ (resp. $3^5$) which are classified into ten
isoclinism families $\Phi_1,\ldots,\Phi_{10}$ (see \cite[Section 4]{Ja}). For groups $G$ of order $3^5$, by using computer computing, Moravec showed that $H_{\rm nr}^2(\bC(G)/\bC, \bQ/\bZ)\ne 0$ if and only if $G$ belongs to
the isoclinism family $\Phi_{10}$ \cite{Mo}. A theoretic proof was found and Moravec's result was generalized to any odd prime number $p$ by Hoshi, Kang and Kunyavskii.

\begin{theorem}[{Hoshi, Kang and Kunyavskii \cite[Theorem 1.12]{HKK}, see \cite[page 424]{Ka} for the last statement}] \label{t1.12}
Let $p$ be any odd prime number and $G$ be a group of order $p^5$. Then
$H_{\rm nr}^2(\bC(G)/\bC, \bQ/\bZ)\ne 0$ if and only if $G$ belongs to
the isoclinism family $\Phi_{10}$.
Moreover, if $H_{\rm nr}^2(\bC(G)/\bC, \bQ/\bZ)\neq 0$, then $H_{\rm nr}^2(\bC(G)/\bC, \bQ/\bZ)\simeq \bZ/p\bZ$.
\end{theorem}

Note that the total number of groups in the isoclinism family $\Phi_{10}$ for groups of order $p^5$ (where $p\ge5$) is
$1+\gcd\{4,p-1\}+\gcd \{3,p-1\}$ (\cite[page 621]{Ja}), while there are precisely three groups in the isoclinism family $\Phi_{10}$ for groups of order $3^5$.

\bigskip
Now it is natural to ask whether $\bC(G)$ is $\bC$-rational if $G$ is a group of order $p^5$ (where $p$ is an odd prime number) satisfying that $H_{\rm nr}^2(\bC(G)/\bC, \bQ/\bZ)= 0$. As we remarked before, this is not a routine job; it is not a corollary of the fact that $H_{\rm nr}^2(\bC(G)/\bC, \bQ/\bZ)= 0$. The following theorem almost solved the case when $G$ is a group of order $243$.

\begin{theorem}[{Chu, Hoshi, Hu and Kang \cite[Theorem 1.13]{CHHK}}]
\label{t1.13}
Let $G$ be a group of order $243$. 
If $H_{\rm nr}^2(\bC(G)/\bC, \bQ/\bZ)=0$, then $\bC(G)$ is $\bC$-rational
except possibly for the five groups $G$ of the isoclinism family $\Phi_7$,
i.e. $G=G(3^5,i)$ with $56 \le i \le 60$.
\end{theorem}

\bigskip
The goal of this paper is to study $H_{\rm nr}^3(\bC(G)/\bC, \bQ/\bZ)$ by applying Peyre's formula \cite[Theorem 1, page 198]{Pe3}. Indeed, we are able to do it when $G$ is any group of order $3^5$, $5^5$ or $7^5$. A byproduct of it is that it enables us to finish the unsolved rationality problem of Theorem \ref{t1.13}.

The main results of our paper are the following two theorems.

\begin{theorem}\label{t1.14}
Let $G$ be a group of order $3^5$.
Then $H_{\rm nr}^3(\bC(G)/\bC,\bQ/\bZ)\neq 0$ if and only if
$G$ belongs to the isoclinism family $\Phi_7$.
Moreover, if $H_{\rm nr}^3(\bC(G)/\bC,\bQ/\bZ)\neq 0$,
then $H_{\rm nr}^3(\bC(G)/\bC,\bQ/\bZ)\simeq \bZ/3\bZ$.
\end{theorem}
\begin{table}[h]\vspace*{-2mm}
\begin{tabular}{c|cccccccccc}
$|G|=3^5$ & $\Phi_1$ & $\Phi_2$ & $\Phi_3$ & $\Phi_4$ & $\Phi_5$ & $\Phi_6$ & $\Phi_7$ & $\Phi_8$ & $\Phi_9$ & $\Phi_{10}$\\\hline
$H_{\rm nr}^2(\bC(G)/\bC,\bQ/\bZ)$ & 0 & 0 & 0 & 0 & 0 & 0 & 0 & 0 & 0 & $\bZ/3\bZ$ \\
$H_{\rm nr}^3(\bC(G)/\bC,\bQ/\bZ)$ & 0 & 0 & 0 & 0 & 0 & 0 & $\bZ/3\bZ$ & 0 & 0 & 0
\end{tabular}\\
\vspace*{2mm}
Table $1$: $H_{\rm nr}^2(\bC(G)/\bC,\bQ/\bZ)$ and $H_{\rm nr}^3(\bC(G)/\bC,\bQ/\bZ)$ for groups $G$
of order $3^5$
\end{table}\vspace*{-2mm}
\begin{theorem}\label{t1.15}
Let $G$ be a group of order $p^5$ where $p=5$ or $p=7$.
Then $H_{\rm nr}^3(\bC(G)/\bC,\bQ/\bZ)\neq 0$ if and only if
$G$ belongs to the isoclinism family $\Phi_6$, $\Phi_7$ or $\Phi_{10}$.
Moreover, if $H_{\rm nr}^3(\bC(G)/\bC,\bQ/\bZ)\neq 0$,
then $H_{\rm nr}^3(\bC(G)/\bC,\bQ/\bZ)\simeq \bZ/p\bZ$.
\end{theorem}
\begin{table}[h]\vspace*{-2mm}
\begin{tabular}{c|cccccccccc}
$|G|=p^5$ $(p=5,7)$ & $\Phi_1$ & $\Phi_2$ & $\Phi_3$ & $\Phi_4$ & $\Phi_5$ & $\Phi_6$ & $\Phi_7$ & $\Phi_8$ & $\Phi_9$ & $\Phi_{10}$\\\hline
$H_{\rm nr}^2(\bC(G)/\bC,\bQ/\bZ)$ & 0 & 0 & 0 & 0 & 0 & 0 & 0 & 0 & 0 & $\bZ/p\bZ$ \\
$H_{\rm nr}^3(\bC(G)/\bC,\bQ/\bZ)$ & 0 & 0 & 0 & 0 & 0 & $\bZ/p\bZ$ & $\bZ/p\bZ$ & 0 & 0 & $\bZ/p\bZ$
\end{tabular}\\
\vspace*{2mm}
Table $2$: $H_{\rm nr}^2(\bC(G)/\bC,\bQ/\bZ)$ and $H_{\rm nr}^3(\bC(G)/\bC,\bQ/\bZ)$ for groups $G$
of order $5^5$ and $7^5$
\end{table}

{}Comparing with Theorem \ref{t1.6} and Theorem \ref{t2.9}, we find a new phenomenon from Theorem \ref{t1.14}: $H_{\rm nr}^2(\bC(G)/\bC,$ $\bQ/\bZ)\neq 0$, but $H_{\rm nr}^3(\bC(G)/\bC,\bQ/\bZ)= 0$ when the group $G$ belongs to $\Phi_7$. Also from Theorem \ref{t1.15} when the group $G$ belongs to $\Phi_{10}$, we find that both $H_{\rm nr}^2(\bC(G)/\bC,\bQ/\bZ)$ and $H_{\rm nr}^3(\bC(G)/\bC,\bQ/\bZ)$  are non-zero.

\bigskip
Theorem \ref{t1.14} enables us to solve the unsettled case of the rationality problem of $\bC(G)$ in Theorem \ref{t1.13}.

\begin{theorem}\label{t1.18}
Let $G$ be a group of order $3^5$. Then $\bC(G)$ is $\bC$-rational $($resp. retract $\bC$-rational$)$ if and only if $G$ belongs to the isoclinism family $\Phi_i$ where $1 \le i \le 6$ or $8 \le i \le 9$.
\end{theorem}

Similarly, Theorem \ref{t1.15} sheds some light on the understanding of the rationality problem of $\bC(G)$ for a group $G$ of order $p^5$ where $p=5$ or $p=7$: If $G$ belongs to the isoclinism family $\Phi_i$ where $i= 6, 7, 10$, then $\bC(G)$ is not retract $\bC$-rational. Applying the same arguments as in \cite[page 238, the first paragraph]{CHHK}, it is easy to deduce that $\bC(G)$ is $\bC$-rational if $G$ is a group of order $5^5$ or $7^5$ and $G$ belongs to the isoclinism family $\Phi_i$ where $1 \le i \le 4$ or $8 \le i \le 9$.

On the other hand, we do not know whether $\bC(G)$ is $\bC$-rational or not if $G$ is a group of order $5^5$ or $7^5$ and $G$ belongs to the isoclinism family $\Phi_5$; the groups in $\Phi_5$ are the so-called extra-special $p$-groups of order $p^5$. In general, if $p$ is any prime number and $m \ge 1$, there are two non-isomorphic extra-special $p$-groups of order $p^{2m+1}$. It is known that $H_{\rm nr}^2(\bC(G)/\bC,\bQ/\bZ)=0$ if $G$ is an extra-special $p$-groups of order $p^3$, $p^5$ for any prime number $p$, or $p^{2m+1}$ with $p \ge 3$ and $m\ge 3$ by \cite[Proposition 2.1]{KK}. However, it is known that $\bC(G)$ is rational if $G$ is an extra-special $p$-groups of order $p^3$, $2^5$, or $3^5$; the answers to the rationality problem for other cases are unknown.

\bigskip
Now we turn to the non-abelian simple groups.

The answers to the rationality problem or the higher unramified cohomology group for non-abelian simple groups are rather scarce, either the affirmative answer or the negative one. The reader is referred to Theorem \ref{t5.1} and Theorem \ref{t5.2} for a brief survey. What is new in this paper is the following theorem about the unramified cohomology groups.

\begin{theorem}\label{t1.16}
Let $K$ be a finitely generated field extension of $\bC$, $\Gamma_K:= {\rm Gal}(K^{\rm sep}/K)$ and $M$ be a torsion $($discrete$)$ $\Gamma_K$-module on which $\Gamma_K$ acts trivially $($e.g. $M= \bQ/\bZ$ or $\bZ/n\bZ$ for some positive integer $n$$)$. Then $H_{\rm nr}^d(\bC(G)/\bC, M)= 0$ for all $d \ge 2$ provided that $G$ is isomorphic to one of the groups: the Mathieu groups $M_{11}, M_{12}$, the Janko group $J_1$, the alternating groups $A_n$ $($where $n \ge 3$$)$, the groups $SL_2(\bF_q)$, $PSL_2(\bF_q)$, $PGL_2(\bF_q)$ where $q$ is any prime power. We do not know the answer to the question whether $\bC(G)$ is $\bC$-rational $($resp. stably $\bC$-rational, retract $\bC$-rational$)$ for almost all the above groups.
\end{theorem}

As mentioned before, the condition $H_{\rm nr}^d(\bC(G)/\bC, \bQ/\bZ)= 0$ (where $d \ge 2$) is not enough to ensure that $\bC(G)$ is $\bC$-rational. However, the situation is not so discouraging, because the vanishing of  $H_{\rm nr}^3(\bC(G)/\bC, \bQ/\bZ)$ is equivalent to the fact that ${\rm Hdg}^4(X,\bZ)={\rm Hdg}^4(X,\bZ)_{\rm alg}$ by the theorem of Colliot-Th\'el\`ene and Voisin (see Theorem \ref{t2.8}). In other words, if $X$ is a smooth projective complex variety whose function field is stably isomorphic to $\bC(G)$ where $G$ is one of the groups listed in Theorem \ref{t1.16}, then every integral Hodge class of $X$ with degree $(2,2)$ is algebraic (see \cite[Section 6.2]{Vo} for details).

\bigskip
We explain some ideas in the proof of Theorem \ref{t1.14}, Theorem \ref{t1.15} and Theorem \ref{t1.16}.

In applying Peyre's formula \cite[Theorem 1, page 198]{Pe3}, we use computer computing to find subgroups  $H_{\rm p}^3(G,\bQ/\bZ)\leq H_{\rm n}^3(G,\bQ/\bZ)\leq H_{\rm nr}^3(G,\bQ/\bZ)\leq H^3(G,\bQ/\bZ)$.
{}From the works of Saltman \cite{Sa5} and Peyre \cite{Pe3}, it is known that $H_{\rm nr}^3(\bC(G)/\bC,\bQ/\bZ)$ is isomorphic to $H_{\rm nr}^3(G,\bQ/\bZ)/H_{\rm n}^3(G,\bQ/\bZ)$. A formula for $H_{\rm nr}^3(G,\bQ/\bZ)$ is proved by Peyre (see Definition \ref{d2.11}). Unfortunately, this formula is too complicated for us to determine the group $H_{\rm nr}^3(G,\bQ/\bZ)$ as a whole. In addition, we do not know how to compute $H_{\rm n}^3(G,\bQ/\bZ)$ in an effective way (see Definition \ref{d2.10}).

An encouraging message is that $H_{\rm p}^3(G,\bQ/\bZ) = H_{\rm n}^3(G,\bQ/\bZ)$ if $G$ is a group of odd order (see Theorem \ref{t2.14}). Thus we take a detour by computing $H^3(G,\bQ/\bZ)/H_{\rm p}^3(G,\bQ/\bZ)$ and then we determine which cohomology classes (in the quotient group) belong to $H_{\rm nr}^3(G,\bQ/\bZ)/H_{\rm p}^3(G,\bQ/\bZ)$ by Peyre's formula in Definition \ref{d2.11}. We would like to point out that, in using Peyre's formulae in Definition \ref{d2.11} and Theorem \ref{t2.14}, in order that the computation is feasible in a personal computer, it is necessary to throw away those subgroups which provide no essential contribution (see Step 2 and Step 5 of Section 3). In this way, we may compute $H_{\rm nr}^3(\bC(G)/\bC,\bQ/\bZ)$ if $G$ is of odd order and the quotient group $H^3(G,\bQ/\bZ)/H_{\rm p}^3(G,\bQ/\bZ)$ is ``small" enough. Fortunately, when $G$ is a group of order $3^5$, $5^5$ or $7^5$, the computing is feasible and may be implemented in a personal computer. This explains the reason why our results are restricted (at least at present) only to groups of order $3^5$, $5^5$ and $7^5$, but not groups of order $11^5$ or $2^6$.

The proof of Theorem \ref{t1.16} may be reduced to the unramified cohomology groups corresponding to Sylow subgroups $G_p$ of the groups $G$ as an application of Lemma \ref{l5.4}. We will show that $\bC(G_p)$ is $\bC$-rational to deduce that $H_{\rm nr}^3(\bC(G)/\bC, M)= 0$. A second proof of Theorem \ref{t1.16}
for $H_{\rm nr}^3(\bC(G)/\bC,\bQ/\bZ)= 0$ with $G=PSL_2(\bF_8), A_6, A_7$ will be given also for the following reasons. First, the main idea of this new proof is essentially the same as that in Theorem \ref{t1.14} and Theorem \ref{t1.15}. Second, this new proof exemplifies new computational techniques: it requires a new ingredient $N^3(G)=0$ which was proved by Saltman (see Theorem \ref{t5.3}). Moreover, in applying Definition \ref{d2.13} and Theorem \ref{t2.14} for computing $H_{\rm p}^3(G,\bQ/\bZ)$, in order that the computer computing is feasible, it is necessary to find an exact sequence of $\bZ[G]$-modules $0 \to \mu \to Q^\ast \to Q \to 0$ (where $Q$ is a permutation $\bZ[G]$-lattice and $Q^\ast$ is $H^1$-trivial) such that the $\bZ$-rank of $Q$ is small enough. See the paragraphs after Theorem \ref{t5.3}.

The computer computing of this paper uses the computer package HAP for GAP (see \cite{HAP} and \cite{GAP}). We device several functions which may be found in Section 6. For examples, HAP has no command to execute the corestriction maps of cohomology groups; so we write the function {\tt Cores} for computing corestriction maps. The function {\tt H4pFromResolution} is used to calculate $H^4_{\rm p}(G, \bZ) \simeq H^3_{\rm p}(G, \bQ/\bZ)$ which is defined in Theorem \ref{t2.14}. The two functions {\tt chooseHI} and
{\tt IsUnramifiedH3} are crucial in our proof which enables us to choose efficiently those pairs of subgroups $(H, I)$ in Definition \ref{d2.11} to decide whether a cohomology class belongs to $H_{\rm nr}^4(\bC(G)/\bC,\bZ)$ or not (see Step 5 of Section 3).

The proof of Theorem \ref{t1.14} and Theorem \ref{t1.15} is given in Section 3. In the proof, theoretic discussions together with computer commands of implementation are provided. The demonstrations for groups of order $3^5$, $5^5$ and $7^5$ are supplied after the proof. The reader may find in Section 6 the algorithm and the computer functions devised for the proofs, such as {\tt chooseHI} and
{\tt IsUnramifiedH3}, etc.

We will emphasize that the computer we use are the usual personal computer. This is the reason why we write the algorithm and the computer functions to be executed by most people in their personal computers within reasonable computing time.

\bigskip
Finally we record a form of the Lyndon-Hochschild-Serre spectral sequence for a direct product group, which will be used several times in this paper. A special case of it was given in Lyndon's paper \cite[Section 8]{Ly}. It became a folklore among experts and was difficult to locate in the literature. We found it in the book of Neukirch, Schmidt and Wingberg \cite{NSW}. Here it is.

\begin{theorem}[{\cite[page 118, Theorem 2.4.6]{NSW}}]\label{t1.17}
Let $G$ and $H$ be finite groups and $A$ be a $\bZ[G \times H]$-module. Then the Lyndon-Hochschild-Serre spectral sequence
\begin{align*}
E_2^{p,q}=H^p(G,H^q(H,A))\Rightarrow H^{p+q}(G\times H,A)
\end{align*}
degenerates at the $E_2$ term, i.e. $E_2^{p,q}=E_3^{p,q}= \cdots =E_{\infty}^{p,q}$. Furthermore,  it splits in the sense that there is a decomposition
\begin{align*}
&H^n(G\times H,A)\simeq \bigoplus_{p+q=n}H^p(G,H^q(H,A)).
\end{align*}

\end{theorem}

\bigskip
Terminology and notations. Throughout this paper, we assume that $G$, $H$, $\Gamma$ are finite groups, while $\Gamma_k$ denotes the absolute Galois group of the field $k$, i.e. $\Gamma_k={\rm Gal}(k^{\rm sep}/k)$ where $k^{\rm sep}$ is the separable closure of $k$. Whenever we write $\gcd\{{\rm char} \, k,n\}=1$, we mean that either ${\rm char}=0$ or ${\rm char}=p > 0$ with $p \nmid n$. We denote $\zeta_n$ a primitive $n$-th root of unity. If we write $\zeta_n \in k$, it is assumed tacitly that $\gcd\{{\rm char} \, k,n\}=1$.

We denote by $C_n$ the cyclic group of order $n$. If $H$ is a subgroup of $G$, $Z_G(H)$ denotes the centralizer of $H$ in $G$. If $R^{\bullet} : \cdots \to R^{n-1} \to R^n \to R^{n+1} \to \cdots$ is a cochain complex and $v \in R^n$ is a cocycle, the cohomology class associated to $v$ is denoted by $[v]$, i.e. $[v]$ is an element of $H^n(R^{\bullet})$. The reader may find the definitions of unexplained terminology about bicomplexes and spectral sequences in the reference books \cite{Mac} and \cite{Ev}.

If $k$ is a field and $G$ is a finite group, the symbols $k(G)$, $\bC(G)$ is defined in Definition \ref{d1.1}; these symbols are scattered throughout this paper.

\bigskip

\section{Preliminaries: the unramified and stable cohomologies}\label{s2}

In this section we will recall the notions of the unramified Brauer group \cite{Sa1}, the Bogomolov multiplier \cite{Bo1}, the higher degree unramified cohomology groups \cite{CTO}, and the constructions designed for computing the degree three unramified cohomology group \cite{Sa5}, \cite{Pe3}.

\begin{definition}[{Saltman \cite[Definition 3.1]{Sa1}, \cite[page 56]{Sa3}}] \label{d2.1}
Let $k \subset K$ be a field extension.
The {\it unramified Brauer group} ${\rm Br}_{\rm nr}(K/k)$
of $K$ over $k$ is defined to be
\[
{\rm Br}_{\rm nr}(K/k)=\bigcap_R {\rm Image} \{ {\rm Br}(R)\to{\rm Br}(K)\}
\]
where ${\rm Br}(R)\to {\rm Br}(K)$ is the natural injective morphism of
Brauer groups, and $R$ runs over all the discrete valuation rings $R$ 
such that $k\subset R\subset K$ and $K$ is the quotient field of $R$.

We omit $k$ from the notation ${\rm Br}_{\rm nr}(K/k)$
and write simply ${\rm Br}_{\rm nr}(K)$ when the base field $k$ is clear
from the context. In particular, we will write ${\rm Br}_{\rm nr}(\bC(G))$ for ${\rm Br}_{\rm nr}(\bC(G)/\bC)$ throughout this paper.
\end{definition}

\begin{proposition}[{Saltman \cite{Sa1}, \cite[Proposition 1.8]{Sa3}}] \label{p2.2}
Suppose that $k$ is an infinite field and $K$ is retract $k$-rational. Then the natural morphism ${\rm Br}(k) \to {\rm Br}(K)$ induces an isomorphism
${\rm Br}(k)\overset{\sim}{\longrightarrow}{\rm Br}_{\rm nr}(K)$.
In particular, if $k$ is an algebraically closed field and $K$ is
retract $k$-rational, then ${\rm Br}_{\rm nr}(K)=0$.
\end{proposition}

The unramified Brauer group is the birational invariant introduced by Saltman to find a group of order $p^9$ such that $\bC(G)$ is not stably $\bC$-rational \cite{Sa1}. This invariant ${\rm Br}_{\rm nr}(K/k)$ is computable and coincides with Grothendieck's Brauer group ${\rm Br}(X)$ when $X$ is a smooth projective variety over $k$ with function field $K$ (see \cite[page 70, Proposition 10.5]{Sa6}). The same invariant ${\rm Br}_{\rm nr}(K/\bC)$ is isomorphic to the birational invariant
$H^3(X,\bZ)_{\rm torsion}$ where $X$ is a smooth complex projective unirational variety $X$ with function
field $K$ (see \cite[page 134, Proposition 6.17]{Vo}). The invariant $H^3(X,\bZ)_{\rm torsion}$ was used by Artin and Mumford \cite{AM} to provide an example of a complex variety which is unirational, but not rational. For details, see \cite[Theorem 1.1 and Corollary]{Bo1}.

Bogomolov found an even more effective formula for computing the unramified Brauer group ${\rm Br}_{\rm nr}(\bC(G))$.

\begin{theorem}[{Bogomolov \cite[Theorem 3.1]{Bo1}, Saltman \cite[Theorem 12]{Sa4}}]\label{t2.3}
Let $G$ be a finite group and $k$ be an algebraically closed field with
{\rm char} $k=0$ or {\rm char} $k=p$ with $p {\not |}$ $|G|$.
Then ${\rm Br}_{\rm nr}(k(G)/k)$ is isomorphic to the group $B_0(G)$ defined by
\[
B_0(G)=\bigcap_A {\rm Ker}\{{\rm res}: H^2(G,\bQ/\bZ)\to H^2(A,\bQ/\bZ)\}
\]
where $A$ runs over all the bicyclic subgroups of $G$ $($a group $A$
is called bicyclic if $A$ is either a cyclic group or a direct
product of two cyclic groups$)$.
\end{theorem}

Since $B_0(G)$ is a subgroup of $H^2(G,\bQ/\bZ)$
which is nothing but the Schur multiplier of $G$ (denoted by $M(G)$ in some literature), Kunyavskii called it {\it the Bogomolov multiplier of $G$}. Thus we find that $B_0(G)$ is isomorphic to ${\rm Br}_{\rm nr}(\bC(G))$; in Theorem \ref{t2.7}, we will find that ${\rm Br}_{\rm nr}(\bC(G))$ is also isomorphic to $H_{\rm nr}^2(\bC(G)/\bC,\bQ/\bZ)$ which appeared already in Theorem \ref{t1.14} and Theorem \ref{t1.15}.

Using the Bogomolov multiplier to compute ${\rm Br}_{\rm nr}(\bC(G))$, Bogomolov was able to decrease the order of the group in Saltman's counter-example. For the convenience of the reader, we record these results as follows.

\begin{theorem} \label{t2.4}
Let $p$ be any prime number and $k$ be any algebraically closed field
with ${\rm char}$ $k\ne p$.\\
{\rm (1) (Saltman \cite[Theorem 3.6]{Sa1})}
There exists a meta-abelian group $G$ of order $p^9$
such that $B_0(G)\ne 0$ and thus $k(G)$ is not $k$-rational;\\
{\rm (2) (Bogomolov \cite[Lemma 5.6]{Bo1})}
There exists a group $G$ of order $p^6$ such that $B_0(G)\ne 0$ and thus $k(G)$ is not $k$-rational.
\end{theorem}

\bigskip
Now we begin to define the higher degree unramified cohomology groups $H_{\rm nr}^i(K/k,\mu_n^{\otimes j})$ (where $i\ge 1$) due to Colliot-Th\'el\`ene and Ojanguren \cite{CTO}. Since $H_{\rm nr}^1(K/\bC, \bQ/\bZ)=0$ if $K$ is the function field of a complex smooth projective unirational variety \cite[page 192]{Pe3}, we will concentrate on $H_{\rm nr}^i(K/k,\mu_n^{\otimes j})$ with $i \ge 2$.

Let $K$ be any field, $K^{\rm sep}$ its separable closure and $\Gamma_K$ be the absolute Galois group of $K$, i.e. $\Gamma_K:= {\rm Gal}(K^{\rm sep}/K)$. If $M$ is a (continuous) $\Gamma_K$-module, define $H^i(K, M):= H^i(\Gamma_K, M)$.

Recall the definition of the $\Gamma_K$-module $\mu_n^{\otimes j}$. Assume that $\gcd\{{\rm char} \, k,n\}=1$. Let $\mu_n$ be the group of all the $n$-th roots of unity contained in $K^{\rm sep}$; thus we may regard $\mu_n$ as a $\Gamma_K$-module. For any $j \ge 1$, $\mu_n^{\otimes j}$ becomes a $\Gamma_K$-module via the diagonal action of $\Gamma_K$. By convention, we write $\mu_n^{\otimes 0}$ for the cyclic group of order $n$, denoted by $\bZ/n \bZ$, on which $\Gamma_K$ acts trivially. For any $j < 0$, $\mu_n^{\otimes j}$ denotes the $\Gamma_K$-module ${\rm Hom} (\mu_n^{\otimes -j}, \bZ/n \bZ)$.

\begin{definition}[{Colliot-Th\'el\`ene and Ojanguren \cite{CTO}, see also \cite[Sections 2--4]{CT}}] \label{d2.5}

Let $K$ be a finitely generated field extension of a field $k$ and $R$ be a discrete valuation ring of $K$ with $k \subset R \subset K$ such that $K$ is the quotient field of $R$. Write $\bk_R$ for the residue field of $R$. By \cite[pages 15--19]{GMS}, \cite[pages 21--22, page 26]{CT}, there is a natural map
\[
r_R: H^i(K,\mu_n^{\otimes j})\to H^{i-1}(\bk_R,\mu_n^{\otimes (j-1)}).
\]
The above map is called {\it the residue map} of $K$ at the place $R$.
\end{definition}

\begin{definition}[{Colliot-Th\'el\`ene and Ojanguren \cite{CTO}, see also \cite[Sections 2--4]{CT}}]\label{d2.6}
Let $n$ be a positive integer and $k$ be a field with {\rm char} $k=0$ or {\rm char} $k=p$ with $p {\not |}$ $n$.
Let $K/k$ be a finitely generated field extension. For any positive integer $i\ge 2 $, any integer $j$, the {\it unramified cohomology group} $H^i_{\rm nr}(K/k,\mu_n^{\otimes j})$
of $K$ over $k$ of degree $i$ is defined to be

\[
H^i_{\rm nr}(K/k,\mu_n^{\otimes j}):=\bigcap_R {\rm Ker}
\{r_R: H^i(K,\mu_n^{\otimes j})\to H^{i-1}(\bk_R,\mu_n^{\otimes (j-1)})\}
\]
where $R$ runs over all the discrete valuation rings $R$ of rank one
such that $k\subset R\subset K$ and $K$ is the quotient field of $R$, and $\bk_R$ is the residue field of $R$.

By \cite[Theorem 4.1.1, page 30]{CT}, if it is assumed furthermore that $K$ is the function field of a complete smooth variety over $k$, the unramified cohomology group $H^i_{\rm nr}(K/k,\mu_n^{\otimes j})$ may be defined as well by
\[
H^i_{\rm nr}(K/k,\mu_n^{\otimes j})=\bigcap_R {\rm Image}
\{ H^i_{\rm \acute{e}t}(R,\mu_n^{\otimes j})\to H^i_{\rm \acute{e}t}(K,\mu_n^{\otimes j})\}
\]
where $R$ runs over all the discrete valuation rings $R$ of rank one
such that $k\subset R\subset K$ and $K$ is the quotient field of $R$.

In case ${\rm char} \, k =0$,
take the direct limit with respect to $n$:
\[
H^i(K/k,\bQ/\bZ(j))=
\lim_{\overset{\longrightarrow}{n}}H^i(K/k,\mu_n^{\otimes j})
\]
and we may define the unramified cohomology group
\[
H_{\rm nr}^i(K/k,\bQ/\bZ(j))
=\bigcap_R {\rm Ker}
\{r_R: H^i(K/k,\bQ/\bZ(j))\to H^{i-1}(\bk_R,\bQ/\bZ(j-1))\}.
\]

We write simply $H^i_{\rm nr}(K,\mu_n^{\otimes j})$ and $H_{\rm nr}^i(K,\bQ/\bZ(j))$ when the base field $k$ is understood.

When $k$ is an algebraically closed field with {\rm char} $k=0$, we will write $H_{\rm nr}^i(K/k,\bQ/\bZ)$ for $H_{\rm nr}^i(K/k,\bQ/\bZ(j))$. If $G$ is a finite group, it is easy to see that $H_{\rm nr}^i(\bC(G),\bQ/\bZ)$ is a subgroup of $H^i(\bC(G), \bQ/\bZ)$ by the above definition.
\end{definition}

\begin{theorem}\label{t2.7}
Let $k$ be an algebraically closed field.
\\
{\rm (1) (Colliot-Th\'el\`ene and Ojanguren \cite[Proposition 1.2]{CTO})}
Suppose that $k$ is a field with {\rm char} $k=0$ or {\rm char} $k=p$ such that $p {\not |}$ $n$. If $K$ and $L$ are stably $k$-isomorphic, then $H_{\rm nr}^i(K/k,\mu_n^{\otimes j}) \overset{\sim}{\longrightarrow} H_{\rm nr}^i(L/k,\mu_n^{\otimes j})$.
In particular, if $K$ is stably $k$-rational, then $H_{\rm nr}^i(K/k,\mu_n^{\otimes j})=0$;\\
{\rm (2) (\cite[Proposition 2.15]{Me},
see also \cite[Remarque 1.2.2]{CTO}, \cite[Sections 2--4]{CT})}
If $K$ is retract $k$-rational, then $H_{\rm nr}^i(K/k,\mu_n^{\otimes j})=0$;\\
{\rm (3) (\cite[Proposition 4.2.3, page 34]{CT})} If $k$ is a field with {\rm char} $k=0$, then ${\rm Br}_{\rm nr}(K/k)\simeq H_{\rm nr}^2(K/k,\bQ/\bZ)$.
\end{theorem}

We remark that, for a smooth complex projective variety $X$, the unramified cohomology group $H^3_{\rm nr}(X,\bQ/\bZ)$ is related to the integral Hodge conjecture. The reader may look into \cite{CTV} and \cite[Section 6.2]{Vo} for the unexplained terminology of the following theorem.

\begin{theorem}[{Colliot-Th\'el\`ene and Voisin \cite{CTV}, see also \cite[Theorem 6.18]{Vo}}]\label{t2.8}
For any smooth projective complex variety $X$, there is an exact sequence
\[
0\rightarrow H^3_{\rm nr}(X,\bZ)\otimes \bQ/\bZ\rightarrow
H^3_{\rm nr}(X,\bQ/\bZ)\rightarrow {\rm Tors}(Z^4(X))\rightarrow 0
\]
where
\[
Z^4(X) = {\rm Hdg}^4(X,\bZ)/{\rm Hdg}^4(X,\bZ)_{\rm alg}
\]
and the lower index ``alg'' means the group of integral Hodge classes which are algebraic.
In particular, if $X$ is rationally connected, then we have
\[
H^3_{\rm nr}(X,\bQ/\bZ)\simeq Z^4(X).
\]
\end{theorem}

The following theorem of Asok computes unramified cohomology groups with coefficients in $\mu_l$ where $l$ is a prime number.

\begin{theorem}[{Asok \cite{As}, see also \cite[Theorem 3]{AMo}}]\label{t2.9}{}~{}\\
{\rm (1) (\cite[Theorem 1]{As})}
For any $n\ge 2$,
there exists a smooth projective complex variety $X$ that is
$\bC$-unirational, for which
$H_{\rm nr}^i(\bC(X),\mu_2^{\otimes i})=0$ for each $1<i<n$, yet
$H_{\rm nr}^n(\bC(X),\mu_2^{\otimes n})\neq 0$, and so
$X$ is not $\mathbb{A}^1$-connected,
nor retract $\bC$-rational;\\
{\rm (2) (\cite[Theorem 3]{As})}
For any prime number $l$ and any $n\geq 2$,
there exists a smooth projective rationally connected complex
variety $Y$ such that
$H_{\rm nr}^n(\bC(Y),\mu_l^{\otimes n})\neq 0$.
In particular, $Y$ is not $\mathbb{A}^1$-connected,
nor retract $\bC$-rational.
\end{theorem}

\bigskip
Finally we recall the constructions of Saltman and Peyre for computing $H^3_{\rm nr}(\bC(G),\bQ/\bZ)$ in \cite {Sa5} and \cite{Pe3}.

Recall that $H^i(\bC(G),\bQ/\bZ)=H^i(\Gamma_K,\bQ/\bZ)$ where $K=\bC(G)$ and
$\Gamma_K$ is the absolute Galois group of $K$. Since the group $G$ is a quotient group of $\Gamma_K$, we may consider the inflation map
\[
\iota: H^i(G,\bQ/\bZ)\rightarrow H^i(\bC(G),\bQ/\bZ).
\]
Both the kernel and the image of the inflation map $\iota$ are significant invariants. Note that $H_{\rm nr}^i(\bC(G)/\bC,\bQ/\bZ)$ is a subgroup of $H^i(\bC(G), \bQ/\bZ)$.

We will define subgroups $H_{\rm p}^3(G,\bQ/\bZ)\leq H_{\rm n}^3(G,\bQ/\bZ)\leq H_{\rm nr}^3(G,\bQ/\bZ)\leq H^3(G,\bQ/\bZ)$. The reader will find readily that the inflation map $\iota$ induces an isomorphism $H_{\rm nr}^3(G,\bQ/\bZ)/H_{\rm n}^3(G,\bQ/\bZ) \to H_{\rm nr}^3(\bC(G),\bQ/\bZ)$, and the natural map $H_{\rm nr}^3(G,\bQ/\bZ)/H_{\rm p}^3(G,\bQ/\bZ) \to H_{\rm nr}^3(G,\bQ/\bZ)/H_{\rm n}^3(G,\bQ/\bZ)$ is an isomorphism up to $2$-torsion, i.e. the index $[H_{\rm n}^3(G,\bQ/\bZ):H_{\rm p}^3(G,\bQ/\bZ)]=2^d$ for some non-negative integer $d$. One of the reasons to introduce $H_{\rm p}^3(G,\bQ/\bZ)$ is that it may be difficult to find $H_{\rm n}^3(G,\bQ/\bZ)$ directly.

\begin{definition}\label{d2.10}
By \cite[page 230, Theorem 5.3]{Sa5}, $H^3_{\rm nr}(\bC(G),\bQ/\bZ)$ is a subgroup of the image of the inflation map $\iota:H^3(G,\bQ/\bZ)\rightarrow H^3(\bC(G),\bQ/\bZ)$. Following Saltman \cite[page 220, the third paragraph]{Sa5}, we define $H_{\rm n}^3(G,\bQ/\bZ)$ by
\[
H^3_{\rm n}(G,\bQ/\bZ)={\rm Ker}\{\iota:H^3(G,\bQ/\bZ)\rightarrow H^3(\bC(G),\bQ/\bZ)\}.
\]
The elements in $H_{\rm n}^3(G,\bQ/\bZ)$ are called {\it the geometrically negligible classes}.

On the other hand, the subgroup $H_{\rm nr}^3(G,\bQ/\bZ)$ is defined in Definition \ref{d2.11} below. It is proved that $H_{\rm nr}^3(G,\bQ/\bZ)$ $= \iota^{-1}(H^3_{\rm nr}(\bC(G),\bQ/\bZ))$ by Peyre \cite[page 204, Proposition 3]{Pe3}. Hence the inflation map $\iota$ induces an isomorphism $H_{\rm nr}^3(G,\bQ/\bZ)/H_{\rm n}^3(G,\bQ/\bZ) \to H_{\rm nr}^3(\bC(G),\bQ/\bZ)$.
\end{definition}

\begin{definition}\label{d2.11}
Let $H$ be any subgroup of $G$ and $Z_G(H)$ be the centralizer of $H$ in $G$.
For any $g\in Z_G(H)$, we define a map
\begin{align*}
\partial_{H,g}: H^3(G,\bQ/\bZ)\rightarrow H^2(H,\bQ/\bZ)
\end{align*}
as follows.
Define $I=\langle g\rangle$.
The natural morphism $m:H\times I\rightarrow G, (h,i)\mapsto hi$ induces a map
\begin{align*}
m^\ast: H^3(G,\bQ/\bZ)\rightarrow H^3(H\times I,\bQ/\bZ).
\end{align*}

Let ${\rm pr}_2:H\times I\rightarrow I$ be the projection map of $H\times I$ onto the second factor $I$, and $i_2:I\rightarrow H\times I, i\mapsto (e,i)$ (where $e$ denotes the identity element of $H$) which identifies $I$ as a subgroup of $H\times I$. Then $i_2$ defines the restriction map
\begin{align*}
i_2^\ast: H^3(H\times I,\bQ/\bZ)\xrightarrow{\rm res} H^3(I,\bQ/\bZ)
\end{align*}
and ${\rm pr}_2$ induces a section of $i_2^\ast$.
This yields a morphism
\begin{align*}
S_{H,I}: H^3(H\times I,\bQ/\bZ)\rightarrow H^3(H\times I,\bQ/\bZ)_1,\
\xi\mapsto \xi-{\rm pr}_2^\ast\circ i_2^\ast(\xi)
\end{align*}
where $H^3(H\times I,\bQ/\bZ)_1$ is defined by
\begin{align*}
H^3(H\times I,\bQ/\bZ)_1
={\rm Ker}\{H^3(H\times I,\bQ/\bZ)\xrightarrow{\rm res} H^3(I,\bQ/\bZ)\}.
\end{align*}
We apply the Lyndon-Hochschild-Serre spectral sequence \cite{HS}
\begin{align*}
E_2^{p,q}=H^p(H,H^q(I,\bQ/\bZ))\Rightarrow H^{p+q}(H\times I,\bQ/\bZ).
\end{align*}
By Theorem \ref{t1.17}, the spectral sequence degenerates at the $E_2$ term and we find an isomorphism
\begin{align*}
&H^3(H\times I,\bQ/\bZ)\\
&\simeq
H^3(H,H^0(I,\bQ/\bZ))\oplus H^2(H,H^1(I,\bQ/\bZ))\oplus
H^1(H,H^2(I,\bQ/\bZ))\oplus H^0(H,H^3(I,\bQ/\bZ))\\
&\simeq
H^3(H,\bQ/\bZ)\oplus H^2(H,H^1(I,\bQ/\bZ))\oplus
0\oplus H^3(I,\bQ/\bZ)
\end{align*}
because
$H^2(I,\bQ/\bZ)=0$.
Thus we get
\begin{align*}
S_{H,I} : H^3(H\times I,\bQ/\bZ)\rightarrow H^3(H\times I,\bQ/\bZ)_1,\
\xi=(\xi_0,\xi_1,0,\xi_3)\mapsto \xi^\prime=(\xi_0,\xi_1,0,0)
\end{align*}
and also a map
\begin{align*}
\varphi: H^3(H\times I,\bQ/\bZ)_1\rightarrow H^2(H,H^1(I,\bQ/\bZ)),\
\xi^\prime=(\xi_0,\xi_1,0,0)\mapsto \xi_1.
\end{align*}

An evaluation at $g$ defines an injection
\begin{align*}
H^1(I,\bQ/\bZ)\xrightarrow{\sim}{\rm Hom}(I,\bQ/\bZ)
\simeq\bZ/|I|\bZ\simeq \frac{1}{|I|}\bZ/\bZ\hookrightarrow \bQ/\bZ
\end{align*}
which yields $\partial_{H,g}=
\partial\circ m^\ast=\psi_g\circ\varphi\circ S_{H,I}\circ m^\ast$:
\begin{align*}
\xymatrix{
\partial_{H,g}: H^3(G,\bQ/\bZ)\ar[r]^{m^\ast}&H^3(H\times I,\bQ/\bZ)
\ar[d]^{S_{H,I}} \ar[r]^{\partial} \ar@{}[dr]|\circlearrowleft
& H^2(H,\bQ/\bZ)  \\
& H^3(H\times I,\bQ/\bZ)_1 \ar[r]^{\varphi} & \ar[u]_{\psi_g} H^2(H,H^1(I,\bQ/\bZ)).
}
\end{align*}

Following Peyre \cite[page 197]{Pe3}, we define
\begin{align*}
H^3_{\rm nr}(G,\bQ/\bZ)
=\bigcap_{H\leq G\atop g\in Z_G(H)}{\rm Ker}(\partial_{H,g}).
\end{align*}
\end{definition}

\bigskip
Now we turn to $H_{\rm p}^3(G,\bQ/\bZ)$, the group of permutation negligible classes.

Let $\mu$ be the group of all the roots of unity in $\bC$; thus $\mu \simeq \bQ/\bZ$ as $G$-modules and as $\Gamma_K$-modules where $K := \bC(G)$. We also recall the definition of a permutation $\bZ[G]$-lattice: A finitely generated $\bZ[G]$-module $P$ is called {\it a permutation $\bZ[G]$-lattice} if $P$ is a free abelian group of finite rank as an abelian group and has a $\bZ$-basis permuted by $G$, i.e. $P= \oplus_{1 \le i \le n} \bZ \cdot x_i$, such that for any $\sigma \in G$, $\sigma \cdot x_i = x_j$ for some $j$ depending on $\sigma$ and $i$.

We recall a proposition of Saltman.

\begin{proposition}[{Saltman \cite[page 221, Lemma 4.6 (b)]{Sa5}}]\label{p2.12}
Let $G$ be a finite group. Then there is a $\bZ[G]$-module $P^\ast$ containing $\mu$ as a submodule such that
{\rm (i)} the quotient module $P:=P^\ast/\mu$ is a permutation $\bZ[G]$-lattice,
{\rm (ii)} $P^\ast$ is $H^1$-trivial, i.e. for all subgroups $H\leq G$, $H^1(H,P^\ast)=0$.
{\rm (iii)} for any $\bZ[G]$-morphism $\varphi: \mu\rightarrow N$,
$\varphi$ extends to $P^\ast$ if and only if, for any subgroup $H\leq G$, the map $H^1(H,\mu)\rightarrow H^1(H,N)$ is the zero map.
\end{proposition}

\begin{definition}[{Saltman \cite[Proposition 4.7]{Sa5}}]\label{d2.13}
Let $0 \to \mu \to Q^\ast \to Q \to 0$ be an exact sequence of $\bZ[G]$-modules such that $Q$ is a permutation $\bZ[G]$-lattice and $H^1(H,Q^\ast)=0$ for all subgroups $H\leq G$ (for examples, take $Q^\ast$ to be the $\bZ[G]$-module $P^\ast$ constructed in Proposition \ref{p2.12}).
We define {\it the group of permutation negligible classes} by
\[
H_{\rm p}^3(G,\bQ/\bZ):={\rm Ker}\{H^3(G,\bQ/\bZ)\rightarrow H^3(G,Q^\ast)\}.
\]
The name comes from the fact that
\[
H_{\rm p}^3(G,\bQ/\bZ)\simeq {\rm Image}\{H^2(G,Q)\xrightarrow{\delta}
H^3(G,\bQ/\bZ)\}
\]
where $\delta:H^2(G,Q)\rightarrow H^3(G,\bQ/\bZ)$ is the connecting homomorphism arising from $0 \to \mu \to Q^\ast \to Q \to 0$.

It seems that the above definition of $H_{\rm p}^3(G,\bQ/\bZ)$ depends on the choice of the $\bZ[G]$-module $Q^\ast$. By \cite[pages 221--222; in particular, page 222, Proposition 4.7]{Sa5}, it is shown that (i) $H_{\rm p}^3(G,\bQ/\bZ)$ is independent of the choice of the $\bZ[G]$-module $Q^\ast$, and (ii)
$H_{\rm p}^3(G,\bQ/\bZ)={\rm Ker}\{H^3(G,\bQ/\bZ)\rightarrow H^3(G,L^\times)\}$ where $L:=k(x_g: g \in G)$ is the rational function field in Definition \ref{d1.1}. In particular,
$H_{\rm p}^3(G,\bQ/\bZ)$ is contained in $H_{\rm n}^3(G,\bQ/\bZ)$.
\end{definition}

By \cite[page 191, Theorem 4.14]{Sa5}, if $G$ is a non-abelian $2$-group containing a cyclic subgroup of index $2$, then $H_{\rm p}^3(G,\bQ/\bZ)$ is a proper subgroup of $H_{\rm n}^3(G,\bQ/\bZ)$. However, Peyre is able to prove the following results.

\begin{theorem}\label{t2.14}
Let $\iota: H^3(G,\bQ/\bZ)\rightarrow H^3(\bC(G),\bQ/\bZ)$ be
the inflation map.\\
{\rm (i)} {\rm ({Peyre \cite[Theorem 1]{Pe3}})}
The size of the kernel of the surjective map
\begin{align*}
\overline{\iota}: H_{\rm nr}^3(G,\bQ/\bZ)/H_{\rm p}^3(G,\bQ/\bZ)\rightarrow
H_{\rm nr}^3(\bC(G),\bQ/\bZ)
\end{align*}
is a power of $2$, equivalently, $[H_{\rm n}^3(G,\bQ/\bZ):H_{\rm p}^3(G,\bQ/\bZ)]=2^d$ for some non-negative integer $d$.
In particular, if $|G|$ is a group of odd order, then $H_{\rm n}^3(G,\bQ/\bZ)=H_{\rm p}^3(G,\bQ/\bZ)$.
~\\
{\rm (ii)} {\rm ({Peyre \cite[pages 196--197]{Pe2}, see also \cite[Remark 2]{Pe3})}}
\begin{align*}
H_{\rm p}^3(G,\bQ/\bZ)=
\sum_{H\leq G}{\rm Cores}_H^G({\rm Image}\{H^1(H,\bQ/\bZ)^{\otimes 2}
\xrightarrow{\cup} H^3(H,\bQ/\bZ)\})
\end{align*}
where the cup-product on the right is given by
the following commutative diagram:
\begin{align*}
\xymatrix{
H^i(G,\bQ/\bZ)\times H^j(G,\bQ/\bZ)
\ar[d]^{\cup} \ar[r]^{\sim} \ar@{}[dr]|\circlearrowleft
& H^{i+1}(G,\bZ)\times H^{j+1}(G,\bZ) \ar[d]^{\cup} \\
H^{i+j+1}(G,\bQ/\bZ) \ar[r]^{\sim} & H^{i+j+2}(G,\bZ) \\
}
\end{align*}
for any $i\geq 1$ and $j\geq 1$.
\end{theorem}

Finally we define the stable cohomology group.

\begin{definition}[{Bogomolov \cite[Definition 6.4, Lemma 6.5, Theorem 6.8]{Bo3}}]\label{d2.15}

The quotient group
\[
H^i_{\rm s}(G,\bQ/\bZ)=H^i(G,\bQ/\bZ)/H_{\rm n}^i(G,\bQ/\bZ)
\]
is called {\it the stable cohomology of $G$} of degree $i$
(see also \cite[page 6]{Bo2}, \cite[page 938]{BP}, \cite[page 57]{BB1},
\cite[page 212]{BB2}).
\end{definition}

When $i=3$, since $H^3_{\rm nr}(\bC(G),\bQ/\bZ)$ is a subgroup of the image of the
inflation map $\iota:H^3(G,\bQ/\bZ)\rightarrow H^3(\bC(G),\bQ/\bZ)$ by \cite[page 230, Theorem 5.3]{Sa5}, it follows that
\[
H^3_{\rm s}(G,\bQ/\bZ)\simeq
\iota(H^3(G,\bQ/\bZ))\geq H^3_{\rm nr}(\bC(G),\bQ/\bZ).
\]

Note that, in the paper of Bogomolov, Petrov and Tschinkel \cite[page 68]{BPT}
(see also \cite[Definition 8.5]{Bo3}, \cite[page 938]{BP} or \cite[page 214]{BB2}), the group $\iota(H^i(G,\bQ/\bZ))\cap H^i_{\rm nr}(\bC(G),\bQ/\bZ)$
is designated as $H^i_{\rm nr}(G,\bQ/\bZ)$, which is different from what we define in Definition \ref{d2.11}.
In our paper, we will not adopt the usage of $H^3_{\rm nr}(G,\bQ/\bZ)$ in \cite{BPT}.

\bigskip
\section{Proof of Theorem \ref{t1.14}\ and Theorem \ref{t1.15}}\label{sePT}

We will evaluate the unramified cohomology group
$H^3_{\rm nr}(\bC(G),\bQ/\bZ)$ for groups $G$ of order $3^5$, $5^5$ and $7^5$.

Let $p$ be odd prime number and $G$ be a group of order $p^5$. We explain briefly our strategy to trap the group $H^3_{\rm nr}(\bC(G),\bQ/\bZ)$. By Peyre's theorem (see Definition \ref{d2.10} and Theorem \ref{t2.14}),
we have
\[
H^3_{\rm nr}(\bC(G),\bQ/\bZ)\simeq
H_{\rm nr}^3(G,\bQ/\bZ)/H_{\rm p}^3(G,\bQ/\bZ).
\]

We use the quotient group $H^3(G,\bQ/\bZ)/H_{\rm p}^3(G,\bQ/\bZ)$ to approximate the group $H_{\rm nr}^3(G,\bQ/\bZ)/H_{\rm p}^3(G,\bQ/\bZ)$. In general, it is difficult to find the group $H_{\rm nr}^3(G,\bQ/\bZ)/H_{\rm p}^3(G,\bQ/\bZ)$ directly. However, when $p=3, 5, 7$, the group $H^3(G,\bQ/\bZ)/H_{\rm p}^3(G,\bQ/\bZ)$ is ``small" enough that it is possible to decide, by a case by case checking, which cohomology classes of the quotient $H^3(G,\bQ/\bZ)/H_{\rm p}^3(G,\bQ/\bZ)$ belong to $H_{\rm nr}^3(G,\bQ/\bZ)/H_{\rm p}^3(G,\bQ/\bZ)$ by applying Peyre's formula in Definition \ref{d2.11}. Thus we find the quotient group $H_{\rm nr}^3(G,\bQ/\bZ)/H_{\rm p}^3(G,\bQ/\bZ)$ by this indirect method.

\bigskip
Here is the proof of Theorem \ref{t1.14} and Theorem \ref{t1.15}. It is rather long and consists of five steps.

Step 1.
We will evaluate the cohomology group $H^3(G,\bQ/\bZ)$, its subgroup $H^3_{\rm p}(G,\bQ/\bZ)$, and the quotient group $H^3(G,\bQ/\bZ)/H^3_{\rm p}(G,\bQ/\bZ)$.

Note that $H^3_{\rm n}(G,\bQ/\bZ)=H^3_{\rm p}(G,\bQ/\bZ)$ because $|G|=p^5$ and $p$ is odd prime number. We will apply Theorem \ref{t2.14} (ii) to find $H^3_{\rm p}(G,\bQ/\bZ)$, and apply Theorem \ref{t2.14} (i) to find the stable cohomology group $H^3_{\rm s}(G,\bQ/\bZ)$.

In this step we will work on $H^4(G,\bZ)$, instead of $H^3(G,\bQ/\bZ)$ for the following reason. In the commutative diagram about cup products (for the situation $i=j=1$) in Theorem \ref{t2.14}, the computer package HAP for GAP takes care of cohomology groups with coefficient in $\bZ$, i.e. it works for $H^4(G,\bZ)$, but not $H^3(G,\bQ/\bZ)$) although they are isomorphic as abstract groups. Thus we will determine $H^4(G,\bZ)$ and $H^4_{\rm p}(G,\bZ)$, which is isomorphic to $H^3_{\rm p}(G,\bQ/\bZ)$ (see the definition of
$H^4_{\rm p}(G,\bZ)$ described below).

\medskip
To compute $H^4(G,\bZ)$, we first take a free resolution of the trivial $\bZ[G]$-module $\bZ$.

We use the HAP function ${\tt ResolutionNormalSeries(LowerCentralSeries(}G{\tt ,5))}$ to obtain a free resolution $RG$:
\begin{align*}
RG:\quad
\cdots\rightarrow P_5\rightarrow P_4\rightarrow P_3\rightarrow P_2
\rightarrow P_1\rightarrow P_0\rightarrow\bZ\rightarrow 0.
\end{align*}
We remark that the above resolution is not the usual bar resolution, and the $\bZ[G]$-rank of $P_i$ ($i \ge 1$)
is equal to
the binomial coefficient $\binom{i+4}{i}$.
For examples, the $\bZ[G]$-rank of
$P_0=\bZ[G]$, $P_1$, $P_2$, $P_3$, $P_4$, $P_5$
are $1$, $5$, $15$, $35$, $70$, $126$ respectively. It is obvious that such a polynomial growth resolution is more efficient for computer computing than the bar resolution. See \cite{EHS} for details.

Form the cochain complex ${\rm Hom}_{\bZ[G]}(P_{\bullet}, \bZ)$ and find the cohomology group $H^i(G,\bZ)=Z^i(G,\bZ)/B^i(G,\bZ)$ where $Z^i(G,\bZ)$ and $B^i(G,\bZ)$ are the groups of the $i$-th cocycles and the $i$-th coboundaries respectively.

By the HAP function {\tt CR\_CocyclesAndCoboundaries(}$RG${\tt ,4,true)}
we obtain bases for $Z^4(G,\bZ)$ and $B^4(G,\bZ)$.

By the command {\tt CR\_CocyclesAndCoboundaries(}$RG${\tt ,4,true).torsionCoefficients} we find the cohomology group $H^4(G,\bZ)$ (see Cases 1a--7d after the proof).

\bigskip
Now we turn to the group $H_{\rm p}^3(G,\bQ/\bZ)$. Use Peyre's notation
\begin{align*}
H_{\rm p}^3(G,\bQ/\bZ)=
\sum_{H\leq G}{\rm Cores}_H^G({\rm Image}\{H^1(H,\bQ/\bZ)^{\otimes 2}
\xrightarrow{\cup} H^3(H,\bQ/\bZ)\})
\end{align*}
in Theorem \ref{t2.14}.
Let $H_{\rm p}^4(G,\bZ)$ be the image of $H_{\rm p}^3(G,\bQ/\bZ)$ in the isomorphism $H^3(G,\bQ/\bZ) \to H^4(G,\bZ)$.
{}From the commuting diagram given in Theorem \ref{t2.14} it is clear that $H_{\rm p}^4(G,\bZ)$ may be defined as follows :
\begin{align*}
H_{\rm p}^4(G,\bZ)=
\sum_{H\leq G}{\rm Cores}_H^G({\rm Image}\{H^2(H,\bZ)^{\otimes 2}
\xrightarrow{\cup} H^4(H,\bZ))\})
\end{align*}
where $H$ runs over all the subgroups of $G$.

\bigskip
Step 2. In the above definition of $H_{\rm p}^4(G,\bZ)$, we may drop many subgroups $H\leq G$ without affecting the group $H_{\rm p}^4(G,\bZ)$.

In fact, we may take only non-perfect subgroups $H$,
i.e. $D(H)\lneq H$. The reason is that, if $D(H)=H$,
then $H^2(H,\bZ)\simeq H^1(H,\bQ/\bZ)={\rm Hom}(H,\bQ/\bZ)\simeq H/D(H)=1$.

Also, it suffices to take subgroups $H$ up to conjugation. The reason is that: if $H^\sigma:= \sigma^{-1} H \sigma$ is the conjugate of $H$ by $\sigma$, then $H$ and $H^\sigma$ have naturally isomorphic cohomology groups and are sent to the same subgroup of $H^4(G,\bZ)$ by the corestriction maps (because the conjugation map on $G$ induces the identity map on $H^q(G, \bQ/\bZ)$ by \cite[page p. 116, Proposition 3]{Se}).

Moreover, if $D(H^\prime)=D(H)$ and $H^\prime < H$ (remember that $D(H):=[H,H]$ by definition), then we do not need to count on $H^\prime$. The reason is that: for any $\chi^\prime\in H^1(H^\prime,\bQ/\bZ)$ there exists
$\chi\in H^1(H,\bQ/\bZ)={\rm Hom}(H,\bQ/\bZ)$ such that
$\chi^\prime={\rm Res}^H_{H^\prime}\chi$
(because $H$ and $H^\prime$ have the same commutator subgroup !). Thus ${\rm Cores}^G_{H^\prime}(\chi_1^\prime\cup\chi_2^\prime)
={\rm Cores}^G_H{\rm Cores}^H_{H^\prime}(\chi_1^\prime\cup\chi_2^\prime)=
{\rm Cores}^G_H{\rm Cores}^H_{H^\prime}({\rm Res}^H_{H^\prime}\chi_1
\cup{\rm Res}^H_{H^\prime}\chi_2)
={\rm Cores}^G_H{\rm Cores}^H_{H^\prime}({\rm Res}^H_{H^\prime}
\chi_1\cup\chi_2)
=[H:H^\prime]\,{\rm Cores}^G_H(\chi_1\cup\chi_2)$.

In conclusion, it suffices to evaluate
\begin{align*}
H_{\rm p}^4(G,\bZ)=
\sum_{D(H)\lneq H\leq G:\, {\rm up\, to\, conjugation}
\atop {H^\prime < H:\,{\rm maximal}\, {\rm with}
\atop D(H^\prime)=D(H)}}
{\rm Cores}_H^G({\rm Image}\{H^2(H,\bZ)^{\otimes 2}
\xrightarrow{\cup} H^4(H,\bZ)\}).
\end{align*}
The above group may be obtained by applying the following GAP function {\tt H4pFromResolution(}$RG${\tt )}
(see Cases 1a--7d after the proof):\\

\noindent
{\tt H4pFromResolution(}$RG${\tt )}
prints the number of conjugacy subgroups $H\leq G$ with $D(H)\lneq H$
which is the maximal one having the same commutator subgroup $D(H)$,
their SmallGroup IDs of GAP and the computing progress rate,
and returns the list $l=[l_1,[l_2,l_3]]$ for a free resolution
$RG$ of $G$
where $l_1$ is the abelian invariant of $H^4_{\rm p}(G,\bZ)$
with respect to Smith normal form,
$l_2$ is the abelian invariants of $H^4(G,\bZ)$
with respect to Smith normal form
and $l_3$ is generators of $H^4_{\rm p}(G,\bZ)$ in $H^4(G,\bZ)$
for a free resolution $RG$ of $G$.\\

\bigskip
Step 3.
Suppose that $\Gamma:=H \times I$ with $I=\langle g \rangle$.
We digress to explain a method of finding the cohomology groups $H^n(\Gamma, \bQ/\bZ)$ (where $n \ge 0$) as presented in the HAP command {\tt ResolutionDirectProduct(}$RH,RI${\tt )}. Note that, in Step 4, $H$ is a subgroup of $G$, $g\in Z_G(H)$, and $\Gamma=H \times I$ where $I=\langle g \rangle$.

\bigskip
First take a free resolution of $\bZ$ over $\bZ[H]$ as in Step 1:
\begin{align*}
RH:\quad \cdots\xrightarrow{d_4} P_3\xrightarrow{d_3} P_2\xrightarrow{d_2} P_1
\xrightarrow {d_1} P_0=\bZ[H]\xrightarrow{\varepsilon}\bZ\rightarrow 0.
\end{align*}

Then take a free resolution of $\bZ$ over $\bZ[I]$ (the ``norm resolution" of the cyclic group $I$):
\begin{align*}
RI:\quad \cdots
\xrightarrow{\sum_k g^k} Q_3 = \bZ[I]
\xrightarrow{g-1} Q_2 = \bZ[I]
\xrightarrow{\sum_k g^k} Q_1 = \bZ[I]
\xrightarrow{g-1} Q_0=\bZ[I]
\xrightarrow{\varepsilon}\bZ\rightarrow 0.
\end{align*}

Because $\bZ[\Gamma] \simeq \bZ[H] \otimes_{\bZ} \bZ[I]$, the group $P_i \otimes_{\bZ} Q_j$ may be regarded as a $\bZ[\Gamma]$-module as follows: Let $p_{i,1},\dots,p_{i,n_i}$ be a $\bZ[H]$-basis of $P_i$
and
$q_j$ be a $\bZ[I]$-basis of $Q_j$.
Define
$r_{i,j,k}=p_{i,k}\otimes q_j$, which becomes a free $\bZ[\Gamma]$-module
$R_{i,j}=\bZ[\Gamma]r_{i,j,1} \oplus \cdots \oplus \bZ[\Gamma]r_{i,j,n_i}$.
Thus we obtain a third quadrant bicomplex
\begin{align*}
\begin{CD}
\vdots @. \vdots @. \vdots @. \vdots\\\
@VV\sum_k g^k V @VV-\sum_k g^k V @VV\sum_k g^k V @VV-\sum_k g^k V \\
R_{0,3} @<d_1<< R_{1,3} @<d_2<< R_{2,3} @<d_3<< R_{3,3} @<d_4<< \cdots\\
@VV g-1 V @VV -(g-1) V @VV g-1 V @VV -(g-1) V \\
R_{0,2} @<d_1<< R_{1,2} @<d_2<< R_{2,2} @<d_3<< R_{3,2} @<d_4<< \cdots\\
@VV\sum_k g^k V @VV-\sum_k g^k V @VV\sum_k g^k V @VV-\sum_k g^k V \\
R_{0,1} @<d_1<< R_{1,1} @<d_2<< R_{2,1} @<d_3<< R_{3,1} @<d_4<< \cdots\\
@VV g-1 V @VV -(g-1) V @VV g-1 V @VV -(g-1) V \\
R_{0,0} @<d_1<< R_{1,0} @<d_2<< R_{2,0} @<d_3<< R_{3,0} @<d_4<< \cdots.\\\\
\end{CD}
\end{align*}
Note that $R_{i,0}\simeq R_{i,1}\simeq \cdots\simeq R_{i,j}\simeq \cdots$
as $\bZ[\Gamma]$-modules and ${\rm rank}_{\bZ[\Gamma]} R_{i,j}=n_i$.

Define $R_{i,j}^\ast={\rm Hom}_{\bZ[\Gamma]}(R_{i,j},\bQ/\bZ)$. We obtain a first quadrant bicomplex
\begin{align*}
\begin{CD}
\vdots @. \vdots @. \vdots @. \vdots\\\
@AA|I|A @AA-|I|A @AA|I|A @AA-|I|A \\
R_{0,3}^\ast @>d_1^\ast>> R_{1,3}^\ast @>d_2^\ast>> R_{2,3}^\ast @>d_3^\ast>> R_{3,3}^\ast @>d_4^\ast>>\cdots \\
@AA0A @AA0A @AA0A @AA0A \\
R_{0,2}^\ast @>d_1^\ast>> R_{1,2}^\ast @>d_2^\ast>> R_{2,2}^\ast @>d_3^\ast>> R_{3,2}^\ast @>d_4^\ast>>\cdots \\
@AA|I|A @AA-|I|A @AA|I|A @AA-|I|A \\
R_{0,1}^\ast @>d_1^\ast>>R_{1,1}^\ast @>d_2^\ast>> R_{2,1}^\ast @>d_3^\ast>> R_{3,1}^\ast @>d_4^\ast>>\cdots \\
@AA0A @AA0A @AA0A @AA0A \\
R_{0,0}^\ast @>d_1^\ast>> R_{1,0}^\ast @>d_2^\ast>> R_{2,0}^\ast @>d_3^\ast>> R_{3,0}^\ast @>d_4^\ast>>\cdots. \\\\
\end{CD}
\end{align*}

Note that the vertical differentials $R_{p,2h}^\ast \to R_{p,2h+1}^\ast$ (where $p, h \ge 0$) are the zero maps and the vertical differentials $R_{p,2h+1}^\ast \to R_{p,2h+2}^\ast$ (where $p, h \ge 0$) are the multiplication by $|I|$ or $-|I|$ because $g$ acts trivially on $H$ and $\bQ/\bZ$. Be aware that we have adjusted the signs of the vertical differentials of this bicomplex.

\medskip
We also note that $R_{p,q}^\ast={\rm Hom}_{\bZ[\Gamma]}(R_{p,q},\bQ/\bZ) \simeq {\rm Hom}_{\bZ[H]}(P_p, {\rm Hom}_{\bZ[I]}(Q_q,\bQ/\bZ))$.

Consider the spectral sequence associated to the first quadrant bicomplex $\{{\rm Hom}_{\bZ[H]}(P_p, {\rm Hom}_{\bZ[I]}(Q_q,\bQ/\bZ)): p,q \ge 0 \}$ \cite[page 341]{Mac}. Its $E_2$ term for the (first) filtration is given by $E_2^{p,q}= H^p(H^q(\bE_0, d^{\prime\prime}),d^{\prime})$ where $d^{\prime}$ and $d^{\prime\prime}$ denote the horizontal differentials and the vertical differentials respectively \cite[page 73]{Ev}. The group $H^p(H^q(\bE_0, d^{\prime\prime}),d^{\prime})$ is nothing but $H^p(H, H^q(I, \bQ/\bZ))$ in the present situation. By the Lyndon-Hochschild-Serre Theorem, the spectral sequence with $E_2^{pq}= H^p(H, H^q(I, \bQ/\bZ))$ converges to the groups $H^{p+q}(H \times I, \bQ/\bZ)$ by \cite[pages 72--73]{Ev}. Since $\Gamma$ is a direct product of $H$ and $I$, the associated spectral sequence of $\{R_{p,q}^\ast: p,q \ge 0 \}$ degenerates at the $E_2$ level and therefore the cohomology group $H^n$ is isomorphic to $\oplus_{p+q=n} E_2^{pq}$ by Theorem \ref{t1.17}. Thus the two target groups of this spectral sequence are isomorphic to $\oplus_{p+q=n} E_2^{pq}$; hence they are isomorphic.

In conclusion, $H^n(H \times I, \bQ/\bZ)$ is isomorphic to the $n$-th homology group of the cochain complex \[
0 \to R_0^\ast \to R_1^\ast \to R_2^\ast \to \cdots \to R_n^\ast \to \cdots
\]
where
$R_n^\ast:=R_{n,0}^\ast \oplus R_{n-1,1}^\ast \oplus \cdots \oplus R_{0,n}^\ast$.

We remark that an alternative proof of the above assertion may be found by using the method indicated in the paper of C. T. C. Wall \cite{Wa}, which was adopted by \cite{HAP}.

\bigskip
Step 4.
The goal is to find the quotient group $H^3_{\rm nr}(G,\bQ/\bZ)/H^3_{\rm p}(G,\bQ/\bZ)$. But it is difficult to find the group $H^3_{\rm nr}(G,\bQ/\bZ)$ as a whole by applying Definition \ref{d2.11} directly. We take a detour by computing the group $H^3(G,\bQ/\bZ)/H^3_{\rm p}(G,\bQ/\bZ)$ first, and then for each cohomology class $\gamma \in H^3(G,\bQ/\bZ)$ we apply Peyre's formula (see Definition \ref{d2.11}) to check that whether $\gamma \in H^3_{\rm nr}(G,\bQ/\bZ)$. Recall that Peyre's formula is defined as:
\begin{align*}
H^3_{\rm nr}(G,\bQ/\bZ)
=\bigcap_{H\leq G\atop g\in Z_G(H)}{\rm Ker}(\partial_{H,g}).
\end{align*}

To check whether $\gamma \in H^3_{\rm nr}(G,\bQ/\bZ)$, it suffices to check in GAP whether $\partial_{H,g}(\gamma)=0$ for any subgroup $H$ of the group $G$ and any $g\in Z_G(H)$. However, as in Step 2, it is unnecessary to check all the subgroups $H$ (see Step 5); we just need to check those ``significant" ones.

In this step, we will indicate how to carry out the map $\partial_{H,g}(\gamma)$ where $\gamma \in H^3(G,\bQ/\bZ)$, $H$ is a subgroup of $G$ and $g\in Z_G(H)$.

\bigskip
Define $\Gamma=H\times I$ where $I= \langle g \rangle$. In the sequel, if $\gamma \in H^3(G,\bQ/\bZ)$, $w$ is a $3$-cocycle in $Z^3(G,\bQ/\bZ)$ (the group of $3$-cocycles) and $\gamma$ is the cohomology class associated to $w$, we will write $\gamma=[w]$.

{}From Definition \ref{d2.11}, recall the definition that $\partial_{H,g}=
\partial\circ m^\ast=\psi_g\circ\varphi\circ S_{H,I}\circ m^\ast$:
\begin{align*}
\xymatrix{
\partial_{H,g}: H^3(G,\bQ/\bZ)\ar[r]^{m^\ast}&H^3(H\times I,\bQ/\bZ)
\ar[d]^{S_{H,I}} \ar[r]^{\partial} \ar@{}[dr]|\circlearrowleft
& H^2(H,\bQ/\bZ)  \\
& H^3(H\times I,\bQ/\bZ)_1 \ar[r]^{\varphi} & \ar[u]_{\psi_g} H^2(H,H^1(I,\bQ/\bZ)).
}
\end{align*}

Write $\xi=m^\ast(\gamma)$ and $[u]:=S_{H,I}(\xi)=\xi-{\rm pr}_2^\ast \circ i_2^\ast (\xi)\in H^3(H \times I, \bQ/\bZ)_1$ where $u \in Z^3(H \times I, \bQ/\bZ)$ (the group of $3$-cocycles in $R_3^\ast=R_{3,0}^\ast \oplus R_{2,1}^\ast
\oplus R_{1,2}^\ast \oplus R_{0,3}^\ast$). Remember that the cochain complex $R^\ast$ defined at the end of Step 3:
\[
0 \to R_0^\ast \to R_1^\ast \to R_2^\ast \to R_3^\ast \to \cdots
\]
where
$R_i^\ast:=R_{i,0}^\ast \oplus R_{i-1,1}^\ast \oplus \cdots \oplus R_{0,i}^\ast$, $d^{\prime}, d^{\prime \prime}$ are the horizontal and the vertical differentials, and the differentials of $R^\ast$ are the total differentials $d^\ast := d^{\prime}+d^{\prime \prime}$.

Consider the first filtration $\{F^pR^\ast : p \ge 0 \}$ of $R^\ast$ defined by $F^pR_n^\ast= \oplus_{p' \ge p}R_{p',n-p'}^\ast$ \cite[page 70]{Ev}. It follows that $F^p H^{p+q}(R^\ast)/F^{p+1} H^{p+q}(R^\ast) \simeq E_{\infty}^{p,q}$ \cite[page 71]{Ev}. Note that $E_{\infty}^{p,q}=E_2^{p,q}\simeq H^p(H, H^q(I, \bQ/\bZ))$ by Theorem \ref{t1.17}.

By Definition \ref{d2.11}), $[u] \in H^3(H\times I,\bQ/\bZ)_1
={\rm Ker}\{H^3(H\times I,\bQ/\bZ)\xrightarrow{\rm res} H^3(I,\bQ/\bZ)\}$, and the components in $H^1(H,H^2(I,\bQ/\bZ))$ and $H^0(H,H^3(I,\bQ/\bZ))$ of $[u]$ are zero. Thus the image of $[u]$ in $F^0 H^3(R^\ast)/F^1 H^3(R^\ast) \simeq H^0(H, H^3(I, \bQ/\bZ)$ is zero. It follows that $[u] \in F^1 H^3(R^\ast)$. Since $F^1 H^3(R^\ast)/F^2 H^3(R^\ast)=0$ (remember that $H^2(I, \bQ/\bZ)=0$), thus $[u] \in F^2 H^3(R^\ast)$. Because $F^2 H^3(R^\ast)$ is the image of $H^3(F^2R^\ast) \to H^3(R^\ast)$ by definition, we may find a $3$-cocycle $v=v_0+v_1$ (where $v_0 \in R_{3,0}^\ast$, $v_1 \in R_{2,1}^\ast$) such that $[u]=[v]$.

We conclude that $[v]=S_{H,I}\circ m^\ast(\gamma)$ where $v=v_0+v_1$ with $v_0 \in R_{3,0}^\ast$, $v_1 \in R_{2,1}^\ast$.

\bigskip
Because the vertical differential $d^{\prime \prime}$ is the zero map on $R_{3,0}^\ast$, it follows that $d^{\prime \prime}(v_0)=0$. From  $d^\ast(v_0+v_1)=0$, it is easy to verify that $d^{\prime}(v_0)=0$, $d^{\prime}(v_1)=0$ and $d^{\prime \prime}(v_1)=0$. Thus  $d^\ast(v_0)=0$ and $d^\ast(v_1)=0$, i.e. both $v_0$ and $v_1$ are $3$-cocycles in the cochain complex $R^\ast$.

In other words, if $\Phi$ is the isomorphism of $H^3(H\times I,\bQ/\bZ)$ onto $H^3(H,H^0(I,\bQ/\bZ))\oplus H^2(H,H^1(I,\bQ/\bZ))\oplus
H^1(H,H^2(I,\bQ/\bZ))\oplus H^0(H,H^3(I,\bQ/\bZ))$ provided by Theorem \ref{t1.17}, then $\Phi$ sends $[v]$ to $[v_0]+[v_1]$. We conclude that $\phi([v])=[v_1]$. Thus $\partial_{H,g}(\gamma)= \psi_g([v_1])$.

\medskip
We will transform the formula $\phi([v])=[v_1]$ into a more convenient form, which helps us understand the element $\psi_g([v_1])$ better.

Recall that $R_{i,j}^\ast={\rm Hom}_{\bZ[H\times I]}(R_{i,j},\bQ/\bZ)$. We define $P_i^\ast={\rm Hom}_{\bZ[H]}(P_i,\bQ/\bZ)$. Since $R_{i,j}$ is a free $\bZ[H\times I]$ with generators $p_{i,k} \otimes q_j$ and $P_i$ is a free $\bZ[H]$ with generators $p_{i,k}$ (where  and $1 \le k \le n_i$), morphisms in ${\rm Hom}_{\bZ[H\times I]}(R_{i,j},\bQ/\bZ)$ and in ${\rm Hom}_{\bZ[H]}(P_i,\bQ/\bZ)$ are uniquely determined by their images of these generators. Because there is a bijection of these two sets of generators, we find an isomorphism $f_{i,j}^\ast: R_{i,j}^\ast\rightarrow P_i^\ast$ as abelian groups.

Consider the following commutative diagram
\begin{align*}
\begin{CD}
R_{0,1}^\ast @>d_1^{\prime}>> R_{1,1}^\ast @>d_2^{\prime}>> R_{2,1}^\ast @>d_3^{\prime}>> R_{3,1}^{\prime} @>d_4^{\prime}>> \cdots \\
@Vf_{0,1}V\simeq V  @Vf_{1,1}V\simeq V @Vf_{2,1}^\ast V\simeq V @Vf_{3,1}V\simeq V \\
P_0^\ast @>d_1^\ast>> P_1^\ast @>d_2^\ast>> P_2^\ast @>d_3^\ast>> P_3^\ast @>d_4^\ast>> \cdots\\
\end{CD}
\end{align*}
where the differentials in the top row are the horizontal differentials of the bicomplex $\{R_{p,q}^\ast \}$ and the differentials in the bottom row are those induced by the differentials for the chain complex $\cdots\to P_3 \to P_2 \to P_1 \to P_0 \to 0$.

Let $v, v_0, v_1$ be as before. We claim that $f_{2,1}^\ast(v_1)\in Z^2(H,\frac{1}{|I|}\bZ/\bZ)$. It is routine to check that
$d_3^\ast(f_{2,1}^\ast(v_1))=f_{3,1}(d_3^{\prime}(v_1))=f_{3,1}(0)=0$ because $v_1\in Z^3(\Gamma,\bQ/\bZ)$ as we note before.
Hence $f_{2,1}^\ast(v_1)\in Z^2(H,\bQ/\bZ)$.

On the other hand, since $v_1$ is a cocycle with respect to the differential $d^\ast=d^{\prime}+d^{\prime \prime}$, it follows that $d^{\prime \prime}(v_1)=0$. Thus $0=d^{\prime \prime}(v_1)=-|I|v_1$. It follows that $|I|(f_{2,1}^\ast(v_1))=f_{2,1}^\ast(|I|v_1)=f_{2,1}^\ast(0)=0$.
Thus $f_{2,1}^\ast(v_1)\in Z^2(H,\frac{1}{|I|}\bZ/\bZ)$. Note that $Z^2(H,\frac{1}{|I|}\bZ/\bZ)\simeq Z^2(H,H^1(I,\bQ/\bZ))$.

In conclusion, we find that $[f_{2,1}^\ast(v_1)]=\varphi(S_{H,I}([v]))\in
H^2(H,H^1(I,\bQ/\bZ))\simeq H^2(H,\frac{1}{|I|}\bZ/\bZ))$.
Hence we get $\partial_{H,g}(\gamma)=\partial\circ m^\ast(\gamma)=\psi_g([f_{2,1}^\ast(v_1)])$, a formula which  will be used later.

\bigskip
Step 5.
We explain how to implement the computation of $H_{\rm nr}^3(\bC(G),\bQ/\bZ) \simeq H_{\rm nr}^3(G,\bQ/\bZ)/H_{\rm p}^3(G,\bQ/\bZ)$ when $G$ is a group of odd order.

Use the GAP computation to compute $H^4(G,\bZ)$ ($\simeq H^3(G,\bQ/\bZ)$) and $H_{\rm p}^4(G,\bZ)$, which is defined in Step 1 and Step 2 of this Section.

If $H_{\rm p}^4(G,\bZ)= H^4(G,\bZ)$, then $H_{\rm n}^3(G, \bQ/\bZ)= H^3_{\rm nr}(G,\bQ/\bZ)$ and therefore $H^3_{\rm nr}(\bC(G),\bQ/\bZ)=0$; the proof is finished. From now on we assume that $H_{\rm p}^4(G,\bZ)\lneq H^4(G,\bZ)$

For any $\overline{f}\in H^4(G,\bZ)/H^4_{\rm p}(G,\bZ)$, we will determine
whether $f\in H^4_{\rm nr}(G,\bZ)$ or not.

Take the free resolution $RG$ in Step 1. The formula in Definition \ref{d2.11} is reformulated as
\begin{align*}
H^4_{\rm nr}(G,\bZ)
:=\bigcap_{H\leq G\atop g\in Z_G(H)}{\rm Ker}(\widetilde{\partial}_{H,g})
\end{align*}
where
\begin{align*}
\xymatrix{
\widetilde{\partial}_{H,g}:H^4(G,\bZ)\ar[r]^{\widetilde{m^\ast}}
\ar[d]^{\simeq}\ar@{}[dr]|\circlearrowleft
&H^4(H\times I,\bZ) \ar[d]^{\simeq} \ar[r]^{\widetilde{\partial}}
\ar@{}[dr]|\circlearrowleft
& H^3(H,\bZ) \ar[d]^{\simeq} \\
\partial_{H,g}:H^3(G,\bQ/\bZ)\ar[r]^{m^\ast}&H^3(H\times I,\bQ/\bZ)
\ar[r]^{\partial}
& H^2(H,\bQ/\bZ).
}
\end{align*}

In applying the formula
\begin{align*}
H^4_{\rm nr}(G,\bZ)
=\bigcap_{H\leq G\atop g\in Z_G(H)}{\rm Ker}(\widetilde{\partial}_{H,g}),
\end{align*}
we may ignore many pairs $(H,I)$. In fact, it suffices to consider only
those pairs $(H,I)$ of $H\leq G$ and $I=\langle g\rangle\leq Z_G(H)$
satisfying the following properties:\\
{\rm (i)} $I=\langle g\rangle$ for some $g$;\\
{\rm (ii)} $(H,I)$ is chosen up to conjugation;\\
{\rm (iii)} $(H^\prime,I^\prime)\leq (H,I)$ is maximal, and thus we may assume that $H=Z_G(I)$ (where $I=\langle g\rangle$) and $g$ belongs to the center of $H$;\\
{\rm (iv) $H^3(H,\bZ)\neq 0$.\\
The reasons are:
{\rm (i)} If $I=\langle g\rangle=\langle g^\prime\rangle$,
then ${\rm Ker}(\widetilde{\partial}_{H,g})={\rm Ker}(\widetilde{\partial}_{H,g^\prime})$;

{\rm (ii)} 
We have ${\rm Ker}(\widetilde{\partial}_{H,g})={\rm Ker}(\widetilde{\partial}_{H^\sigma,g^\sigma})$;

{\rm (iii)} If $H^\prime\leq H$ and $I^\prime=\langle g^\prime\rangle
\leq I=\langle g\rangle$,
then ${\rm Ker}(\widetilde{\partial}_{H,g})\leq {\rm Ker}(\widetilde{\partial}_{H^\prime,g^\prime})$. Moreover, if $\langle g\rangle\leq Z_G(H)$, then $H \leq Z_G(\langle g \rangle)$; thus $(H, \langle g \rangle) \leq (Z_G(\langle g \rangle), \langle g \rangle)$. By the maximality, we may take $(Z_G(\langle g \rangle), \langle g \rangle)$ and ignore $(H, \langle g \rangle)$. Clearly $g$ belongs to the center of $Z_G(\langle g \rangle)$;

{\rm (iv)} If $H^3(H,\bZ)=0$, from the top row of the above commutative diagram, we find that ${\rm Ker}(\widetilde{\partial}_{H,g})=H^4(G,\bZ)$. Obviously $H^4(G,\bZ)$ contains $H^4_{\rm nr}(G,\bZ)$; thus we do not need to count ${\rm Ker}(\widetilde{\partial}_{H,g})$ which arises from the subgroup $H$ with $H^3(H,\bZ)=0$.

\bigskip
With the aid of (i)--(iv), we write the following command
{\tt IsUnramifiedH3(}$RG$, $l${\tt )} of GAP (HAP)
and apply it to determine the group
$H^4_{\rm nr}(G,\bZ)\simeq H^3_{\rm nr}(G,\bQ/\bZ)$:\\

\noindent
{\tt IsUnramifiedH3(}$RG$, $l${\tt )}
prints the number of pairs $(H,I)$ of $H\leq G$ and $I\leq Z(H)$
which satisfy (i)--(iv) above, the computing progress rate
and the list $l^\prime=[l_1,l_2]$ where
$l_1$ is the abelian invariant of $H^3(H,\bZ)\simeq H^2(H,\bQ/\bZ)$
with respect to Smith normal form and
$l_2$ is the generator of $\widetilde{\partial}_{H,g}(l)$ in $H^3(H,\bZ)$
and returns true (resp. false)
if the generator $l$ is in $H^4_{\rm nr}(G,\bZ)$
(resp. is not in $H^4_{\rm nr}(G,\bZ)$)
for a free resolution $RG$ of $G$ and a generator $l$ of $H^4(G,\bZ)$.\\

For some cases of groups $G$ with $|G|=5^5$ or $7^5$, to execute the command
{\tt IsUnramifiedH3(}$RG$, $l${\tt )}
requires quite a lot of time. We may improve the command by a modified one {\tt IsUnramifiedH3(}$RG$, $l${\tt :Subgroup)}; this modified command
(with {\tt Subgroup} option) works well for $p=3$, $5$, $7$. The modified command is the old command augmented by a new requirement {\rm (v)}:

{\rm (v)} $H_{\rm p}^4(H,\bZ)\lneq H^4(H,\bZ)$.\\

We explain the reason for the condition (v). If $H_{\rm p}^4(H,\bZ)=H^4(H,\bZ)$, we claim that ${\rm Ker}(\widetilde{\partial}_{H,g})=H^4(G,\bZ)$. Thus ${\rm Ker}(\widetilde{\partial}_{H,g})$ arising from such subgroup $H$ may be ignored.

Here is the proof of the claim. By condition (iii), $g$ belongs to the center of $H$. Thus we may consider (a) ${\rm Ker}(\partial_{H,g})$ (where the pair $(H, \langle g \rangle)$ arises if we regard $H$ as a subgroup of $G$, and (b) ${\rm Ker}(\partial^{\prime}_{H,g})$ (where the pair $(H, \langle g \rangle)$ arises if we regard $H$ as a subgroup of $H$ itself (the map $\partial^{\prime}_{H,g}$ is defined in an obvious way). From Definition \ref{d2.11}, it is easy to verify that $\partial^{\prime}_{H,g} \circ {\rm res} = \partial_{H,g}$ where res denotes the restriction map. Suppose that $H_{\rm p}^4(H,\bZ)=H^4(H,\bZ)$. Then $\partial^{\prime}_{H,g}$ is the zero map. Thus $\partial_{H,g}$ is also the zero map, i.e. ${\rm Ker}(\widetilde{\partial}_{H,g})=H^4(G,\bZ)$ as we claimed before.

In conclusion, the new function
{\tt IsUnramifiedH3(}$RG$, $l${\tt :Subgroup)}
with {\tt Subgroup} option
is given as follows: In addition to the function {\tt IsUnramifiedH3(}$RG$, $l${\tt )} with the conditions (i)--(iv) for subgroups $H\leq G$, an extra condition (v) is added.

\bigskip
Finally, we may check the condition
$H_{\rm p}^4(H,\bZ)=H^4(H,\bZ)$ by the function
{\tt H4pFromResolution(}$RG${\tt )}, and it holds 
for the following groups $H$:\\
$H=G(p^2,i)$ $(p\in\{3,5,7\},\ i\in \{1,2\})$;\\
$H=G(p^3,i)$ $(p\in\{3,5,7\},\ i\in \{1,2,3,4\})$;\\
$H=G(3^4,i)$ $(i\in \{1,2,3,4,5,6,8,9,10,13,14\})$;\\
$H=G(p^4,i)$ $(p\in\{5,7\},\ i\in \{1,2,3,4,5,6,9,10,13,14\})$;\\
$H=G(3^5,i)$ $(i\in \{1,3,4,5,6,7,8,9,10,11,12,19,20,21,22,23,24,25,26,27,33,35,43,44,45,46,47,49,50,66\})$;\\
$H=G(5^5,i)$ $(i\in \{1,15,16,17,24,25,26,27,28,29,42,44,51,52,53,54,55,56,57,59,60,76\})$.\\

The function {\tt IsUnramifiedH3(}$RG$, $l${\tt :Subgroup)} is similar to
{\tt IsUnramifiedH3(}$RG$, $l${\tt)}
but we exclude the above subgroups $H\leq G$ in applying the algorithm for computing $H_{\rm nr}^4(G,\bZ)$:\\

\noindent
{\tt IsUnramifiedH3(}$RG$, $l${\tt :Subgroup)}
prints the number of pairs $(H,I)$ of $H\leq G$ and $I\leq Z(H)$
which satisfy (i)--(iv) but $H$ is not in the above groups (with $H_{\rm p}^4(H,\bZ)=H^4(H,\bZ)$),
the computing progress rate
and the list $l^\prime=[l_1,l_2]$ where
$l_1$ is the abelian invariant of $H^3(H,\bZ)\simeq H^2(H,\bQ/\bZ)$
with respect to Smith normal form and
$l_2$ is the generator of $\widetilde{\partial}_{H,g}(l)$ in $H^3(H,\bZ)$
and returns true (resp. false)
if the generator $l$ is in $H^4_{\rm nr}(G,\bZ)$
(resp. is not in $H^4_{\rm nr}(G,\bZ)$)
for a free resolution $RG$ of $G$ and a generator $l$ of $H^4(G,\bZ)$.\\

This enables us to compute $H_{\rm nr}^4(G,\bZ)$ within a reasonable time.

For $p=3,5$, by using the functions
{\tt H4pFromResolution(}$RG${\tt )} and
{\tt IsUnramifiedH3(}$RG$, $l${\tt :Subgroup)},
we evaluate the groups
$H^4(G,\bZ)$, $H^4_{\rm nr}(G,\bZ)$, $H^4_{\rm p}(G,\bZ)$
and $H^3_{\rm nr}(\bC(G),\bQ/\bZ)$ for all groups $G$ of order $p^5$.
For $p=7$, it takes much more time to reach the answer by using
the computer with GAP programs
(the computer we use is just a usual personal computer).
By Theorem \ref{t1.10},
if $G_1$ and $G_2$ belongs to the same isoclinism
family, then $H^3_{\rm nr}(\bC(G_1),\bQ/\bZ)$ and
$H^3_{\rm nr}(\bC(G_2),\bQ/\bZ)$ are isomorphic.
Hence for the case where $|G|=7^5$ we choose a suitable group
$G$ in each isoclinism family $\Phi_i$ where $i\in\{5,6,7,10\}$; when $p=3,5$, we may check all the groups in an isoclinism family, because the computing time is rather short.
Note that if $G$ belongs to $\Phi_i$ where $1\leq i\leq 4$ or $8\leq i\leq 9$,
then $\bC(G)$ is $\bC$-rational (see the paragraph after Theorem \ref{t1.18}),
and hence $H^3_{\rm nr}(\bC(G),\bQ/\bZ)=0$.

This finishes the proof.

\bigskip
In the following we demonstrate the computation of $H^4_{\rm nr}(G,\bZ)\simeq
H^3_{\rm nr}(G,\bQ/\bZ)$
for groups $G$ of order $3^5$ belonging to
$\Phi_7$ (Cases 1a--1e) and $\Phi_{10}$ (Cases 2a--2c), and
for groups $G$ of order $5^5$ belonging to
$\Phi_5$ (Cases 3a--3b), $\Phi_6$ (Cases 4a--4l),
$\Phi_7$ (Cases 5a--5e), $\Phi_{10}$ (Cases 6a--6f),
and also for groups $G$ of order $7^5$ (Cases 7a--7d). \\

Recall that
\begin{align*}
H^3(G,\bQ/\bZ)
\geq H_{\rm nr}^3(G,\bQ/\bZ)
\geq H_{\rm n}^3(G,\bQ/\bZ)
\geq H_{\rm p}^3(G,\bQ/\bZ)
\end{align*}
and $H_{\rm nr}^3(\bC(G),\bQ/\bZ)\simeq H_{\rm nr}^3(G,\bQ/\bZ)/H_{\rm n}^3(G,\bQ/\bZ)$.
Remember that, if $G$ is a group of odd order, then $H_{\rm n}^3(G,\bQ/\bZ)$
$=$ $H_{\rm p}^3(G,\bQ/\bZ)$ by Peyre's theorem (Theorem \ref{t2.14}).
Also remember that, if $\bC(G)$ is (stably, retract) $\bC$-rational, then $H_{\rm nr}^3(\bC(G),\bQ/\bZ)=0$ and therefore $H_{\rm nr}^3(G,\bQ/\bZ)=H_{\rm n}^3(G,\bQ/\bZ)$.\\

\newpage
A summary for $p=3$ :\\
\begin{table}[!h]\vspace*{-2mm}
\begin{tabular}{cl|ccc|c}
\multicolumn{2}{c|}{$|G|=3^5$} & $H^3(G,\bQ/\bZ)$ & $H^3_{\rm nr}(G,\bQ/\bZ)$
 & $H^3_{\rm p}(G,\bQ/\bZ)$
& $H^3_{\rm nr}(\bC(G),\bQ/\bZ)$\\\hline
&$G(3^5,56)$ & $(\bZ/3\bZ)^{\oplus 7}$
& $(\bZ/3\bZ)^{\oplus 6}$ & $(\bZ/3\bZ)^{\oplus 5}$
& $\bZ/3\bZ$\\
&$G(3^5,57)$ & $(\bZ/3\bZ)^{\oplus 6}$
& $(\bZ/3\bZ)^{\oplus 6}$ & $(\bZ/3\bZ)^{\oplus 5}$
& $\bZ/3\bZ$\\
$\Phi_7$ & $G(3^5,58)$ & $(\bZ/3\bZ)^{\oplus 9}$
& $(\bZ/3\bZ)^{\oplus 7}$ & $(\bZ/3\bZ)^{\oplus 6}$
& $\bZ/3\bZ$\\
&$G(3^5,59)$ & $(\bZ/3\bZ)^{\oplus 6}$
& $(\bZ/3\bZ)^{\oplus 6}$ & $(\bZ/3\bZ)^{\oplus 5}$
& $\bZ/3\bZ$\\
&$G(3^5,60)$ & $(\bZ/3\bZ)^{\oplus 6}$
& $(\bZ/3\bZ)^{\oplus 6}$ & $(\bZ/3\bZ)^{\oplus 5}$
& $\bZ/3\bZ$\\\hline
&$G(3^5,28)$ & $(\bZ/3\bZ)^{\oplus 2}\oplus \bZ/9\bZ$
& $(\bZ/3\bZ)^{\oplus 3}$ & $(\bZ/3\bZ)^{\oplus 3}$ 
& 0\\
$\Phi_{10}$&$G(3^5,29)$ & $\bZ/3\bZ\oplus \bZ/9\bZ$
& $(\bZ/3\bZ)^{\oplus 2}$ & $(\bZ/3\bZ)^{\oplus 2}$ 
& 0\\
& $G(3^5,30)$ & $\bZ/3\bZ\oplus \bZ/9\bZ$
& $(\bZ/3\bZ)^{\oplus 2}$ & $(\bZ/3\bZ)^{\oplus 2}$ 
& 0
\end{tabular}\\
\vspace*{2mm}
Table $3$: $H^3_{\rm nr}(\bC(G),\bQ/\bZ)$ for groups $G$ of order $3^5$
which belong to $\Phi_7$ or $\Phi_{10}$
\end{table}

\medskip
{\bf Case 1a: $G=G(3^5,56)$ which belongs to $\Phi_7$.}\\
We have
$H^4(G,\bZ)\simeq(\bZ/3\bZ)^{\oplus 7}=\langle f_1,\ldots,f_7\rangle$,
$H^4_{\rm p}(G,\bZ)\simeq (\bZ/3\bZ)^{\oplus 5}=\langle f_1,f_2,f_3,f_5f_6,f_7\rangle$, $f_4\not\in H^4_{\rm nr}(G,\bZ)$ and
$f_5\in H^4_{\rm nr}(G,\bZ)$.
Hence we get $H^3_{\rm nr}(\bC(G),\bQ/\bZ)\simeq H^4_{\rm nr}(\bC(G),\bZ)\simeq \bZ/3\bZ$.
\medskip\begin{verbatim}
gap> Read("H3nr.gap");

gap> G56:=SmallGroup(243,56);
<pc group of size 243 with 5 generators>
gap> RG56:=ResolutionNormalSeries(LowerCentralSeries(G56),5);
Resolution of length 5 in characteristic 0 for <pc group with 243 generators> .
gap> CR_CocyclesAndCoboundaries(RG56,4,true).torsionCoefficients; # H^4(G,Z)
[ [ 3, 3, 3, 3, 3, 3, 3 ] ]
gap> H4pFromResolution(RG56);
12[ [ 27, 5 ], [ 27, 5 ], [ 27, 5 ], [ 27, 5 ], [ 27, 2 ], [ 27, 5 ],
    [ 27, 2 ], [ 81, 13 ], [ 81, 12 ], [ 81, 13 ], [ 81, 12 ], [ 243, 56 ] ]
1/12
...
12/12
[ [ 3, 3, 3, 3, 3 ],
[ [ 3, 3, 3, 3, 3, 3, 3 ],
  [ [ 1, 0, 0, 0, 0, 0, 0 ], [ 0, 1, 0, 0, 0, 0, 0 ], [ 0, 0, 1, 0, 0, 0, 0 ],
    [ 0, 0, 0, 0, 1, 1, 0 ], [ 0, 0, 0, 0, 0, 0, 1 ] ] ] ]
gap> IsUnramifiedH3(RG56,[0,0,0,1,0,0,0]:Subgroup);
1/9
[ [ 3, 3, 3 ], [ 2, 1, 1 ] ]
false
gap> IsUnramifiedH3(RG56,[0,0,0,0,1,0,0]:Subgroup);
1/9
[ [ 3, 3, 3 ], [ 0, 0, 0 ] ]
...
9/9
[ [ 3, 3, 3 ], [ 0, 0, 0 ] ]
true
\end{verbatim}
\medskip
{\bf Case 1b: $G=G(3^5,57)$ which belongs to $\Phi_7$.}\\
We have
$H^4(G,\bZ)\simeq(\bZ/3\bZ)^{\oplus 6}=\langle f_1,\ldots,f_6\rangle$,
$H^4_{\rm p}(G,\bZ)\simeq (\bZ/3\bZ)^{\oplus 5}=\langle f_1,f_2,f_3,f_4,f_6\rangle$ and $f_5\in H^4_{\rm nr}(G,\bZ)$.
Hence we get $H^3_{\rm nr}(\bC(G),\bQ/\bZ)\simeq H^4_{\rm nr}(\bC(G),\bZ)\simeq \bZ/3\bZ$.
\medskip\begin{verbatim}
gap> G57:=SmallGroup(243,57);
<pc group of size 243 with 5 generators>
gap> RG57:=ResolutionNormalSeries(LowerCentralSeries(G57),5);
Resolution of length 5 in characteristic 0 for <pc group with 243 generators> .
gap> CR_CocyclesAndCoboundaries(RG57,4,true).torsionCoefficients; # H^4(G,Z)
[ [ 3, 3, 3, 3, 3, 3 ] ]
gap> H4pFromResolution(RG57);
12[ [ 27, 2 ], [ 27, 2 ], [ 27, 5 ], [ 27, 2 ], [ 27, 2 ], [ 27, 5 ],
    [ 27, 2 ], [ 81, 13 ], [ 81, 13 ], [ 81, 13 ], [ 81, 12 ], [ 243, 57 ] ]
1/12
...
12/12
[ [ 3, 3, 3, 3, 3 ],
[ [ 3, 3, 3, 3, 3, 3 ],
  [ [ 1, 0, 0, 0, 0, 0 ], [ 0, 1, 0, 0, 0, 0 ], [ 0, 0, 1, 0, 0, 0 ],
    [ 0, 0, 0, 1, 0, 0 ], [ 0, 0, 0, 0, 0, 1 ] ] ] ]
gap> IsUnramifiedH3(RG57,[0,0,0,0,1,0]:Subgroup);
1/5
[ [ 3, 3, 3 ], [ 0, 0, 0 ] ]
...
5/5
[ [ 3, 3, 3 ], [ 0, 0, 0 ] ]
true
\end{verbatim}
\medskip
{\bf Case 1c: $G=G(3^5,58)$ which belongs to $\Phi_7$.}\\
We have
$H^4(G,\bZ)\simeq(\bZ/3\bZ)^{\oplus 9}=\langle f_1,\ldots,f_9\rangle$,
$H^4_{\rm p}(G,\bZ)\simeq (\bZ/3\bZ)^{\oplus 6}=\langle f_1,f_2f_8,f_4f_8,f_5f_9,f_6,f_7\rangle$, $f_3, f_9\not\in H^4_{\rm nr}(G,\bZ)$ and
$f_8\in H^4_{\rm nr}(G,\bZ)$.
Hence we get $H^3_{\rm nr}(\bC(G),\bQ/\bZ)\simeq H^4_{\rm nr}(\bC(G),\bZ)\simeq \bZ/3\bZ$.
\medskip\begin{verbatim}
gap> G58:=SmallGroup(243,58);
<pc group of size 243 with 5 generators>
gap> RG58:=ResolutionNormalSeries(LowerCentralSeries(G58),5);
Resolution of length 5 in characteristic 0 for <pc group with 243 generators> .
gap> CR_CocyclesAndCoboundaries(RG58,4,true).torsionCoefficients; # H^4(G,Z)
[ [ 3, 3, 3, 3, 3, 3, 3, 3, 3 ] ]
gap> H4pFromResolution(RG58);
12[ [ 27, 2 ], [ 27, 2 ], [ 27, 2 ], [ 27, 5 ], [ 27, 5 ], [ 27, 5 ],
    [ 27, 5 ], [ 81, 13 ], [ 81, 12 ], [ 81, 12 ], [ 81, 12 ], [ 243, 58 ] ]
1/12
...
12/12
[ [ 3, 3, 3, 3, 3, 3 ],
[ [ 3, 3, 3, 3, 3, 3, 3, 3, 3 ],
  [ [ 1, 0, 0, 0, 0, 0, 0, 0, 0 ], [ 0, 1, 0, 0, 0, 0, 0, 1, 0 ],
    [ 0, 0, 0, 1, 0, 0, 0, 1, 0 ], [ 0, 0, 0, 0, 1, 0, 0, 0, 1 ],
    [ 0, 0, 0, 0, 0, 1, 0, 0, 0 ], [ 0, 0, 0, 0, 0, 0, 1, 0, 0 ] ] ] ]
gap> IsUnramifiedH3(RG58,[0,0,1,0,0,0,0,0,0]:Subgroup);
1/13
[ [ 3, 3, 3, 3 ], [ 0, 0, 0, 1 ] ]
false
gap> IsUnramifiedH3(RG58,[0,0,0,0,0,0,0,1,0]:Subgroup);
1/13
[ [ 3, 3, 3, 3 ], [ 0, 0, 0, 0 ] ]
...
13/13
[ [ 3, 3, 3, 3 ], [ 0, 0, 0, 0 ] ]
true
gap> IsUnramifiedH3(RG58,[0,0,0,0,0,0,0,0,1]:Subgroup);
1/13
[ [ 3, 3, 3, 3 ], [ 0, 0, 2, 1 ] ]
false
\end{verbatim}
\medskip
{\bf Case 1d: $G=G(3^5,59)$ which belongs to $\Phi_7$.}\\
We have
$H^4(G,\bZ)\simeq(\bZ/3\bZ)^{\oplus 6}=\langle f_1,\ldots,f_6\rangle$,
$H^4_{\rm p}(G,\bZ)\simeq (\bZ/3\bZ)^{\oplus 5}=\langle f_1,f_2,f_3,f_4,f_5f_6\rangle$ and $f_5\in H^4_{\rm nr}(G,\bZ)$.
Hence we get $H^3_{\rm nr}(\bC(G),\bQ/\bZ)\simeq H^4_{\rm nr}(\bC(G),\bZ)\simeq \bZ/3\bZ$.
\medskip\begin{verbatim}
gap> G59:=SmallGroup(243,59);
<pc group of size 243 with 5 generators>
gap> RG59:=ResolutionNormalSeries(LowerCentralSeries(G59),5);
Resolution of length 5 in characteristic 0 for <pc group with 243 generators> .
gap> CR_CocyclesAndCoboundaries(RG59,4,true).torsionCoefficients; # H^4(G,Z)
[ [ 3, 3, 3, 3, 3, 3 ] ]
gap> H4pFromResolution(RG59);
12[ [ 27, 2 ], [ 27, 2 ], [ 27, 2 ], [ 27, 5 ], [ 27, 2 ], [ 27, 2 ],
  [ 27, 2 ], [ 81, 13 ], [ 81, 13 ], [ 81, 13 ], [ 81, 13 ], [ 243, 59 ] ]
1/12
...
12/12
[ [ 3, 3, 3, 3, 3 ],
[ [ 3, 3, 3, 3, 3, 3 ],
  [ [ 1, 0, 0, 0, 0, 0 ], [ 0, 1, 0, 0, 0, 0 ], [ 0, 0, 1, 0, 0, 0 ],
    [ 0, 0, 0, 1, 0, 0 ], [ 0, 0, 0, 0, 1, 1 ] ] ] ]
gap> IsUnramifiedH3(RG59,[0,0,0,0,1,0]:Subgroup);
1/1
[ [ 3, 3, 3 ], [ 0, 0, 0 ] ]
true
\end{verbatim}
\medskip
{\bf Case 1e}: $G=G(3^5,60)$ which belongs to $\Phi_7$.}\\
We have
$H^4(G,\bZ)\simeq(\bZ/3\bZ)^{\oplus 6}=\langle f_1,\ldots,f_6\rangle$,
$H^4_{\rm p}(G,\bZ)\simeq (\bZ/3\bZ)^{\oplus 5}=\langle f_1,f_2,f_3,f_4,f_6\rangle$ and $f_5\in H^4_{\rm nr}(G,\bZ)$.
Hence we get $H^3_{\rm nr}(\bC(G),\bQ/\bZ)\simeq H^4_{\rm nr}(\bC(G),\bZ)\simeq \bZ/3\bZ$.
\medskip\begin{verbatim}
gap> G60:=SmallGroup(243,60);
<pc group of size 243 with 5 generators>
gap> RG60:=ResolutionNormalSeries(LowerCentralSeries(G60),5);
Resolution of length 5 in characteristic 0 for <pc group with 243 generators> .
gap> CR_CocyclesAndCoboundaries(RG60,4,true).torsionCoefficients; # H^4(G,Z)
[ [ 3, 3, 3, 3, 3, 3 ] ]
gap> H4pFromResolution(RG60);
12[ [ 27, 2 ], [ 27, 2 ], [ 27, 5 ], [ 27, 2 ], [ 27, 2 ], [ 27, 2 ],
    [ 27, 2 ], [ 81, 14 ], [ 81, 14 ], [ 81, 13 ], [ 81, 14 ], [ 243, 60 ] ]
1/12
...
12/12
[ [ 3, 3, 3, 3, 3 ],
[ [ 3, 3, 3, 3, 3, 3 ],
  [ [ 1, 0, 0, 0, 0, 0 ], [ 0, 1, 0, 0, 0, 0 ], [ 0, 0, 1, 0, 0, 0 ],
    [ 0, 0, 0, 1, 0, 0 ], [ 0, 0, 0, 0, 0, 1 ] ] ] ]
gap> IsUnramifiedH3(RG60,[0,0,0,0,1,0]:Subgroup);
1/2
[ [ 3, 3, 3 ], [ 0, 0, 0 ] ]
2/2
[ [ 3, 3, 3 ], [ 0, 0, 0 ] ]
true
\end{verbatim}
\medskip
{\bf Case 2a: $G=G(3^5,28)$ which belongs to $\Phi_{10}$.}\\
We have
$H^4(G,\bZ)\simeq(\bZ/3\bZ)^{\oplus 2}\oplus\bZ/9\bZ
=\langle f_1,f_2,f_3\rangle$,
$H^4_{\rm p}(G,\bZ)\simeq (\bZ/3\bZ)^{\oplus 3}=\langle f_1,f_2,f_3^3\rangle$, $f_3\not\in H^4_{\rm nr}(G,\bZ)$.
Hence we get $H^3_{\rm nr}(\bC(G),\bQ/\bZ)\simeq H^4_{\rm nr}(\bC(G),\bZ)=0$.
\medskip
\begin{verbatim}
gap> G28:=SmallGroup(243,28);
<pc group of size 243 with 5 generators>
gap> RG28:=ResolutionNormalSeries(LowerCentralSeries(G28),5);
Resolution of length 5 in characteristic 0 for <pc group with 243 generators> .
gap> CR_CocyclesAndCoboundaries(RG28,4,true).torsionCoefficients; # H^4(G,Z)
[ [ 3, 3, 9 ] ]
gap> H4pFromResolution(RG28);
13[ [ 9, 2 ], [ 9, 2 ], [ 9, 1 ], [ 27, 2 ], [ 27, 2 ], [ 27, 3 ], [ 27, 4 ],
    [ 27, 3 ], [ 81, 9 ], [ 81, 9 ], [ 81, 10 ], [ 81, 4 ], [ 243, 28 ] ]
1/13
...
13/13
[ [ 3, 3, 3 ],
[ [ 3, 3, 9 ],
  [ [ 1, 0, 0 ], [ 0, 1, 0 ], [ 0, 0, 3 ] ] ] ]
gap> IsUnramifiedH3(RG28,[0,0,1]:Subgroup);
1/1
[ [ 9 ], [ 6 ] ]
false
\end{verbatim}
\medskip
{\bf Case 2b: $G=G(3^5,29)$ which belongs to $\Phi_{10}$.}\\
We have
$H^4(G,\bZ)\simeq\bZ/3\bZ\oplus\bZ/9\bZ=\langle f_1,f_2\rangle$,
$H^4_{\rm p}(G,\bZ)\simeq (\bZ/3\bZ)^{\oplus 2}=\langle f_1,f_2^3\rangle$
and $f_2\not\in H^4_{\rm nr}(G,\bZ)$.
Hence we get $H^3_{\rm nr}(\bC(G),\bQ/\bZ)\simeq H^4_{\rm nr}(\bC(G),\bZ)=0$.
\medskip\begin{verbatim}
gap> G29:=SmallGroup(243,29);
<pc group of size 243 with 5 generators>
gap> RG29:=ResolutionNormalSeries(LowerCentralSeries(G29),5);
Resolution of length 5 in characteristic 0 for <pc group with 243 generators> .
gap> CR_CocyclesAndCoboundaries(RG29,4,true).torsionCoefficients; # H^4(G,Z)
[ [ 3, 9 ] ]
gap> H4pFromResolution(RG29);
13[ [ 9, 1 ], [ 9, 1 ], [ 9, 1 ], [ 27, 2 ], [ 27, 2 ], [ 27, 4 ], [ 27, 4 ],
    [ 27, 4 ], [ 81, 10 ], [ 81, 10 ], [ 81, 10 ], [ 81, 4 ], [ 243, 29 ] ]
1/13
...
13/13
[ [ 3, 3 ],
[ [ 3, 9 ],
  [ [ 1, 0 ], [ 0, 3 ] ] ] ]
gap> IsUnramifiedH3(RG29,[0,1]:Subgroup);
1/1
[ [ 3 ], [ 2 ] ]
false
\end{verbatim}
\medskip
{\bf Case 2c: $G=G(3^5,30)$ which belongs to $\Phi_{10}$.}\\
We have
$H^4(G,\bZ)\simeq\bZ/3\bZ\oplus\bZ/9\bZ=\langle f_1,f_2\rangle$,
$H^4_{\rm p}(G,\bZ)\simeq (\bZ/3\bZ)^{\oplus 2}=\langle f_1,f_2^3\rangle$
and $f_2\not\in H^4_{\rm nr}(G,\bZ)$.
Hence we get $H^3_{\rm nr}(\bC(G),\bQ/\bZ)\simeq H^4_{\rm nr}(\bC(G),\bZ)=0$.
\medskip\begin{verbatim}
gap> G30:=SmallGroup(243,30);
<pc group of size 243 with 5 generators>
gap> RG30:=ResolutionNormalSeries(LowerCentralSeries(G30),5);
Resolution of length 5 in characteristic 0 for <pc group with 243 generators> .
gap> CR_CocyclesAndCoboundaries(RG30,4,true).torsionCoefficients; # H^4(G,Z)
[ [ 3, 9 ] ]
gap> H4pFromResolution(RG30);
13[ [ 9, 2 ], [ 9, 1 ], [ 9, 1 ], [ 27, 4 ], [ 27, 2 ], [ 27, 2 ], [ 27, 4 ],
    [ 27, 3 ], [ 81, 9 ], [ 81, 10 ], [ 81, 4 ], [ 81, 10 ], [ 243, 30 ] ]
1/13
...
13/13
[ [ 3, 3 ],
[ [ 3, 9 ],
  [ [ 1, 0 ], [ 0, 3 ] ] ] ]
gap> IsUnramifiedH3(RG30,[0,1]:Subgroup);
1/1
[ [ 3 ], [ 2 ] ]
false
\end{verbatim}

\newpage
A summary for $p=5$ :\\
\begin{table}[!h]\vspace*{-2mm}
\begin{tabular}{cl|ccc|c}
\multicolumn{2}{c|}{$|G|=5^5$} & $H^3(G,\bQ/\bZ)$ & $H^3_{\rm nr}(G,\bQ/\bZ)$
 & $H^3_{\rm p}(G,\bQ/\bZ)$
& $H^3_{\rm nr}(\bC(G),\bQ/\bZ)$\\\hline
\raisebox{-1.6ex}[0cm][0cm]{$\Phi_5$}
&$G(5^5,75)$ & $(\bZ/5\bZ)^{\oplus 15}$
& $(\bZ/5\bZ)^{\oplus 10}$ & $(\bZ/5\bZ)^{\oplus 10}$
& 0\\
&$G(5^5,76)$ & $(\bZ/5\bZ)^{\oplus 10}$
& $(\bZ/5\bZ)^{\oplus 10}$ & $(\bZ/5\bZ)^{\oplus 10}$ 
& 0\\\hline
&$G(5^5,3)$ & $(\bZ/5\bZ)^{\oplus 7}$
& $(\bZ/5\bZ)^{\oplus 7}$ & $(\bZ/5\bZ)^{\oplus 6}$ 
& $\bZ/5\bZ$\\
&$G(5^5,4)$ & $(\bZ/5\bZ)^{\oplus 5}$
& $(\bZ/5\bZ)^{\oplus 5}$ & $(\bZ/5\bZ)^{\oplus 4}$ 
& $\bZ/5\bZ$\\
&$G(5^5,5)$ & $(\bZ/5\bZ)^{\oplus 5}$
& $(\bZ/5\bZ)^{\oplus 5}$ & $(\bZ/5\bZ)^{\oplus 4}$ 
& $\bZ/5\bZ$\\
&$G(5^5,6)$ & $(\bZ/5\bZ)^{\oplus 5}$
& $(\bZ/5\bZ)^{\oplus 5}$ & $(\bZ/5\bZ)^{\oplus 4}$ 
& $\bZ/5\bZ$\\
&$G(5^5,7)$ & $(\bZ/5\bZ)^{\oplus 3}\oplus\bZ/25\bZ$
& $(\bZ/5\bZ)^{\oplus 3}\oplus\bZ/25\bZ$ & $(\bZ/5\bZ)^{\oplus 2}\oplus\bZ/25\bZ$ 
& $\bZ/5\bZ$\\
\raisebox{-1.6ex}[0cm][0cm]{$\Phi_6$}
&$G(5^5,8)$ & $(\bZ/5\bZ)^{\oplus 4}$
& $(\bZ/5\bZ)^{\oplus 4}$ & $(\bZ/5\bZ)^{\oplus 3}$ 
& $\bZ/5\bZ$\\
&$G(5^5,9)$ & $(\bZ/5\bZ)^{\oplus 4}$
& $(\bZ/5\bZ)^{\oplus 4}$ & $(\bZ/5\bZ)^{\oplus 3}$ 
& $\bZ/5\bZ$\\
&$G(5^5,10)$ & $(\bZ/5\bZ)^{\oplus 3}\oplus\bZ/25\bZ$
& $(\bZ/5\bZ)^{\oplus 3}\oplus\bZ/25\bZ$ & $(\bZ/5\bZ)^{\oplus 2}\oplus\bZ/25\bZ$ 
& $\bZ/5\bZ$\\
&$G(5^5,11)$ & $(\bZ/5\bZ)^{\oplus 4}$
& $(\bZ/5\bZ)^{\oplus 4}$ & $(\bZ/5\bZ)^{\oplus 3}$ 
& $\bZ/5\bZ$\\
&$G(5^5,12)$ & $(\bZ/5\bZ)^{\oplus 4}$
& $(\bZ/5\bZ)^{\oplus 4}$ & $(\bZ/5\bZ)^{\oplus 3}$ 
& $\bZ/5\bZ$\\
&$G(5^5,13)$ & $(\bZ/5\bZ)^{\oplus 4}$
& $(\bZ/5\bZ)^{\oplus 4}$ & $(\bZ/5\bZ)^{\oplus 3}$ 
& $\bZ/5\bZ$\\
&$G(5^5,14)$ & $(\bZ/5\bZ)^{\oplus 4}$
& $(\bZ/5\bZ)^{\oplus 4}$ & $(\bZ/5\bZ)^{\oplus 3}$ 
& $\bZ/5\bZ$\\\hline
&$G(5^5,66)$ & $(\bZ/5\bZ)^{\oplus 10}$
& $(\bZ/5\bZ)^{\oplus 7}$ & $(\bZ/5\bZ)^{\oplus 6}$ 
& $\bZ/5\bZ$\\
&$G(5^5,67)$ & $(\bZ/5\bZ)^{\oplus 6}$
& $(\bZ/5\bZ)^{\oplus 6}$ & $(\bZ/5\bZ)^{\oplus 5}$ 
& $\bZ/5\bZ$\\
$\Phi_7$&$G(5^5,68)$ & $(\bZ/5\bZ)^{\oplus 6}$
& $(\bZ/5\bZ)^{\oplus 6}$ & $(\bZ/5\bZ)^{\oplus 5}$ 
& $\bZ/5\bZ$\\
&$G(5^5,69)$ & $(\bZ/5\bZ)^{\oplus 6}$
& $(\bZ/5\bZ)^{\oplus 6}$ & $(\bZ/5\bZ)^{\oplus 5}$ 
& $\bZ/5\bZ$\\
&$G(5^5,70)$ & $(\bZ/5\bZ)^{\oplus 6}$
& $(\bZ/5\bZ)^{\oplus 6}$ & $(\bZ/5\bZ)^{\oplus 5}$ 
& $\bZ/5\bZ$\\\hline
&$G(5^5,33)$ & $(\bZ/5\bZ)^{\oplus 6}$
& $(\bZ/5\bZ)^{\oplus 5}$ & $(\bZ/5\bZ)^{\oplus 4}$ 
& $\bZ/5\bZ$\\
&$G(5^5,34)$ & $(\bZ/5\bZ)^{\oplus 4}$
& $(\bZ/5\bZ)^{\oplus 3}$ & $(\bZ/5\bZ)^{\oplus 2}$ 
& $\bZ/5\bZ$\\
\raisebox{-1.6ex}[0cm][0cm]{$\Phi_{10}$}
&$G(5^5,35)$ & $(\bZ/5\bZ)^{\oplus 4}$
& $(\bZ/5\bZ)^{\oplus 3}$ & $(\bZ/5\bZ)^{\oplus 2}$ 
& $\bZ/5\bZ$\\
&$G(5^5,36)$ & $(\bZ/5\bZ)^{\oplus 4}$
& $(\bZ/5\bZ)^{\oplus 3}$ & $(\bZ/5\bZ)^{\oplus 2}$ 
& $\bZ/5\bZ$\\
&$G(5^5,37)$ & $(\bZ/5\bZ)^{\oplus 4}$
& $(\bZ/5\bZ)^{\oplus 3}$ & $(\bZ/5\bZ)^{\oplus 2}$ 
& $\bZ/5\bZ$\\
&$G(5^5,38)$ & $(\bZ/5\bZ)^{\oplus 3}$
& $(\bZ/5\bZ)^{\oplus 3}$ & $(\bZ/5\bZ)^{\oplus 2}$ 
& $\bZ/5\bZ$\\
\end{tabular}\\
\vspace*{2mm}
Table $4$: $H^3_{\rm nr}(\bC(G),\bQ/\bZ)$ for groups $G$ of order $5^5$
which belong to $\Phi_5$, $\Phi_6$, $\Phi_7$ or $\Phi_{10}$
\end{table}

\medskip
{\bf Case 3a: $G=G(5^5,75)$ which belongs to $\Phi_5$.}\\
We have
$H^4(G,\bZ)\simeq(\bZ/5\bZ)^{\oplus 15}=\langle f_1,\ldots,f_{15}\rangle$,
$H^4_{\rm p}(G,\bZ)\simeq (\bZ/5\bZ)^{\oplus 10}=\langle f_1,f_2,f_3,f_4,f_5f_8f_{15},f_7,f_{11},f_{12},f_{13},f_{14}\rangle$ and
$f_6,f_8,f_9,f_{10},f_{15}\not\in H^4_{\rm nr}(G,\bZ)$.
Hence we get $H^3_{\rm nr}(\bC(G),\bQ/\bZ)\simeq H^4_{\rm nr}(\bC(G),\bZ)=0$.
\medskip\begin{verbatim}
gap> Read("H3nr.gap");

gap> G75:=SmallGroup(5^5,75);
<pc group of size 3125 with 5 generators>
gap> RG75:=ResolutionNormalSeries(LowerCentralSeries(G75),5);
Resolution of length 5 in characteristic 0 for <pc group with 3125 generators> .
gap> CR_CocyclesAndCoboundaries(RG75,4,true).torsionCoefficients; # H^4(G,Z)
[ [ 5, 5, 5, 5, 5, 5, 5, 5, 5, 5, 5, 5, 5, 5, 5 ] ]
gap> H4pFromResolution(RG75);
157[ [ 125, 5 ], ...,  [ 125, 5 ], [ 3125, 75 ] ]
1/157
...
157/157
[ [ 5, 5, 5, 5, 5, 5, 5, 5, 5, 5 ],
[ [ 5, 5, 5, 5, 5, 5, 5, 5, 5, 5, 5, 5, 5, 5, 5 ],
  [ [ 1, 0, 0, 0, 0, 0, 0, 0, 0, 0, 0, 0, 0, 0, 0 ],
    [ 0, 1, 0, 0, 0, 0, 0, 0, 0, 0, 0, 0, 0, 0, 0 ],
    [ 0, 0, 1, 0, 0, 0, 0, 0, 0, 0, 0, 0, 0, 0, 0 ],
    [ 0, 0, 0, 1, 0, 0, 0, 0, 0, 0, 0, 0, 0, 0, 0 ],
    [ 0, 0, 0, 0, 1, 0, 0, 1, 0, 0, 0, 0, 0, 0, 1 ],
    [ 0, 0, 0, 0, 0, 0, 1, 0, 0, 0, 0, 0, 0, 0, 0 ],
    [ 0, 0, 0, 0, 0, 0, 0, 0, 0, 0, 1, 0, 0, 0, 0 ],
    [ 0, 0, 0, 0, 0, 0, 0, 0, 0, 0, 0, 1, 0, 0, 0 ],
    [ 0, 0, 0, 0, 0, 0, 0, 0, 0, 0, 0, 0, 1, 0, 0 ],
    [ 0, 0, 0, 0, 0, 0, 0, 0, 0, 0, 0, 0, 0, 1, 0 ] ] ] ]
gap> IsUnramifiedH3(RG75,[0,0,0,0,0,1,0,0,0,0,0,0,0,0,0]:Subgroup);
1/157
[ [ 5, 5, 5, 5, 5 ], [ 0, 0, 1, 0, 0 ] ]
false
gap> IsUnramifiedH3(RG75,[0,0,0,0,0,0,0,1,0,0,0,0,0,0,0]:Subgroup);
1/157
[ [ 5, 5, 5, 5, 5 ], [ 0, 0, 0, 4, 1 ] ]
false
gap> IsUnramifiedH3(RG75,[0,0,0,0,0,0,0,0,1,0,0,0,0,0,0]:Subgroup);
1/157
[ [ 5, 5, 5, 5, 5 ], [ 0, 1, 4, 4, 4 ] ]
false
gap> IsUnramifiedH3(RG75,[0,0,0,0,0,0,0,0,0,1,0,0,0,0,0]:Subgroup);
1/157
[ [ 5, 5, 5, 5, 5 ], [ 4, 0, 0, 4, 0 ] ]
false
gap> IsUnramifiedH3(RG75,[0,0,0,0,0,0,0,0,0,0,0,0,0,0,1]:Subgroup);
1/157
[ [ 5, 5, 5, 5, 5 ], [ 0, 0, 0, 2, 0 ] ]
false
\end{verbatim}
\medskip
{\bf Case 3b: $G=G(5^5,76)$ which belongs to $\Phi_5$.}\\
We have
$H^4(G,\bZ)\simeq(\bZ/5\bZ)^{\oplus 10}=\langle f_1,\ldots,f_{10}\rangle$ and
$H^4_{\rm p}(G,\bZ)\simeq (\bZ/5\bZ)^{\oplus 10}=\langle f_1,\ldots,f_{10}\rangle$.
Hence we get $H^3_{\rm nr}(\bC(G),\bQ/\bZ)\simeq H^4_{\rm nr}(\bC(G),\bZ)=0$.
\medskip\begin{verbatim}
gap> G76:=SmallGroup(5^5,76);
<pc group of size 3125 with 5 generators>
gap> RG76:=ResolutionNormalSeries(LowerCentralSeries(G76),5);
Resolution of length 5 in characteristic 0 for <pc group with 3125 generators> .
gap> CR_CocyclesAndCoboundaries(RG76,4,true).torsionCoefficients; # H^4(G,Z)
[ [ 5, 5, 5, 5, 5, 5, 5, 5, 5, 5 ] ]
gap> H4pFromResolution(RG76);
157[ [ 125, 2 ], ..., [ 125, 2 ], [ 3125, 76 ] ]
1/157
...
157/157
[ [ 5, 5, 5, 5, 5, 5, 5, 5, 5, 5 ],
[ [ 5, 5, 5, 5, 5, 5, 5, 5, 5, 5 ],
  [ [ 1, 0, 0, 0, 0, 0, 0, 0, 0, 0 ],
    [ 0, 1, 0, 0, 0, 0, 0, 0, 0, 0 ],
    [ 0, 0, 1, 0, 0, 0, 0, 0, 0, 0 ],
    [ 0, 0, 0, 1, 0, 0, 0, 0, 0, 0 ],
    [ 0, 0, 0, 0, 1, 0, 0, 0, 0, 0 ],
    [ 0, 0, 0, 0, 0, 1, 0, 0, 0, 0 ],
    [ 0, 0, 0, 0, 0, 0, 1, 0, 0, 0 ],
    [ 0, 0, 0, 0, 0, 0, 0, 1, 0, 0 ],
    [ 0, 0, 0, 0, 0, 0, 0, 0, 1, 0 ],
    [ 0, 0, 0, 0, 0, 0, 0, 0, 0, 1 ] ] ] ]
\end{verbatim}
\medskip
{\bf Case 4a: $G=G(5^5,3)$ which belongs to $\Phi_6$.}\\
We have
$H^4(G,\bZ)\simeq(\bZ/5\bZ)^{\oplus 7}=\langle f_1,\ldots,f_7\rangle$,
$H^4_{\rm p}(G,\bZ)\simeq (\bZ/5\bZ)^{\oplus 6}=\langle f_1f_6^2,f_2f_6^4,f_3,f_4f_6,f_5,f_7\rangle$ and $f_6\in H^4_{\rm nr}(G,\bZ)$.
Hence we get $H^3_{\rm nr}(\bC(G),\bQ/\bZ)\simeq H^4_{\rm nr}(\bC(G),\bZ)\simeq \bZ/5\bZ$.
\medskip\begin{verbatim}
gap> G3:=SmallGroup(5^5,3);
<pc group of size 3125 with 5 generators>
gap> RG3:=ResolutionNormalSeries(LowerCentralSeries(G3),5);
Resolution of length 5 in characteristic 0 for <pc group with 3125 generators> .
gap> CR_CocyclesAndCoboundaries(RG3,4,true).torsionCoefficients; # H^4(G,Z)
[ [ 5, 5, 5, 5, 5, 5, 5 ] ]
gap> H4pFromResolution(RG3);
14[ [ 125, 5 ], [ 125, 5 ], [ 125, 5 ], [ 125, 5 ], [ 125, 5 ], [ 125, 5 ],
    [ 125, 5 ], [ 625, 12 ], [ 625, 12 ], [ 625, 12 ], [ 625, 12 ],
    [ 625, 12 ], [ 625, 12 ], [ 3125, 3 ] ]
1/14
...
14/14
[ [ 5, 5, 5, 5, 5, 5 ],
[ [ 5, 5, 5, 5, 5, 5, 5 ],
  [ [ 1, 0, 0, 0, 0, 2, 0 ], [ 0, 1, 0, 0, 0, 4, 0 ],
    [ 0, 0, 1, 0, 0, 0, 0 ], [ 0, 0, 0, 1, 0, 1, 0 ],
    [ 0, 0, 0, 0, 1, 0, 0 ], [ 0, 0, 0, 0, 0, 0, 1 ] ] ] ]
gap> IsUnramifiedH3(RG3,[0,0,0,0,0,1,0]:Subgroup);
1/37
[ [ 5, 5, 5 ], [ 0, 0, 0 ] ]
...
37/37
[ [ 5, 5, 5 ], [ 0, 0, 0 ] ]
true
\end{verbatim}
\medskip
{\bf Case 4b: $G=G(5^5,4)$ which belongs to $\Phi_6$.}\\
We have
$H^4(G,\bZ)\simeq(\bZ/5\bZ)^{\oplus 5}=\langle f_1,\ldots,f_5\rangle$,
$H^4_{\rm p}(G,\bZ)\simeq (\bZ/5\bZ)^{\oplus 4}=\langle f_2,f_3,f_4,f_5\rangle$ and $f_1\in H^4_{\rm nr}(G,\bZ)$.
Hence we get $H^3_{\rm nr}(\bC(G),\bQ/\bZ)\simeq H^4_{\rm nr}(\bC(G),\bZ)\simeq \bZ/5\bZ$.
\medskip\begin{verbatim}
gap> G4:=SmallGroup(5^5,4);
<pc group of size 3125 with 5 generators>
gap> RG4:=ResolutionNormalSeries(LowerCentralSeries(G4),5);
Resolution of length 5 in characteristic 0 for <pc group with 3125 generators> .
gap> CR_CocyclesAndCoboundaries(RG4,4,true).torsionCoefficients; # H^4(G,Z)
[ [ 5, 5, 5, 5, 5 ] ]
gap> H4pFromResolution(RG4);
14[ [ 125, 5 ], [ 125, 2 ], [ 125, 2 ], [ 125, 2 ], [ 125, 2 ], [ 125, 5 ],
    [ 125, 2 ], [ 625, 3 ], [ 625, 13 ], [ 625, 12 ], [ 625, 3 ], [ 625, 3 ],
    [ 625, 3 ], [ 3125, 4 ] ]
1/14
...
14/14
[ [ 5, 5, 5, 5 ],
[ [ 5, 5, 5, 5, 5 ],
  [ [ 0, 1, 0, 0, 0 ], [ 0, 0, 1, 0, 0 ],
    [ 0, 0, 0, 1, 0 ], [ 0, 0, 0, 0, 1 ] ] ] ]
gap> IsUnramifiedH3(RG4,[1,0,0,0,0]:Subgroup);
1/12
[ [ 5 ], [ 0 ] ]
...
12/12
[ [ 5, 5, 5 ], [ 0, 0, 0 ] ]
true
\end{verbatim}
\medskip
{\bf Case 4c: $G=G(5^5,5)$ which belongs to $\Phi_6$.}\\
We have
$H^4(G,\bZ)\simeq(\bZ/5\bZ)^{\oplus 5}=\langle f_1,\ldots,f_5\rangle$,
$H^4_{\rm p}(G,\bZ)\simeq (\bZ/5\bZ)^{\oplus 4}=\langle f_1f_5^4,f_2f_5^3,f_3f_5^4,f_4\rangle$ and $f_5\in H^4_{\rm nr}(G,\bZ)$.
Hence we get $H^3_{\rm nr}(\bC(G),\bQ/\bZ)\simeq H^4_{\rm nr}(\bC(G),\bZ)\simeq \bZ/5\bZ$.
\medskip\begin{verbatim}
gap> G5:=SmallGroup(5^5,5);
<pc group of size 3125 with 5 generators>
gap> RG5:=ResolutionNormalSeries(LowerCentralSeries(G5),5);
Resolution of length 5 in characteristic 0 for <pc group with 3125 generators> .
gap> CR_CocyclesAndCoboundaries(RG5,4,true).torsionCoefficients; # H^4(G,Z)
[ [ 5, 5, 5, 5, 5 ] ]
gap> H4pFromResolution(RG5);
14[ [ 125, 5 ], [ 125, 2 ], [ 125, 5 ], [ 125, 2 ], [ 125, 2 ],
    [ 125, 2 ], [ 125, 2 ], [ 625, 12 ], [ 625, 3 ], [ 625, 3 ],
    [ 625, 3 ], [ 625, 3 ], [ 625, 3 ], [ 3125, 5 ] ]
1/14
...
14/14
[ [ 5, 5, 5, 5 ],
[ [ 5, 5, 5, 5, 5 ],
  [ [ 1, 0, 0, 0, 4 ], [ 0, 1, 0, 0, 3 ],
    [ 0, 0, 1, 0, 4 ], [ 0, 0, 0, 1, 0 ] ] ] ]
gap> IsUnramifiedH3(RG5,[0,0,0,0,1]:Subgroup);
1/12
[ [ 5 ], [ 0 ] ]
...
12/12
[ [ 5, 5, 5 ], [ 0, 0, 0 ] ]
true
\end{verbatim}
\medskip
{\bf Case 4d: $G=G(5^5,6)$ which belongs to $\Phi_6$.}\\
We have
$H^4(G,\bZ)\simeq(\bZ/5\bZ)^{\oplus 5}=\langle f_1,\ldots,f_5\rangle$,
$H^4_{\rm p}(G,\bZ)\simeq (\bZ/5\bZ)^{\oplus 4}=\langle f_1,f_2f_3^4,f_4,f_5\rangle$ and $f_3\in H^4_{\rm nr}(G,\bZ)$.
Hence we get $H^3_{\rm nr}(\bC(G),\bQ/\bZ)\simeq H^4_{\rm nr}(\bC(G),\bZ)\simeq \bZ/5\bZ$.
\medskip\begin{verbatim}
gap> G6:=SmallGroup(5^5,6);
<pc group of size 3125 with 5 generators>
gap> RG6:=ResolutionNormalSeries(LowerCentralSeries(G6),5);
Resolution of length 5 in characteristic 0 for <pc group with 3125 generators> .
gap> CR_CocyclesAndCoboundaries(RG6,4,true).torsionCoefficients; # H^4(G,Z)
[ [ 5, 5, 5, 5, 5 ] ]
gap> H4pFromResolution(RG6);
14[ [ 125, 5 ], [ 125, 2 ], [ 125, 5 ], [ 125, 2 ], [ 125, 2 ],
    [ 125, 2 ], [ 125, 2 ], [ 625, 12 ], [ 625, 3 ], [ 625, 3 ],
    [ 625, 3 ], [ 625, 3 ], [ 625, 3 ], [ 3125, 6 ] ]
1/14
...
14/14
[ [ 5, 5, 5, 5 ],
[ [ 5, 5, 5, 5, 5 ],
  [ [ 1, 0, 0, 0, 0 ], [ 0, 1, 4, 0, 0 ],
    [ 0, 0, 0, 1, 0 ], [ 0, 0, 0, 0, 1 ] ] ] ]
gap> IsUnramifiedH3(RG6,[0,0,1,0,0]:Subgroup);
1/12
[ [ 5 ], [ 0 ] ]
...
12/12
[ [ 5, 5, 5 ], [ 0, 0, 0 ] ]
true
\end{verbatim}
\medskip
{\bf Case 4e: $G=G(5^5,7)$ which belongs to $\Phi_6$.}\\
We have
$H^4(G,\bZ)\simeq(\bZ/5\bZ)^{\oplus 3}\oplus\bZ/5\bZ=\langle f_1,\ldots,f_4\rangle$,
$H^4_{\rm p}(G,\bZ)\simeq (\bZ/5\bZ)^{\oplus 2}\oplus\bZ/5\bZ=\langle f_1f_4^4,f_2f_4^4,f_3f_4,f_4^5\rangle$ and $f_4\in H^4_{\rm nr}(G,\bZ)$.
Hence we get $H^3_{\rm nr}(\bC(G),\bQ/\bZ)\simeq H^4_{\rm nr}(\bC(G),\bZ)\simeq \bZ/5\bZ$.
\medskip\begin{verbatim}
gap> G7:=SmallGroup(5^5,7);
<pc group of size 3125 with 5 generators>
gap> RG7:=ResolutionNormalSeries(LowerCentralSeries(G7),5);
Resolution of length 5 in characteristic 0 for <pc group with 3125 generators> .
gap> CR_CocyclesAndCoboundaries(RG7,4,true).torsionCoefficients; # H^4(G,Z)
[ [ 5, 5, 5, 25 ] ]
gap> H4pFromResolution(RG7);
14[ [ 125, 2 ], [ 125, 5 ], [ 125, 2 ], [ 125, 2 ], [ 125, 2 ],
    [ 125, 2 ], [ 125, 2 ], [ 625, 13 ], [ 625, 3 ], [ 625, 13 ],
    [ 625, 3 ], [ 625, 3 ], [ 625, 3 ], [ 3125, 7 ] ]
1/14
...
14/14
[ [ 5, 5, 25 ],
[ [ 5, 5, 5, 25 ],
  [ [ 1, 0, 0, 4 ], [ 0, 1, 0, 4 ], [ 0, 0, 1, 1 ], [ 0, 0, 0, 5 ] ] ] ]
gap> IsUnramifiedH3(RG7,[0,0,0,1]:Subgroup);
1/7
[ [ 5 ], [ 0 ] ]
...
7/7
[ [ 5, 5, 5 ], [ 0, 0, 0 ] ]
true
\end{verbatim}
\medskip
{\bf Case 4f: $G=G(5^5,8)$ which belongs to $\Phi_6$.}\\
We have
$H^4(G,\bZ)\simeq(\bZ/5\bZ)^{\oplus 4}=\langle f_1,\ldots,f_4\rangle$,
$H^4_{\rm p}(G,\bZ)\simeq (\bZ/5\bZ)^{\oplus 3}=\langle f_1f_4^2,f_2f_4^3,f_3f_4^4\rangle$ and $f_4\in H^4_{\rm nr}(G,\bZ)$.
Hence we get $H^3_{\rm nr}(\bC(G),\bQ/\bZ)\simeq H^4_{\rm nr}(\bC(G),\bZ)\simeq \bZ/5\bZ$.
\medskip\begin{verbatim}
gap> G8:=SmallGroup(5^5,8);
<pc group of size 3125 with 5 generators>
gap> RG8:=ResolutionNormalSeries(LowerCentralSeries(G8),5);
Resolution of length 5 in characteristic 0 for <pc group with 3125 generators> .
gap> CR_CocyclesAndCoboundaries(RG8,4,true).torsionCoefficients; # H^4(G,Z)
[ [ 5, 5, 5, 5 ] ]
gap> H4pFromResolution(RG8);
14[ [ 125, 2 ], [ 125, 2 ], [ 125, 2 ], [ 125, 2 ], [ 125, 2 ],
    [ 125, 2 ], [ 125, 5 ], [ 625, 3 ], [ 625, 13 ], [ 625, 3 ],
    [ 625, 3 ], [ 625, 3 ], [ 625, 3 ], [ 3125, 8 ] ]
1/14
...
14/14
[ [ 5, 5, 5 ],
[ [ 5, 5, 5, 5 ],
  [ [ 1, 0, 0, 2 ], [ 0, 1, 0, 3 ], [ 0, 0, 1, 4 ] ] ] ]
gap> IsUnramifiedH3(RG8,[0,0,0,1]:Subgroup);
1/1
[ [ 5, 5, 5 ], [ 0, 0, 0 ] ]
true
\end{verbatim}
\medskip
{\bf Case 4g: $G=G(5^5,9)$ which belongs to $\Phi_6$.}\\
We have
$H^4(G,\bZ)\simeq(\bZ/5\bZ)^{\oplus 4}=\langle f_1,\ldots,f_4\rangle$,
$H^4_{\rm p}(G,\bZ)\simeq (\bZ/5\bZ)^{\oplus 3}=\langle f_1f_4^2,f_2f_4^3,f_3f_4^4\rangle$ and $f_4\in H^4_{\rm nr}(G,\bZ)$.
Hence we get $H^3_{\rm nr}(\bC(G),\bQ/\bZ)\simeq H^4_{\rm nr}(\bC(G),\bZ)\simeq \bZ/5\bZ$.
\medskip\begin{verbatim}
gap> G9:=SmallGroup(5^5,9);
<pc group of size 3125 with 5 generators>
gap> RG9:=ResolutionNormalSeries(LowerCentralSeries(G9),5);
Resolution of length 5 in characteristic 0 for <pc group with 3125 generators> .
gap> CR_CocyclesAndCoboundaries(RG9,4,true).torsionCoefficients; # H^4(G,Z)
[ [ 5, 5, 5, 5 ] ]
gap> H4pFromResolution(RG9);
14[ [ 125, 5 ], [ 125, 2 ], [ 125, 2 ], [ 125, 2 ], [ 125, 2 ],
    [ 125, 2 ], [ 125, 2 ], [ 625, 3 ], [ 625, 3 ], [ 625, 3 ],
    [ 625, 3 ], [ 625, 3 ], [ 625, 3 ], [ 3125, 9 ] ]
1/14
...
14/14
[ [ 5, 5, 5 ],
[ [ 5, 5, 5, 5 ],
  [ [ 1, 0, 0, 2 ], [ 0, 1, 0, 3 ], [ 0, 0, 1, 4 ] ] ] ]
gap> IsUnramifiedH3(RG9,[0,0,0,1]:Subgroup);
1/1
[ [ 5, 5, 5 ], [ 0, 0, 0 ] ]
true
\end{verbatim}
\medskip
{\bf Case 4h: $G=G(5^5,10)$ which belongs to $\Phi_6$.}\\
We have
$H^4(G,\bZ)\simeq(\bZ/5\bZ)^{\oplus 3}\oplus\bZ/5\bZ=\langle f_1,\ldots,f_4\rangle$,
$H^4_{\rm p}(G,\bZ)\simeq (\bZ/5\bZ)^{\oplus 2}\oplus\bZ/5\bZ=\langle f_1,f_2f_4,f_3,f_4^5\rangle$ and $f_4\in H^4_{\rm nr}(G,\bZ)$.
Hence we get $H^3_{\rm nr}(\bC(G),\bQ/\bZ)\simeq H^4_{\rm nr}(\bC(G),\bZ)\simeq \bZ/5\bZ$.
\medskip\begin{verbatim}
gap> G10:=SmallGroup(5^5,10);
<pc group of size 3125 with 5 generators>
gap> RG10:=ResolutionNormalSeries(LowerCentralSeries(G10),5);
Resolution of length 5 in characteristic 0 for <pc group with 3125 generators> .
gap> CR_CocyclesAndCoboundaries(RG10,4,true).torsionCoefficients; # H^4(G,Z)
[ [ 5, 5, 5, 25 ] ]
gap> H4pFromResolution(RG10);
14[ [ 125, 5 ], [ 125, 2 ], [ 125, 2 ], [ 125, 2 ], [ 125, 2 ],
    [ 125, 2 ], [ 125, 2 ], [ 625, 3 ], [ 625, 3 ], [ 625, 3 ],
    [ 625, 3 ], [ 625, 3 ], [ 625, 3 ], [ 3125, 10 ] ]
1/14
...
14/14
[ [ 5, 5, 25 ],
[ [ 5, 5, 5, 25 ],
  [ [ 1, 0, 0, 0 ], [ 0, 1, 0, 1 ], [ 0, 0, 1, 0 ], [ 0, 0, 0, 5 ] ] ] ]
gap> IsUnramifiedH3(RG10,[0,0,0,1]:Subgroup);
1/7
[ [ 5 ], [ 0 ] ]
...
7/7
[ [ 5, 5, 5 ], [ 0, 0, 0 ] ]
true
\end{verbatim}
\medskip
{\bf Case 4i: $G=G(5^5,11)$ which belongs to $\Phi_6$.}\\
We have
$H^4(G,\bZ)\simeq(\bZ/5\bZ)^{\oplus 4}=\langle f_1,\ldots,f_4\rangle$,
$H^4_{\rm p}(G,\bZ)\simeq (\bZ/5\bZ)^{\oplus 3}=\langle f_1f_4^2,f_2f_4^3,f_3f_4^4\rangle$ and $f_4\in H^4_{\rm nr}(G,\bZ)$.
Hence we get $H^3_{\rm nr}(\bC(G),\bQ/\bZ)\simeq H^4_{\rm nr}(\bC(G),\bZ)\simeq \bZ/5\bZ$.
\medskip\begin{verbatim}
gap> G11:=SmallGroup(5^5,11);
<pc group of size 3125 with 5 generators>
gap> RG11:=ResolutionNormalSeries(LowerCentralSeries(G11),5);
Resolution of length 5 in characteristic 0 for <pc group with 3125 generators> .
gap> CR_CocyclesAndCoboundaries(RG11,4,true).torsionCoefficients; # H^4(G,Z)
[ [ 5, 5, 5, 5 ] ]
gap> H4pFromResolution(RG11);
14[ [ 125, 2 ], [ 125, 2 ], [ 125, 5 ], [ 125, 2 ], [ 125, 2 ],
    [ 125, 2 ], [ 125, 2 ], [ 625, 13 ], [ 625, 13 ], [ 625, 3 ],
    [ 625, 3 ], [ 625, 3 ], [ 625, 3 ], [ 3125, 11 ] ]
1/14
...
14/14
[ [ 5, 5, 5 ],
[ [ 5, 5, 5, 5 ],
  [ [ 1, 0, 0, 2 ], [ 0, 1, 0, 3 ], [ 0, 0, 1, 4 ] ] ] ]
gap> IsUnramifiedH3(RG11,[0,0,0,1]:Subgroup);
1/1
[ [ 5, 5, 5 ], [ 0, 0, 0 ] ]
true
\end{verbatim}
\medskip
{\bf Case 4j: $G=G(5^5,12)$ which belongs to $\Phi_6$.}\\
We have
$H^4(G,\bZ)\simeq(\bZ/5\bZ)^{\oplus 4}=\langle f_1,\ldots,f_4\rangle$,
$H^4_{\rm p}(G,\bZ)\simeq (\bZ/5\bZ)^{\oplus 3}=\langle f_1f_4^2,f_2f_4^3,f_3f_4^4\rangle$ and $f_4\in H^4_{\rm nr}(G,\bZ)$.
Hence we get $H^3_{\rm nr}(\bC(G),\bQ/\bZ)\simeq H^4_{\rm nr}(\bC(G),\bZ)\simeq \bZ/5\bZ$.
\medskip\begin{verbatim}
gap> G12:=SmallGroup(5^5,12);
<pc group of size 3125 with 5 generators>
gap> RG12:=ResolutionNormalSeries(LowerCentralSeries(G12),5);
Resolution of length 5 in characteristic 0 for <pc group with 3125 generators> .
gap> CR_CocyclesAndCoboundaries(RG12,4,true).torsionCoefficients; # H^4(G,Z)
[ [ 5, 5, 5, 5 ] ]
gap> H4pFromResolution(RG12);
14[ [ 125, 5 ], [ 125, 2 ], [ 125, 2 ], [ 125, 2 ], [ 125, 2 ],
    [ 125, 2 ], [ 125, 2 ], [ 625, 3 ], [ 625, 3 ], [ 625, 3 ],
    [ 625, 3 ], [ 625, 3 ], [ 625, 3 ], [ 3125, 12 ] ]
1/14
...
14/14
[ [ 5, 5, 5 ],
[ [ 5, 5, 5, 5 ],
  [ [ 1, 0, 0, 2 ], [ 0, 1, 0, 3 ], [ 0, 0, 1, 4 ] ] ] ]
gap> IsUnramifiedH3(RG12,[0,0,0,1]:Subgroup);
1/1
[ [ 5, 5, 5 ], [ 0, 0, 0 ] ]
true
\end{verbatim}
\medskip
{\bf Case 4k: $G=G(5^5,13)$ which belongs to $\Phi_6$.}\\
We have
$H^4(G,\bZ)\simeq(\bZ/5\bZ)^{\oplus 4}=\langle f_1,\ldots,f_4\rangle$,
$H^4_{\rm p}(G,\bZ)\simeq (\bZ/5\bZ)^{\oplus 3}=\langle f_1f_4^2,f_2f_4^3,f_3f_4^4\rangle$ and $f_4\in H^4_{\rm nr}(G,\bZ)$.
Hence we get $H^3_{\rm nr}(\bC(G),\bQ/\bZ)\simeq H^4_{\rm nr}(\bC(G),\bZ)\simeq \bZ/5\bZ$.
\medskip\begin{verbatim}
gap> G13:=SmallGroup(5^5,13);
<pc group of size 3125 with 5 generators>
gap> RG13:=ResolutionNormalSeries(LowerCentralSeries(G13),5);
Resolution of length 5 in characteristic 0 for <pc group with 3125 generators> .
gap> CR_CocyclesAndCoboundaries(RG13,4,true).torsionCoefficients; # H^4(G,Z)
[ [ 5, 5, 5, 5 ] ]
gap> H4pFromResolution(RG13);
14[ [ 125, 2 ], [ 125, 2 ], [ 125, 2 ], [ 125, 2 ], [ 125, 2 ],
    [ 125, 2 ], [ 125, 5 ], [ 625, 3 ], [ 625, 13 ], [ 625, 3 ],
    [ 625, 3 ], [ 625, 3 ], [ 625, 3 ], [ 3125, 13 ] ]
1/14
...
14/14
[ [ 5, 5, 5 ],
[ [ 5, 5, 5, 5 ],
  [ [ 1, 0, 0, 2 ], [ 0, 1, 0, 3 ], [ 0, 0, 1, 4 ] ] ] ]
gap> IsUnramifiedH3(RG13,[0,0,0,1]:Subgroup);
1/1
[ [ 5, 5, 5 ], [ 0, 0, 0 ] ]
true
\end{verbatim}
\medskip
{\bf Case 4l: $G=G(5^5,14)$ which belongs to $\Phi_6$.}\\
We have
$H^4(G,\bZ)\simeq(\bZ/5\bZ)^{\oplus 4}=\langle f_1,\ldots,f_4\rangle$,
$H^4_{\rm p}(G,\bZ)\simeq (\bZ/5\bZ)^{\oplus 3}=\langle f_1,f_2,f_3\rangle$ and $f_4\in H^4_{\rm nr}(G,\bZ)$.
Hence we get $H^3_{\rm nr}(\bC(G),\bQ/\bZ)\simeq H^4_{\rm nr}(\bC(G),\bZ)\simeq \bZ/5\bZ$.
\medskip\begin{verbatim}
gap> G14:=SmallGroup(5^5,14);
<pc group of size 3125 with 5 generators>
gap> RG14:=ResolutionNormalSeries(LowerCentralSeries(G14),5);
Resolution of length 5 in characteristic 0 for <pc group with 3125 generators> .
gap> CR_CocyclesAndCoboundaries(RG14,4,true).torsionCoefficients; # H^4(G,Z)
[ [ 5, 5, 5, 5 ] ]
gap> H4pFromResolution(RG14);
14[ [ 125, 2 ], [ 125, 2 ], [ 125, 2 ], [ 125, 2 ], [ 125, 5 ],
    [ 125, 2 ], [ 125, 2 ], [ 625, 13 ], [ 625, 13 ], [ 625, 13 ],
    [ 625, 13 ], [ 625, 13 ], [ 625, 13 ], [ 3125, 14 ] ]
1/14
...
14/14
[ [ 5, 5, 5 ],
[ [ 5, 5, 5, 5 ],
  [ [ 1, 0, 0, 0 ], [ 0, 1, 0, 0 ], [ 0, 0, 1, 0 ] ] ] ]
gap> IsUnramifiedH3(RG14,[0,0,0,1]:Subgroup);
1/1
[ [ 5, 5, 5 ], [ 0, 0, 0 ] ]
true
\end{verbatim}
\medskip
{\bf Case 5a: $G=G(5^5,66)$ which belongs to $\Phi_7$.}\\
We have
$H^4(G,\bZ)\simeq(\bZ/5\bZ)^{\oplus 10}=\langle f_1,\ldots,f_{10}\rangle$,
$H^4_{\rm p}(G,\bZ)\simeq (\bZ/5\bZ)^{\oplus 6}=\langle f_2,f_3f_9^4,f_4,f_5f_9,f_6f_9f_{10},f_8f_9\rangle$, $f_1,f_7,f_{10}\not\in H^4_{\rm nr}(G\bZ)$
and $f_8\in H^4_{\rm nr}(G,\bZ)$.
Hence we get $H^3_{\rm nr}(\bC(G),\bQ/\bZ)\simeq H^4_{\rm nr}(\bC(G),\bZ)\simeq \bZ/5\bZ$.
\medskip\begin{verbatim}
gap> G66:=SmallGroup(5^5,66);
<pc group of size 3125 with 5 generators>
gap> RG66:=ResolutionNormalSeries(LowerCentralSeries(G66),5);
Resolution of length 5 in characteristic 0 for <pc group with 3125 generators> .
gap> CR_CocyclesAndCoboundaries(RG66,4,true).torsionCoefficients; # H^4(G,Z)
[ [ 5, 5, 5, 5, 5, 5, 5, 5, 5, 5 ] ]
gap> H4pFromResolution(RG66);
18[ [ 125, 5 ], [ 125, 5 ], [ 125, 5 ], [ 125, 5 ], [ 125, 5 ],
    [ 125, 5 ], [ 125, 5 ], [ 125, 5 ], [ 125, 5 ], [ 125, 5 ],
    [ 125, 5 ], [ 625, 12 ], [ 625, 12 ], [ 625, 12 ], [ 625, 12 ],
    [ 625, 12 ], [ 625, 12 ], [ 3125, 66 ] ]
1/18
...
18/18
[ [ 5, 5, 5, 5, 5, 5 ],
[ [ 5, 5, 5, 5, 5, 5, 5, 5, 5, 5 ],
  [ [ 0, 1, 0, 0, 0, 0, 0, 0, 0, 0 ], [ 0, 0, 1, 0, 0, 0, 0, 0, 4, 0 ],
    [ 0, 0, 0, 1, 0, 0, 0, 0, 0, 0 ], [ 0, 0, 0, 0, 1, 0, 0, 0, 1, 0 ],
    [ 0, 0, 0, 0, 0, 1, 0, 0, 1, 1 ], [ 0, 0, 0, 0, 0, 0, 0, 1, 1, 0 ] ] ] ]
gap> IsUnramifiedH3(RG66,[0,0,0,0,0,0,0,1,0,0]:Subgroup);
1/37
[ [ 5, 5, 5, 5 ], [ 0, 0, 0, 0 ] ]
...
37/37
[ [ 5, 5, 5 ], [ 0, 0, 0 ] ]
true
gap> IsUnramifiedH3(RG66,[1,0,0,0,0,0,0,0,0,0]:Subgroup);
1/37
[ [ 5, 5, 5, 5 ], [ 0, 0, 1, 1 ] ]
false
gap> IsUnramifiedH3(RG66,[0,0,0,0,0,0,1,0,0,0]:Subgroup);
1/37
[ [ 5, 5, 5, 5 ], [ 0, 4, 4, 0 ] ]
false
gap> IsUnramifiedH3(RG66,[0,0,0,0,0,0,0,0,0,1]:Subgroup);
1/37
[ [ 5, 5, 5, 5 ], [ 0, 0, 3, 0 ] ]
false
\end{verbatim}
\medskip
{\bf Case 5b: $G=G(5^5,67)$ which belongs to $\Phi_7$.}\\
We have
$H^4(G,\bZ)\simeq(\bZ/5\bZ)^{\oplus 6}=\langle f_1,\ldots,f_6\rangle$,
$H^4_{\rm p}(G,\bZ)\simeq (\bZ/5\bZ)^{\oplus 5}=\langle f_1,f_2,f_3,f_4,f_5f_6\rangle$, $f_1,f_7,f_{10}\not\in H^4_{\rm nr}(G\bZ)$
and $f_6\in H^4_{\rm nr}(G,\bZ)$.
Hence we get $H^3_{\rm nr}(\bC(G),\bQ/\bZ)\simeq H^4_{\rm nr}(\bC(G),\bZ)\simeq \bZ/5\bZ$.
\medskip\begin{verbatim}
gap> G67:=SmallGroup(5^5,67);
<pc group of size 3125 with 5 generators>
gap> RG67:=ResolutionNormalSeries(LowerCentralSeries(G67),5);
Resolution of length 5 in characteristic 0 for <pc group with 3125 generators> .
gap> CR_CocyclesAndCoboundaries(RG67,4,true).torsionCoefficients; # H^4(G,Z)
[ [ 5, 5, 5, 5, 5, 5 ] ]
gap> H4pFromResolution(RG67);
18[ [ 125, 5 ], [ 125, 5 ], [ 125, 5 ], [ 125, 5 ], [ 125, 5 ],
    [ 125, 5 ], [ 125, 2 ], [ 125, 2 ], [ 125, 2 ], [ 125, 2 ],
    [ 125, 2 ], [ 625, 13 ], [ 625, 12 ], [ 625, 13 ], [ 625, 13 ],
    [ 625, 13 ], [ 625, 13 ], [ 3125, 67 ] ]
1/18
...
18/18
[ [ 5, 5, 5, 5, 5 ],
[ [ 5, 5, 5, 5, 5, 5 ],
  [ [ 1, 0, 0, 0, 0, 0 ], [ 0, 1, 0, 0, 0, 0 ], [ 0, 0, 1, 0, 0, 0 ],
    [ 0, 0, 0, 1, 0, 0 ], [ 0, 0, 0, 0, 1, 1 ] ] ] ]
gap> IsUnramifiedH3(RG67,[0,0,0,0,0,1]:Subgroup);
1/7
[ [ 5, 5, 5 ], [ 0, 0, 0 ] ]
...
7/7
[ [ 5, 5, 5 ], [ 0, 0, 0 ] ]
true
\end{verbatim}
\medskip
{\bf Case 5c: $G=G(5^5,68)$ which belongs to $\Phi_7$.}\\
We have
$H^4(G,\bZ)\simeq(\bZ/5\bZ)^{\oplus 6}=\langle f_1,\ldots,f_6\rangle$,
$H^4_{\rm p}(G,\bZ)\simeq (\bZ/5\bZ)^{\oplus 5}=\langle f_1,f_2,f_3,f_4,f_5f_6^2\rangle$, $f_1,f_7,f_{10}\not\in H^4_{\rm nr}(G\bZ)$
and $f_6\in H^4_{\rm nr}(G,\bZ)$.
Hence we get $H^3_{\rm nr}(\bC(G),\bQ/\bZ)\simeq H^4_{\rm nr}(\bC(G),\bZ)\simeq \bZ/5\bZ$.
\medskip\begin{verbatim}
gap> G68:=SmallGroup(5^5,68);
<pc group of size 3125 with 5 generators>
gap> RG68:=ResolutionNormalSeries(LowerCentralSeries(G68),5);
Resolution of length 5 in characteristic 0 for <pc group with 3125 generators> .
gap> CR_CocyclesAndCoboundaries(RG68,4,true).torsionCoefficients; # H^4(G,Z)
[ [ 5, 5, 5, 5, 5, 5 ] ]
gap> H4pFromResolution(RG68);
18[ [ 125, 2 ], [ 125, 5 ], [ 125, 2 ], [ 125, 2 ], [ 125, 2 ],
    [ 125, 2 ], [ 125, 2 ], [ 125, 2 ], [ 125, 5 ], [ 125, 2 ],
    [ 125, 2 ], [ 625, 12 ], [ 625, 13 ], [ 625, 13 ], [ 625, 13 ],
    [ 625, 13 ], [ 625, 13 ], [ 3125, 68 ] ]
1/18
...
18/18
[ [ 5, 5, 5, 5, 5 ],
[ [ 5, 5, 5, 5, 5, 5 ],
  [ [ 1, 0, 0, 0, 0, 0 ], [ 0, 1, 0, 0, 0, 0 ], [ 0, 0, 1, 0, 0, 0 ],
    [ 0, 0, 0, 1, 0, 0 ], [ 0, 0, 0, 0, 1, 2 ] ] ] ]
gap> IsUnramifiedH3(RG68,[0,0,0,0,0,1]:Subgroup);
1/7
[ [ 5, 5, 5 ], [ 0, 0, 0 ] ]
...
7/7
[ [ 5, 5, 5, 5 ], [ 0, 0, 0, 0 ] ]
true
\end{verbatim}
\medskip
{\bf Case 5d: $G=G(5^5,69)$ which belongs to $\Phi_7$.}\\
We have
$H^4(G,\bZ)\simeq(\bZ/5\bZ)^{\oplus 6}=\langle f_1,\ldots,f_6\rangle$,
$H^4_{\rm p}(G,\bZ)\simeq (\bZ/5\bZ)^{\oplus 5}=\langle f_1,f_2,f_3,f_4,f_5f_6^2\rangle$, $f_1,f_7,f_{10}\not\in H^4_{\rm nr}(G\bZ)$
and $f_6\in H^4_{\rm nr}(G,\bZ)$.
Hence we get $H^3_{\rm nr}(\bC(G),\bQ/\bZ)\simeq H^4_{\rm nr}(\bC(G),\bZ)\simeq \bZ/5\bZ$.
\medskip\begin{verbatim}
gap> G69:=SmallGroup(5^5,69);
<pc group of size 3125 with 5 generators>
gap> RG69:=ResolutionNormalSeries(LowerCentralSeries(G69),5);
Resolution of length 5 in characteristic 0 for <pc group with 3125 generators> .
gap> CR_CocyclesAndCoboundaries(RG69,4,true).torsionCoefficients; # H^4(G,Z)
[ [ 5, 5, 5, 5, 5, 5 ] ]
gap> H4pFromResolution(RG69);
18[ [ 125, 2 ], [ 125, 5 ], [ 125, 2 ], [ 125, 2 ], [ 125, 2 ],
    [ 125, 2 ], [ 125, 2 ], [ 125, 2 ], [ 125, 5 ], [ 125, 2 ],
    [ 125, 2 ], [ 625, 12 ], [ 625, 13 ], [ 625, 13 ], [ 625, 13 ],
    [ 625, 13 ], [ 625, 13 ], [ 3125, 69 ] ]
1/18
...
18/18
[ [ 5, 5, 5, 5, 5 ],
[ [ 5, 5, 5, 5, 5, 5 ],
  [ [ 1, 0, 0, 0, 0, 0 ], [ 0, 1, 0, 0, 0, 0 ], [ 0, 0, 1, 0, 0, 0 ],
    [ 0, 0, 0, 1, 0, 0 ], [ 0, 0, 0, 0, 1, 2 ] ] ] ]
gap> IsUnramifiedH3(RG69,[0,0,0,0,0,1]:Subgroup);
1/7
[ [ 5, 5, 5 ], [ 0, 0, 0 ] ]
...
7/7
[ [ 5, 5, 5, 5 ], [ 0, 0, 0, 0 ] ]
true
\end{verbatim}
\medskip
{\bf Case 5e: $G=G(5^5,70)$ which belongs to $\Phi_7$.}\\
We have
$H^4(G,\bZ)\simeq(\bZ/5\bZ)^{\oplus 6}=\langle f_1,\ldots,f_6\rangle$,
$H^4_{\rm p}(G,\bZ)\simeq (\bZ/5\bZ)^{\oplus 5}=\langle f_1,f_2,f_3,f_4,f_5f_6\rangle$, $f_1,f_7,f_{10}\not\in H^4_{\rm nr}(G\bZ)$
and $f_6\in H^4_{\rm nr}(G,\bZ)$.
Hence we get $H^3_{\rm nr}(\bC(G),\bQ/\bZ)\simeq H^4_{\rm nr}(\bC(G),\bZ)\simeq \bZ/5\bZ$.
\medskip\begin{verbatim}
gap> G70:=SmallGroup(5^5,70);
<pc group of size 3125 with 5 generators>
gap> RG70:=ResolutionNormalSeries(LowerCentralSeries(G70),5);
Resolution of length 5 in characteristic 0 for <pc group with 3125 generators> .
gap> CR_CocyclesAndCoboundaries(RG70,4,true).torsionCoefficients; # H^4(G,Z)
[ [ 5, 5, 5, 5, 5, 5 ] ]
gap> H4pFromResolution(RG70);
18[ [ 125, 5 ], [ 125, 2 ], [ 125, 2 ], [ 125, 2 ], [ 125, 2 ],
    [ 125, 2 ], [ 125, 2 ], [ 125, 2 ], [ 125, 2 ], [ 125, 2 ],
    [ 125, 2 ], [ 625, 14 ], [ 625, 13 ], [ 625, 14 ], [ 625, 14 ],
    [ 625, 14 ], [ 625, 14 ], [ 3125, 70 ] ]
1/18
...
18/18
[ [ 5, 5, 5, 5, 5 ],
[ [ 5, 5, 5, 5, 5, 5 ],
  [ [ 1, 0, 0, 0, 0, 0 ], [ 0, 1, 0, 0, 0, 0 ], [ 0, 0, 1, 0, 0, 0 ],
    [ 0, 0, 0, 1, 0, 0 ], [ 0, 0, 0, 0, 1, 1 ] ] ] ]
gap> IsUnramifiedH3(RG70,[0,0,0,0,1,0]:Subgroup);
1/2
[ [ 5, 5, 5 ], [ 0, 0, 0 ] ]
2/2
[ [ 5, 5, 5 ], [ 0, 0, 0 ] ]
true
\end{verbatim}
\medskip
{\bf Case 6a: $G=G(5^5,33)$ which belongs to $\Phi_{10}$.}\\
We have
$H^4(G,\bZ)\simeq(\bZ/5\bZ)^{\oplus 6}=\langle f_1,\ldots,f_6\rangle$,
$H^4_{\rm p}(G,\bZ)\simeq (\bZ/5\bZ)^{\oplus 4}=\langle f_1,f_2f_5^3f_6^4,f_3,f_4\rangle$, $f_6\not\in H^4_{\rm nr}(G\bZ)$
and $f_5\in H^4_{\rm nr}(G,\bZ)$.
Hence we get $H^3_{\rm nr}(\bC(G),\bQ/\bZ)\simeq H^4_{\rm nr}(\bC(G),\bZ)\simeq \bZ/5\bZ$.
\medskip\begin{verbatim}
gap> G33:=SmallGroup(5^5,33);
<pc group of size 3125 with 5 generators>
gap> RG33:=ResolutionNormalSeries(LowerCentralSeries(G33),5);
Resolution of length 5 in characteristic 0 for <pc group with 3125 generators> .
gap> CR_CocyclesAndCoboundaries(RG33,4,true).torsionCoefficients; # H^4(G,Z)
[ [ 5, 5, 5, 5, 5, 5 ] ]
gap> H4pFromResolution(RG33);
19[ [ 25, 2 ], [ 25, 2 ], [ 25, 2 ], [ 25, 2 ], [ 25, 2 ], [ 125, 5 ],
    [ 125, 3 ], [ 125, 3 ], [ 125, 3 ], [ 125, 5 ], [ 125, 3 ],
    [ 125, 3 ], [ 625, 7 ], [ 625, 7 ], [ 625, 7 ], [ 625, 7 ],
    [ 625, 7 ], [ 625, 12 ], [ 3125, 33 ] ]
1/19
...
19/19
[ [ 5, 5, 5, 5 ],
[ [ 5, 5, 5, 5, 5, 5 ],
  [ [ 1, 0, 0, 0, 0, 0 ], [ 0, 1, 0, 0, 3, 4 ],
    [ 0, 0, 1, 0, 0, 0 ], [ 0, 0, 0, 1, 0, 0 ] ] ] ]
gap> IsUnramifiedH3(RG33,[0,0,0,0,1,0]:Subgroup);
1/8
[ [ 5, 5, 5 ], [ 0, 0, 0 ] ]
...
8/8
[ [ 5, 5, 5 ], [ 0, 0, 0 ] ]
true
gap> IsUnramifiedH3(RG33,[0,0,0,0,0,1]:Subgroup);
1/8
[ [ 5, 5, 5 ], [ 0, 2, 0 ] ]
false
\end{verbatim}
\medskip
{\bf Case 6b: $G=G(5^5,34)$ which belongs to $\Phi_{10}$.}\\
We have
$H^4(G,\bZ)\simeq(\bZ/5\bZ)^{\oplus 4}=\langle f_1,\ldots,f_4\rangle$,
$H^4_{\rm p}(G,\bZ)\simeq (\bZ/5\bZ)^{\oplus 2}=\langle f_1f_4,f_2\rangle$, $f_4\not\in H^4_{\rm nr}(G\bZ)$
and $f_3\in H^4_{\rm nr}(G,\bZ)$.
Hence we get $H^3_{\rm nr}(\bC(G),\bQ/\bZ)\simeq H^4_{\rm nr}(\bC(G),\bZ)\simeq \bZ/5\bZ$.
\medskip\begin{verbatim}
gap> G34:=SmallGroup(5^5,34);
<pc group of size 3125 with 5 generators>
gap> RG34:=ResolutionNormalSeries(LowerCentralSeries(G34),5);
Resolution of length 5 in characteristic 0 for <pc group with 3125 generators> .
gap> CR_CocyclesAndCoboundaries(RG34,4,true).torsionCoefficients; # H^4(G,Z)
[ [ 5, 5, 5, 5 ] ]
gap> H4pFromResolution(RG34);
19[ [ 25, 1 ], [ 25, 1 ], [ 25, 1 ], [ 25, 1 ], [ 25, 1 ], [ 125, 5 ],
    [ 125, 4 ], [ 125, 4 ], [ 125, 4 ], [ 125, 4 ], [ 125, 5 ],
    [ 125, 4 ], [ 625, 8 ], [ 625, 8 ], [ 625, 12 ], [ 625, 8 ],
    [ 625, 8 ], [ 625, 8 ], [ 3125, 34 ] ]
1/19
...
19/19
[ [ 5, 5 ],
[ [ 5, 5, 5, 5 ],
  [ [ 1, 0, 0, 1 ], [ 0, 1, 0, 0 ] ] ] ]
gap> IsUnramifiedH3(RG34,[0,0,1,0]:Subgroup);
1/8
[ [ 5 ], [ 0 ] ]
...
8/8
[ [ 5, 5, 5 ], [ 0, 0, 0 ] ]
true
gap> IsUnramifiedH3(RG34,[0,0,0,1]:Subgroup);
1/8
[ [ 5 ], [ 3 ] ]
false
\end{verbatim}
\medskip
{\bf Case 6c: $G=G(5^5,35)$ which belongs to $\Phi_{10}$.}\\
We have
$H^4(G,\bZ)\simeq(\bZ/5\bZ)^{\oplus 4}=\langle f_1,\ldots,f_4\rangle$,
$H^4_{\rm p}(G,\bZ)\simeq (\bZ/5\bZ)^{\oplus 2}=\langle f_1f_3^2f_4^4,f_2\rangle$, $f_4\not\in H^4_{\rm nr}(G\bZ)$
and $f_3\in H^4_{\rm nr}(G,\bZ)$.
Hence we get $H^3_{\rm nr}(\bC(G),\bQ/\bZ)\simeq H^4_{\rm nr}(\bC(G),\bZ)\simeq \bZ/5\bZ$.
\medskip\begin{verbatim}
gap> G35:=SmallGroup(5^5,35);
<pc group of size 3125 with 5 generators>
gap> RG35:=ResolutionNormalSeries(LowerCentralSeries(G35),5);
Resolution of length 5 in characteristic 0 for <pc group with 3125 generators> .
gap> CR_CocyclesAndCoboundaries(RG35,4,true).torsionCoefficients; # H^4(G,Z)
[ [ 5, 5, 5, 5 ] ]
gap> H4pFromResolution(RG35);
19[ [ 25, 1 ], [ 25, 1 ], [ 25, 1 ], [ 25, 1 ], [ 25, 1 ], [ 125, 5 ],
    [ 125, 4 ], [ 125, 4 ], [ 125, 4 ], [ 125, 4 ], [ 125, 5 ],
    [ 125, 4 ], [ 625, 8 ], [ 625, 8 ], [ 625, 12 ], [ 625, 8 ],
    [ 625, 8 ], [ 625, 8 ], [ 3125, 35 ] ]
1/19
...
19/19
[ [ 5, 5 ],
[ [ 5, 5, 5, 5 ],
  [ [ 1, 0, 2, 4 ], [ 0, 1, 0, 0 ] ] ] ]
gap> IsUnramifiedH3(RG35,[0,0,1,0]:Subgroup);
1/8
[ [ 5 ], [ 0 ] ]
...
8/8
[ [ 5, 5, 5 ], [ 0, 0, 0 ] ]
true
gap> IsUnramifiedH3(RG35,[0,0,0,1]:Subgroup);
1/8
[ [ 5 ], [ 2 ] ]
false
\end{verbatim}
\medskip
{\bf Case 6d: $G=G(5^5,36)$ which belongs to $\Phi_{10}$.}\\
We have
$H^4(G,\bZ)\simeq(\bZ/5\bZ)^{\oplus 4}=\langle f_1,\ldots,f_4\rangle$,
$H^4_{\rm p}(G,\bZ)\simeq (\bZ/5\bZ)^{\oplus 2}=\langle f_1f_3^2f_4^4,f_2\rangle$, $f_4\not\in H^4_{\rm nr}(G\bZ)$
and $f_3\in H^4_{\rm nr}(G,\bZ)$.
Hence we get $H^3_{\rm nr}(\bC(G),\bQ/\bZ)\simeq H^4_{\rm nr}(\bC(G),\bZ)\simeq \bZ/5\bZ$.
\medskip\begin{verbatim}
gap> G36:=SmallGroup(5^5,36);
<pc group of size 3125 with 5 generators>
gap> RG36:=ResolutionNormalSeries(LowerCentralSeries(G36),5);
Resolution of length 5 in characteristic 0 for <pc group with 3125 generators> .
gap> CR_CocyclesAndCoboundaries(RG36,4,true).torsionCoefficients; # H^4(G,Z)
[ [ 5, 5, 5, 5 ] ]
gap> H4pFromResolution(RG36);
19[ [ 25, 1 ], [ 25, 1 ], [ 25, 1 ], [ 25, 1 ], [ 25, 1 ], [ 125, 5 ],
    [ 125, 4 ], [ 125, 4 ], [ 125, 4 ], [ 125, 4 ], [ 125, 5 ],
    [ 125, 4 ], [ 625, 8 ], [ 625, 8 ], [ 625, 12 ], [ 625, 8 ],
    [ 625, 8 ], [ 625, 8 ], [ 3125, 36 ] ]
1/19
...
19/19
[ [ 5, 5 ],
[ [ 5, 5, 5, 5 ],
  [ [ 1, 0, 2, 4 ], [ 0, 1, 0, 0 ] ] ] ]
gap> IsUnramifiedH3(RG36,[0,0,1,0]:Subgroup);
1/8
[ [ 5 ], [ 0 ] ]
...
8/8
[ [ 5, 5, 5 ], [ 0, 0, 0 ] ]
true
gap> IsUnramifiedH3(RG36,[0,0,0,1]:Subgroup);
1/8
[ [ 5 ], [ 2 ] ]
false
\end{verbatim}
\medskip
{\bf Case 6e: $G=G(5^5,37)$ which belongs to $\Phi_{10}$.}\\
We have
$H^4(G,\bZ)\simeq(\bZ/5\bZ)^{\oplus 4}=\langle f_1,\ldots,f_4\rangle$,
$H^4_{\rm p}(G,\bZ)\simeq (\bZ/5\bZ)^{\oplus 2}=\langle f_1f_3^2f_4^4,f_2\rangle$, $f_4\not\in H^4_{\rm nr}(G\bZ)$
and $f_3\in H^4_{\rm nr}(G,\bZ)$.
Hence we get $H^3_{\rm nr}(\bC(G),\bQ/\bZ)\simeq H^4_{\rm nr}(\bC(G),\bZ)\simeq \bZ/5\bZ$.
\medskip\begin{verbatim}
gap> G37:=SmallGroup(5^5,37);
<pc group of size 3125 with 5 generators>
gap> RG37:=ResolutionNormalSeries(LowerCentralSeries(G37),5);
Resolution of length 5 in characteristic 0 for <pc group with 3125 generators> .
gap> CR_CocyclesAndCoboundaries(RG37,4,true).torsionCoefficients; # H^4(G,Z)
[ [ 5, 5, 5, 5 ] ]
gap> H4pFromResolution(RG37);
19[ [ 25, 1 ], [ 25, 1 ], [ 25, 1 ], [ 25, 1 ], [ 25, 1 ], [ 125, 5 ],
    [ 125, 4 ], [ 125, 4 ], [ 125, 4 ], [ 125, 4 ], [ 125, 5 ],
    [ 125, 4 ], [ 625, 8 ], [ 625, 8 ], [ 625, 12 ], [ 625, 8 ],
    [ 625, 8 ], [ 625, 8 ], [ 3125, 37 ] ]
1/19
...
19/19
[ [ 5, 5 ],
[ [ 5, 5, 5, 5 ],
  [ [ 1, 0, 2, 4 ], [ 0, 1, 0, 0 ] ] ] ]
gap> IsUnramifiedH3(RG37,[0,0,1,0]:Subgroup);
1/8
[ [ 5 ], [ 0 ] ]
...
8/8
[ [ 5, 5, 5 ], [ 0, 0, 0 ] ]
true
gap> IsUnramifiedH3(RG37,[0,0,0,1]:Subgroup);
1/8
[ [ 5 ], [ 2 ] ]
false
\end{verbatim}
\medskip
{\bf Case 6f: $G=G(5^5,38)$ which belongs to $\Phi_{10}$.}\\
We have
$H^4(G,\bZ)\simeq(\bZ/5\bZ)^{\oplus 3}=\langle f_1,f_2,f_3\rangle$,
$H^4_{\rm p}(G,\bZ)\simeq (\bZ/5\bZ)^{\oplus 2}=\langle f_1f_3^4,f_2f_3^3\rangle$ and $f_3\in H^4_{\rm nr}(G,\bZ)$.
Hence we get $H^3_{\rm nr}(\bC(G),\bQ/\bZ)\simeq H^4_{\rm nr}(\bC(G),\bZ)\simeq \bZ/5\bZ$.
\medskip\begin{verbatim}
gap> G38:=SmallGroup(5^5,38);
<pc group of size 3125 with 5 generators>
gap> RG38:=ResolutionNormalSeries(LowerCentralSeries(G38),5);
Resolution of length 5 in characteristic 0 for <pc group with 3125 generators> .
gap> CR_CocyclesAndCoboundaries(RG38,4,true).torsionCoefficients; # H^4(G,Z)
[ [ 5, 5, 5 ] ]
gap> H4pFromResolution(RG38);
19[ [ 25, 2 ], [ 25, 1 ], [ 25, 1 ], [ 25, 1 ], [ 25, 1 ], [ 125, 5 ],
    [ 125, 3 ], [ 125, 4 ], [ 125, 2 ], [ 125, 4 ], [ 125, 4 ],
    [ 125, 4 ], [ 625, 7 ], [ 625, 13 ], [ 625, 8 ], [ 625, 8 ],
    [ 625, 8 ], [ 625, 8 ], [ 3125, 38 ] ]
1/19
...
19/19
[ [ 5, 5 ],
[ [ 5, 5, 5 ],
  [ [ 1, 0, 4 ], [ 0, 1, 3 ] ] ] ]
gap> IsUnramifiedH3(RG38,[0,0,1]:Subgroup);
1/2
[ [ 5 ], [ 0 ] ]
2/2
[ [ 5, 5, 5 ], [ 0, 0, 0 ] ]
true
\end{verbatim}

\newpage
A summary for $p=7$
\footnote{The computation for $p=7$ takes much more time
than the cases $p=3,5$. For examples, in the authors's personal computer and for $p=7$,
it takes about one day for $\Phi_5$ and $\Phi_7$ each,
and about 10 days for $\Phi_6$ and $\Phi_{10}$ each.
}:\\
\begin{table}[!h]\vspace*{-2mm}
\begin{tabular}{cl|ccc|c}
\multicolumn{2}{c|}{$|G|=7^5$} & $H^3(G,\bQ/\bZ)$ & $H^3_{\rm nr}(G,\bQ/\bZ)$
 & $H^3_{\rm p}(G,\bQ/\bZ)$ & $H^3_{\rm nr}(\bC(G),\bQ/\bZ)$\\\hline
$\Phi_5$ & $G(7^5,82)$ & $(\bZ/7\bZ)^{\oplus 10}$
& $(\bZ/7\bZ)^{\oplus 10}$ & $(\bZ/7\bZ)^{\oplus 10}$
& $0$\\\hline
$\Phi_6$ & $G(7^5,7)$ & $(\bZ/7\bZ)^{\oplus 4}$
& $(\bZ/7\bZ)^{\oplus 4}$ & $(\bZ/7\bZ)^{\oplus 3}$
& $\bZ/7\bZ$\\\hline
$\Phi_7$ & $G(7^5,73)$ & $(\bZ/7\bZ)^{\oplus 6}$
& $(\bZ/7\bZ)^{\oplus 6}$ & $(\bZ/7\bZ)^{\oplus 5}$
& $\bZ/7\bZ$\\\hline
$\Phi_{10}$ & $G(7^5,38)$ & $(\bZ/7\bZ)^{\oplus 3}$
& $(\bZ/7\bZ)^{\oplus 3}$ & $(\bZ/7\bZ)^{\oplus 2}$
& $\bZ/7\bZ$\\
\end{tabular}\\
\vspace*{2mm}
Table $5$: $H^3_{\rm nr}(\bC(G),\bQ/\bZ)$ for groups $G$ of order $7^5$
which belong to $\Phi_5$, $\Phi_6$, $\Phi_7$ or $\Phi_{10}$
\end{table}

\medskip
{\bf Case 7a: $G=G(7^5,82)$ which belongs to $\Phi_5$.}\\
We have
$H^4(G,\bZ)\simeq(\bZ/7\bZ)^{\oplus 10}=\langle f_1,\ldots,f_{10}\rangle$ and
$H^4_{\rm p}(G,\bZ)\simeq (\bZ/7\bZ)^{\oplus 10}=\langle f_1,\ldots,f_{10}\rangle$.
Hence we get $H^3_{\rm nr}(\bC(G),\bQ/\bZ)\simeq H^4_{\rm nr}(\bC(G),\bZ)=0$.
\begin{verbatim}
gap> Read("H3nr.gap");

gap> G82:=SmallGroup(7^5,82);
<pc group of size 16807 with 5 generators>
gap> RG82:=ResolutionNormalSeries(LowerCentralSeries(G82),5);
Resolution of length 5 in characteristic 0 for <pc group with 16807 generators> .
gap> CR_CocyclesAndCoboundaries(RG82,4,true).torsionCoefficients; # H^4(G,Z)
[ [ 7, 7, 7, 7, 7, 7, 7, 7, 7, 7 ] ]
gap> H4pFromResolution(RG82);
401[ [ 343, 5 ], [ 343, 2 ], [ 343, 2 ], ..., [ 343, 2 ], [ 16807, "?" ] ]
1/401
..
401/401
[ [ 7, 7, 7, 7, 7, 7, 7, 7, 7, 7 ],
[ [ 7, 7, 7, 7, 7, 7, 7, 7, 7, 7 ],
  [ [ 1, 0, 0, 0, 0, 0, 0, 0, 0, 0 ], [ 0, 1, 0, 0, 0, 0, 0, 0, 0, 0 ],
    [ 0, 0, 1, 0, 0, 0, 0, 0, 0, 0 ], [ 0, 0, 0, 1, 0, 0, 0, 0, 0, 0 ],
    [ 0, 0, 0, 0, 1, 0, 0, 0, 0, 0 ], [ 0, 0, 0, 0, 0, 1, 0, 0, 0, 0 ],
    [ 0, 0, 0, 0, 0, 0, 1, 0, 0, 0 ], [ 0, 0, 0, 0, 0, 0, 0, 1, 0, 0 ],
    [ 0, 0, 0, 0, 0, 0, 0, 0, 1, 0 ], [ 0, 0, 0, 0, 0, 0, 0, 0, 0, 1 ] ] ] ]
\end{verbatim}
\medskip
{\bf Case 7b: $G=G(7^5,7)$ which belongs to $\Phi_6$.}\\
We have
$H^4(G,\bZ)\simeq(\bZ/7\bZ)^{\oplus 4}=\langle f_1,f_2,f_3,f_4\rangle$,
$H^4_{\rm p}(G,\bZ)\simeq (\bZ/7\bZ)^{\oplus 3}=\langle f_1f_4^5,f_2,f_3\rangle$ and $f_4\in H^4_{\rm nr}(G,\bZ)$.
Hence we get $H^3_{\rm nr}(\bC(G),\bQ/\bZ)\simeq H^4_{\rm nr}(\bC(G),\bZ)\simeq \bZ/7\bZ$.
\begin{verbatim}
gap> G7:=SmallGroup(7^5,7);
<pc group of size 16807 with 5 generators>
gap> RG7:=ResolutionNormalSeries(LowerCentralSeries(G7),5);
Resolution of length 5 in characteristic 0 for <pc group with 16807 generators> .
gap> CR_CocyclesAndCoboundaries(RG7,4,true).torsionCoefficients; # H^4(G,Z)
[ [ 7, 7, 7, 7 ] ]
gap> H4pFromResolution(RG7);
18[ [ 343, 2 ], [ 343, 2 ], [ 343, 2 ], [ 343, 2 ], [ 343, 2 ],
    [ 343, 2 ], [ 343, 2 ], [ 343, 2 ], [ 343, 5 ], [ 2401, 13 ],
    [ 2401, 13 ], [ 2401, 13 ], [ 2401, 13 ], [ 2401, 13 ],
    [ 2401, 13 ], [ 2401, 13 ], [ 2401, 13 ], [ 16807, "?" ] ]
1/18
...
18/18
[ [ 7, 7, 7 ],
[ [ 7, 7, 7, 7 ],
  [ [ 1, 0, 0, 5 ], [ 0, 1, 0, 0 ], [ 0, 0, 1, 0 ] ] ] ]
gap> IsUnramifiedH3(RG7,[0,0,0,1]:Subgroup);
1/1
[ [ 7, 7, 7 ], [ 0, 0, 0 ] ]
true
\end{verbatim}
\medskip
{\bf Case 7c: $G=G(7^5,73)$ which belongs to $\Phi_7$.}\\
We have
$H^4(G,\bZ)\simeq(\bZ/7\bZ)^{\oplus 6}=\langle f_1,\ldots,f_6\rangle$,
$H^4_{\rm p}(G,\bZ)\simeq (\bZ/7\bZ)^{\oplus 5}=\langle f_1,f_2f_6^6,f_3f_3f_6^3,f_4,f_5f_6^6\rangle$ and $f_6\in H^4_{\rm nr}(G,\bZ)$.
Hence we get $H^3_{\rm nr}(\bC(G),\bQ/\bZ)\simeq H^4_{\rm nr}(\bC(G),\bZ)\simeq \bZ/7\bZ$.
\begin{verbatim}
gap> G73:=SmallGroup(7^5,73);
<pc group of size 16807 with 5 generators>
gap> RG73:=ResolutionNormalSeries(LowerCentralSeries(G73),5);
Resolution of length 5 in characteristic 0 for <pc group with 16807 generators> .
gap> CR_CocyclesAndCoboundaries(RG73,4,true).torsionCoefficients; # H^4(G,Z)
[ [ 7, 7, 7, 7, 7, 7 ] ]
gap> H4pFromResolution(RG73);
24[ [ 343, 2 ], [ 343, 2 ], [ 343, 2 ], [ 343, 2 ], [ 343, 2 ], [ 343, 2 ],
    [ 343, 2 ], [ 343, 5 ], [ 343, 2 ], [ 343, 2 ], [ 343, 2 ], [ 343, 2 ],
    [ 343, 2 ], [ 343, 2 ], [ 343, 2 ], [ 2401, 14 ], [ 2401, 14 ],
    [ 2401, 14 ], [ 2401, 14 ], [ 2401, 13 ], [ 2401, 14 ],
    [ 2401, 14 ], [ 2401, 14 ], [ 16807, "?" ] ]
1/24
...
24/24
[ [ 7, 7, 7, 7, 7 ],
[ [ 7, 7, 7, 7, 7, 7 ],
  [ [ 1, 0, 0, 0, 0, 0 ], [ 0, 1, 0, 0, 0, 6 ], [ 0, 0, 1, 0, 0, 2 ],
    [ 0, 0, 0, 1, 0, 0 ], [ 0, 0, 0, 0, 1, 6 ] ] ] ]
gap> IsUnramifiedH3(RG73,[0,0,0,0,0,1]:Subgroup);
1/2
[ [ 7, 7, 7 ], [ 0, 0, 0 ] ]
2/2
[ [ 7, 7, 7 ], [ 0, 0, 0 ] ]
true
\end{verbatim}
\medskip
{\bf Case 7d: $G=G(7^5,38)$ which belongs to $\Phi_{10}$.}\\
We have
$H^4(G,\bZ)\simeq(\bZ/7\bZ)^{\oplus 3}=\langle f_1,f_2,f_3\rangle$,
$H^4_{\rm p}(G,\bZ)\simeq (\bZ/7\bZ)^{\oplus 2}=\langle f_1f_3^2,f_2f_3^5\rangle$ and $f_3\in H^4_{\rm nr}(G,\bZ)$.
Hence we get $H^3_{\rm nr}(\bC(G),\bQ/\bZ)\simeq H^4_{\rm nr}(\bC(G),\bZ)\simeq \bZ/7\bZ$.
\begin{verbatim}
gap> G38:=SmallGroup(7^5,38);
<pc group of size 16807 with 5 generators>
gap> RG38:=ResolutionNormalSeries(LowerCentralSeries(G38),5);
Resolution of length 5 in characteristic 0 for <pc group with 16807 generators> .
gap> CR_CocyclesAndCoboundaries(RG38,4,true).torsionCoefficients; # H^4(G,Z)
[ [ 7, 7, 7 ] ]
gap> H4pFromResolution(RG38);
25[ [ 49, 1 ], [ 49, 2 ], [ 49, 1 ], [ 49, 1 ], [ 49, 1 ], [ 49, 1 ],
    [ 49, 1 ], [ 343, 5 ], [ 343, 3 ], [ 343, 2 ], [ 343, 4 ], [ 343, 4 ],
    [ 343, 4 ], [ 343, 4 ], [ 343, 4 ], [ 343, 4 ], [ 2401, 7 ], [ 2401, 8 ],
    [ 2401, 8 ], [ 2401, 8 ], [ 2401, 13 ], [ 2401, 8 ], [ 2401, 8 ],
    [ 2401, 8 ], [ 16807, "?" ] ]
1/25
...
25/25
[ [ 7, 7 ],
[ [ 7, 7, 7 ],
  [ [ 1, 0, 2 ], [ 0, 1, 5 ] ] ] ]
gap> IsUnramifiedH3(RG38,[0,0,1]:Subgroup);
1/2
[ [ 7 ], [ 0 ] ]
2/2
[ [ 7, 7, 7 ], [ 0, 0, 0 ] ]
true
\end{verbatim}

\section{Tables of $H^3_{\rm nr}(\bC(G),\bQ/\bZ)$ and $H^3_{\rm s}(G,\bQ/\bZ)$ for groups $G$ of order $3^5$ and $5^5$}\label{sestable}

By using the functions {\tt H4pFromResolution($G$)}
and {\tt IsUnramifiedH3($RG$,$l$:Subgroup)} given
in the previous section,
we obtain the stable cohomology $H^3_{\rm s}(G,\bQ/\bZ)$
and the unramified cohomology $H^3_{\rm nr}(\bC(G),\bQ/\bZ)$
for all groups $G$ of order $243$ and $3125$.
The result is given in Table $6$ and Table $7$
for groups $G$ of order $243$ and $3125$ respectively:

\newpage
{\scriptsize
\begin{table}[h]
\begin{tabular}{cl|ccc|cc}
\multicolumn{2}{c|}{$|G|=3^5$} & $H^3(G,\bQ/\bZ)$ & $H^3_{\rm nr}(G,\bQ/\bZ)$
 & $H^3_{\rm p}(G,\bQ/\bZ)$
 & $H^3_{\rm s}(G,\bQ/\bZ)$ & $H^3_{\rm nr}(\bC(G),\bQ/\bZ)$\\\hline
&$G(3^5,1)$ & $\bZ/243\bZ$ & $=$ & $\bZ/243\bZ$ & 0 & 0\\
&$G(3^5,10)$ & $(\bZ/9\bZ)^{\oplus 2}\oplus\bZ/27\bZ$
& $=$ & $(\bZ/9\bZ)^{\oplus 2}\oplus\bZ/27\bZ$ & 0 & 0\\
&$G(3^5,23)$ & $(\bZ/3\bZ)^{\oplus 2}\oplus\bZ/81\bZ$
& $=$ & $(\bZ/3\bZ)^{\oplus 2}\oplus\bZ/81\bZ$ & 0 & 0\\
$\Phi_1$&$G(3^5,31)$ & $(\bZ/3\bZ)^{\oplus 4}\oplus(\bZ/9\bZ)^{\oplus 3}$
& $(\bZ/3\bZ)^{\oplus 3}\oplus(\bZ/9\bZ)^{\oplus 3}$
& $(\bZ/3\bZ)^{\oplus 3}\oplus(\bZ/9\bZ)^{\oplus 3}$ &
$\bZ/3\bZ$ & 0\\
&$G(3^5,48)$ & $(\bZ/3\bZ)^{\oplus 6}\oplus\bZ/27\bZ$
& $(\bZ/3\bZ)^{\oplus 5}\oplus\bZ/27\bZ$
& $(\bZ/3\bZ)^{\oplus 5}\oplus\bZ/27\bZ$ &
$\bZ/3\bZ$ & 0\\
&$G(3^5,61)$ & $(\bZ/3\bZ)^{\oplus 13}\oplus\bZ/9\bZ$
& $(\bZ/3\bZ)^{\oplus 9}\oplus\bZ/9\bZ$
& $(\bZ/3\bZ)^{\oplus 9}\oplus\bZ/9\bZ$ &
$(\bZ/3\bZ)^{\oplus 4}$ & 0\\
&$G(3^5,67)$ & $(\bZ/3\bZ)^{\oplus 25}$
& $(\bZ/3\bZ)^{\oplus 15}$
& $(\bZ/3\bZ)^{\oplus 15}$ &
$(\bZ/3\bZ)^{\oplus 10}$ & 0\\\hline
&$G(3^5,2)$ & $(\bZ/9\bZ)^{\oplus 4}$
& $\bZ/3\bZ\oplus (\bZ/9\bZ)^{\oplus 3}$
& $\bZ/3\bZ\oplus (\bZ/9\bZ)^{\oplus 3}$ &
$\bZ/3\bZ$ & 0\\
&$G(3^5,11)$ & $\bZ/3\bZ\oplus (\bZ/9\bZ)^{\oplus 2}$
& $=$
& $\bZ/3\bZ\oplus (\bZ/9\bZ)^{\oplus 2}$ &
0 & 0\\
&$G(3^5,12)$ & $(\bZ/3\bZ)^{\oplus 3}\oplus\bZ/27\bZ$
& $=$
& $(\bZ/3\bZ)^{\oplus 3}\oplus\bZ/27\bZ$ &
0 & 0\\
&$G(3^5,21)$ & $(\bZ/3\bZ)^{\oplus 2}\oplus\bZ/27\bZ$
& $=$
& $(\bZ/3\bZ)^{\oplus 2}\oplus\bZ/27\bZ$ &
0 & 0\\
&$G(3^5,24)$ & $\bZ/3\bZ\oplus\bZ/27\bZ$
& $=$
& $\bZ/3\bZ\oplus\bZ/27\bZ$ &
0 & 0\\
&$G(3^5,32)$ & $(\bZ/3\bZ)^{\oplus 8}\oplus\bZ/9\bZ$
& $(\bZ/3\bZ)^{\oplus 6}\oplus\bZ/9\bZ$
& $(\bZ/3\bZ)^{\oplus 6}\oplus\bZ/9\bZ$ &
$(\bZ/3\bZ)^{\oplus 2}$ & 0\\
&$G(3^5,33)$ & $(\bZ/3\bZ)^{\oplus 6}\oplus\bZ/9\bZ$
& $(\bZ/3\bZ)^{\oplus 5}\oplus\bZ/9\bZ$
& $(\bZ/3\bZ)^{\oplus 5}\oplus\bZ/9\bZ$ &
$\bZ/3\bZ$ & 0\\
$\Phi_2$
& $G(3^5,34)$ & $(\bZ/3\bZ)^{\oplus 3}\oplus(\bZ/9\bZ)^{\oplus 2}$
& $=$
& $(\bZ/3\bZ)^{\oplus 3}\oplus(\bZ/9\bZ)^{\oplus 2}$ &
0 & 0\\
&$G(3^5,35)$ & $(\bZ/3\bZ)^{\oplus 8}\oplus\bZ/9\bZ$
& $(\bZ/3\bZ)^{\oplus 6}\oplus\bZ/9\bZ$
& $(\bZ/3\bZ)^{\oplus 6}\oplus\bZ/9\bZ$ &
$(\bZ/3\bZ)^{\oplus 2}$ & 0\\
&$G(3^5,36)$ & $(\bZ/3\bZ)^{\oplus 4}\oplus\bZ/9\bZ$
& $=$
& $(\bZ/3\bZ)^{\oplus 4}\oplus\bZ/9\bZ$ &
0 & 0\\
&$G(3^5,49)$ & $(\bZ/3\bZ)^{\oplus 4}\oplus\bZ/9\bZ$
& $=$
& $(\bZ/3\bZ)^{\oplus 4}\oplus\bZ/9\bZ$ &
0 & 0\\
&$G(3^5,50)$ & $(\bZ/3\bZ)^{\oplus 3}\oplus\bZ/27\bZ$
& $=$
& $(\bZ/3\bZ)^{\oplus 3}\oplus\bZ/27\bZ$ &
0 & 0\\
&$G(3^5,62)$ & $(\bZ/3\bZ)^{\oplus 17}$
& $(\bZ/3\bZ)^{\oplus 11}$
& $(\bZ/3\bZ)^{\oplus 11}$ &
$(\bZ/3\bZ)^{\oplus 6}$ & 0\\
&$G(3^5,63)$ & $(\bZ/3\bZ)^{\oplus 11}$
& $(\bZ/3\bZ)^{\oplus 9}$
& $(\bZ/3\bZ)^{\oplus 9}$ &
$(\bZ/3\bZ)^{\oplus 2}$ & 0\\
&$G(3^5,64)$ & $(\bZ/3\bZ)^{\oplus 9}\oplus\bZ/9\bZ$
& $(\bZ/3\bZ)^{\oplus 7}\oplus\bZ/9\bZ$
& $(\bZ/3\bZ)^{\oplus 7}\oplus\bZ/9\bZ$ &
$(\bZ/3\bZ)^{\oplus 2}$ & 0\\\hline
&$G(3^5,13)$ & $(\bZ/3\bZ)^{\oplus 5}\oplus\bZ/9\bZ$
& $(\bZ/3\bZ)^{\oplus 3}\oplus\bZ/9\bZ$
& $(\bZ/3\bZ)^{\oplus 3}\oplus\bZ/9\bZ$ &
$(\bZ/3\bZ)^{\oplus 2}$ & 0\\
&$G(3^5,14)$ & $(\bZ/3\bZ)^{\oplus 2}\oplus(\bZ/9\bZ)^{\oplus 2}$
& $\bZ/3\bZ\oplus(\bZ/9\bZ)^{\oplus 2}$
& $\bZ/3\bZ\oplus(\bZ/9\bZ)^{\oplus 2}$ &
$\bZ/3\bZ$ & 0\\
&$G(3^5,15)$ & $(\bZ/3\bZ)^{\oplus 2}\oplus(\bZ/9\bZ)^{\oplus 2}$
& $(\bZ/3\bZ)^{\oplus 3}\oplus\bZ/9\bZ$
& $(\bZ/3\bZ)^{\oplus 3}\oplus\bZ/9\bZ$ &
$\bZ/3\bZ$ & 0\\
&$G(3^5,16)$ & $\bZ/3\bZ\oplus(\bZ/9\bZ)^{\oplus 2}$
& $(\bZ/3\bZ)^{\oplus 2}\oplus\bZ/9\bZ$
& $(\bZ/3\bZ)^{\oplus 2}\oplus\bZ/9\bZ$ &
$\bZ/3\bZ$ & 0\\
&$G(3^5,17)$ & $(\bZ/3\bZ)^{\oplus 4}\oplus\bZ/9\bZ$
& $(\bZ/3\bZ)^{\oplus 3}\oplus\bZ/9\bZ$
& $(\bZ/3\bZ)^{\oplus 3}\oplus\bZ/9\bZ$ &
$\bZ/3\bZ$ & 0\\
&$G(3^5,18)$ & $(\bZ/3\bZ)^{\oplus 3}\oplus\bZ/9\bZ$
& $(\bZ/3\bZ)^{\oplus 2}\oplus\bZ/9\bZ$
& $(\bZ/3\bZ)^{\oplus 2}\oplus\bZ/9\bZ$ &
$\bZ/3\bZ$ & 0\\
$\Phi_3$
&$G(3^5,19)$ & $\bZ/3\bZ\oplus\bZ/27\bZ$
& $=$
& $\bZ/3\bZ\oplus\bZ/27\bZ$ &
0 & 0\\
&$G(3^5,20)$ & $\bZ/3\bZ\oplus\bZ/27\bZ$
& $=$
& $\bZ/3\bZ\oplus\bZ/27\bZ$ &
0 & 0\\
&$G(3^5,51)$ & $(\bZ/3\bZ)^{\oplus 8}$
& $(\bZ/3\bZ)^{\oplus 6}$
& $(\bZ/3\bZ)^{\oplus 6}$ &
$(\bZ/3\bZ)^{\oplus 2}$ & 0\\
&$G(3^5,52)$ & $(\bZ/3\bZ)^{\oplus 5}\oplus\bZ/9\bZ$
& $(\bZ/3\bZ)^{\oplus 4}\oplus\bZ/9\bZ$
& $(\bZ/3\bZ)^{\oplus 4}\oplus\bZ/9\bZ$ &
$\bZ/3\bZ$ & 0\\
&$G(3^5,53)$ & $(\bZ/3\bZ)^{\oplus 9}$
& $(\bZ/3\bZ)^{\oplus 7}$
& $(\bZ/3\bZ)^{\oplus 7}$ &
$(\bZ/3\bZ)^{\oplus 2}$ & 0\\
&$G(3^5,54)$ & $(\bZ/3\bZ)^{\oplus 5}\oplus\bZ/9\bZ$
& $(\bZ/3\bZ)^{\oplus 4}\oplus\bZ/9\bZ$
& $(\bZ/3\bZ)^{\oplus 4}\oplus\bZ/9\bZ$ &
$\bZ/3\bZ$ & 0\\
&$G(3^5,55)$ & $(\bZ/3\bZ)^{\oplus 5}\oplus\bZ/9\bZ$
& $(\bZ/3\bZ)^{\oplus 4}\oplus\bZ/9\bZ$
& $(\bZ/3\bZ)^{\oplus 4}\oplus\bZ/9\bZ$ &
$\bZ/3\bZ$ & 0\\\hline
&$G(3^5,37)$ & $(\bZ/3\bZ)^{\oplus 12}$
& $(\bZ/3\bZ)^{\oplus 9}$
& $(\bZ/3\bZ)^{\oplus 9}$ &
$(\bZ/3\bZ)^{\oplus 3}$ & 0\\
&$G(3^5,38)$ & $(\bZ/3\bZ)^{\oplus 8}$
& $(\bZ/3\bZ)^{\oplus 6}$
& $(\bZ/3\bZ)^{\oplus 6}$ &
$(\bZ/3\bZ)^{\oplus 2}$ & 0\\
&$G(3^5,39)$ & $(\bZ/3\bZ)^{\oplus 7}$
& $(\bZ/3\bZ)^{\oplus 6}$
& $(\bZ/3\bZ)^{\oplus 6}$ &
$\bZ/3\bZ$ & 0\\
&$G(3^5,40)$ & $(\bZ/3\bZ)^{\oplus 7}$
& $(\bZ/3\bZ)^{\oplus 6}$
& $(\bZ/3\bZ)^{\oplus 6}$ &
$\bZ/3\bZ$ & 0\\
&$G(3^5,41)$ & $(\bZ/3\bZ)^{\oplus 5}$
& $(\bZ/3\bZ)^{\oplus 4}$
& $(\bZ/3\bZ)^{\oplus 4}$ &
$\bZ/3\bZ$ & 0\\
$\Phi_4$
&$G(3^5,42)$ & $(\bZ/3\bZ)^{\oplus 4}\oplus\bZ/9\bZ$
& $(\bZ/3\bZ)^{\oplus 3}\oplus\bZ/9\bZ$
& $(\bZ/3\bZ)^{\oplus 3}\oplus\bZ/9\bZ$ &
$\bZ/3\bZ$ & 0\\
&$G(3^5,43)$ & $(\bZ/3\bZ)^{\oplus 3}\oplus\bZ/9\bZ$
& $=$
& $(\bZ/3\bZ)^{\oplus 3}\oplus\bZ/9\bZ$ &
0 & 0\\
&$G(3^5,44)$ & $(\bZ/3\bZ)^{\oplus 4}$
& $=$
& $(\bZ/3\bZ)^{\oplus 4}$ &
0 & 0\\
&$G(3^5,45)$ & $(\bZ/3\bZ)^{\oplus 3}\oplus\bZ/9\bZ$
& $=$
& $(\bZ/3\bZ)^{\oplus 3}\oplus\bZ/9\bZ$ &
0 & 0\\
&$G(3^5,46)$ & $(\bZ/3\bZ)^{\oplus 4}$
& $=$
& $(\bZ/3\bZ)^{\oplus 4}$ &
0 & 0\\
&$G(3^5,47)$ & $(\bZ/3\bZ)^{\oplus 4}$
& $=$
& $(\bZ/3\bZ)^{\oplus 4}$ &
0 & 0\\\hline
\raisebox{-1.6ex}[0cm][0cm]{$\Phi_5$}
&$G(3^5,65)$ & $(\bZ/3\bZ)^{\oplus 15}$
& $(\bZ/3\bZ)^{\oplus 10}$
& $(\bZ/3\bZ)^{\oplus 10}$ &
$(\bZ/3\bZ)^{\oplus 5}$ & 0\\
&$G(3^5,66)$ & $(\bZ/3\bZ)^{\oplus 10}$
& $=$
& $(\bZ/3\bZ)^{\oplus 10}$ &
0 & 0\\\hline
&$G(3^5,3)$ & $(\bZ/3\bZ)^{\oplus 6}$
& $=$
& $(\bZ/3\bZ)^{\oplus 6}$ &
0 & 0\\
&$G(3^5,4)$ & $(\bZ/3\bZ)^{\oplus 5}$
& $=$
& $(\bZ/3\bZ)^{\oplus 5}$ &
0 & 0\\
&$G(3^5,5)$ & $(\bZ/3\bZ)^{\oplus 4}$
& $=$
& $(\bZ/3\bZ)^{\oplus 4}$ &
0 & 0\\
$\Phi_6$
&$G(3^5,6)$ & $(\bZ/3\bZ)^{\oplus 5}$
& $=$
& $(\bZ/3\bZ)^{\oplus 5}$ &
0 & 0\\
&$G(3^5,7)$ & $(\bZ/3\bZ)^{\oplus 4}$
& $=$
& $(\bZ/3\bZ)^{\oplus 4}$ &
0 & 0\\
&$G(3^5,8)$ & $(\bZ/3\bZ)^{\oplus 3}\oplus\bZ/9\bZ$
& $=$
& $(\bZ/3\bZ)^{\oplus 3}\oplus\bZ/9\bZ$ &
0 & 0\\
&$G(3^5,9)$ & $(\bZ/3\bZ)^{\oplus 3}\oplus\bZ/9\bZ$
& $=$
& $(\bZ/3\bZ)^{\oplus 3}\oplus\bZ/9\bZ$ &
0 & 0\\\hline
&$G(3^5,56)$ & $(\bZ/3\bZ)^{\oplus 7}$
& $(\bZ/3\bZ)^{\oplus 6}$ & $(\bZ/3\bZ)^{\oplus 5}$ &
$(\bZ/3\bZ)^{\oplus 2}$ & $\bZ/3\bZ$\\
&$G(3^5,57)$ & $(\bZ/3\bZ)^{\oplus 6}$
& $(\bZ/3\bZ)^{\oplus 6}$ & $(\bZ/3\bZ)^{\oplus 5}$ &
$\bZ/3\bZ$ & $\bZ/3\bZ$\\
$\Phi_7$ & $G(3^5,58)$ & $(\bZ/3\bZ)^{\oplus 9}$
& $(\bZ/3\bZ)^{\oplus 7}$ & $(\bZ/3\bZ)^{\oplus 6}$ &
$(\bZ/3\bZ)^{\oplus 3}$ & $\bZ/3\bZ$\\
&$G(3^5,59)$ & $(\bZ/3\bZ)^{\oplus 6}$
& $(\bZ/3\bZ)^{\oplus 6}$ & $(\bZ/3\bZ)^{\oplus 5}$ &
$\bZ/3\bZ$ & $\bZ/3\bZ$\\
&$G(3^5,60)$ & $(\bZ/3\bZ)^{\oplus 6}$
& $(\bZ/3\bZ)^{\oplus 6}$ & $(\bZ/3\bZ)^{\oplus 5}$ &
$\bZ/3\bZ$ & $\bZ/3\bZ$\\\hline
$\Phi_8$
&$G(3^5,22)$ & $\bZ/3\bZ\oplus\bZ/9\bZ$
& $=$
& $\bZ/3\bZ\oplus\bZ/9\bZ$ &
0 & 0\\\hline
&$G(3^5,25)$ & $\bZ/3\bZ\oplus\bZ/9\bZ$
& $=$
& $\bZ/3\bZ\oplus\bZ/9\bZ$ &
0 & 0\\
$\Phi_9$
&$G(3^5,26)$ & $(\bZ/3\bZ)^{\oplus 3}\oplus\bZ/9\bZ$
& $=$
& $(\bZ/3\bZ)^{\oplus 3}\oplus\bZ/9\bZ$ &
0 & 0\\
&$G(3^5,27)$ & $\bZ/3\bZ\oplus\bZ/9\bZ$
& $=$
& $\bZ/3\bZ\oplus\bZ/9\bZ$ &
0 & 0\\\hline
&$G(3^5,28)$ & $(\bZ/3\bZ)^{\oplus 2}\oplus \bZ/9\bZ$
& $(\bZ/3\bZ)^{\oplus 3}$ & $(\bZ/3\bZ)^{\oplus 3}$ & $\bZ/3\bZ$ & 0\\
$\Phi_{10}$&$G(3^5,29)$ & $\bZ/3\bZ\oplus \bZ/9\bZ$
& $(\bZ/3\bZ)^{\oplus 2}$ & $(\bZ/3\bZ)^{\oplus 2}$ & $\bZ/3\bZ$ & 0\\
& $G(3^5,30)$ & $\bZ/3\bZ\oplus \bZ/9\bZ$
& $(\bZ/3\bZ)^{\oplus 2}$ & $(\bZ/3\bZ)^{\oplus 2}$ & $\bZ/3\bZ$ & 0
\end{tabular}\vspace*{2mm}\\
{\normalsize
Table $6$: $H^3_{\rm s}(G,\bQ/\bZ)$ and
$H^3_{\rm nr}(\bC(G),\bQ/\bZ)$ for groups $G$ of order $243$
}
\end{table}
}
\newpage
{\scriptsize
\begin{table}[!h]
\begin{tabular}{cl|ccc|cc}
\multicolumn{2}{c|}{$|G|=5^5$} & $H^3(G,\bQ/\bZ)$ & $H^3_{\rm nr}(G,\bQ/\bZ)$
 & $H^3_{\rm p}(G,\bQ/\bZ)$
& $H^3_{\rm s}(G,\bQ/\bZ)$ & $H^3_{\rm nr}(\bC(G),\bQ/\bZ)$\\\hline
&$G(5^5,1)$ & $\bZ/3125\bZ$ & $=$ & $\bZ/3125\bZ$ & 0 & 0\\
&$G(5^5,15)$ & $(\bZ/25\bZ)^{\oplus 2}\oplus\bZ/125\bZ$
& $=$ & $(\bZ/25\bZ)^{\oplus 2}\oplus\bZ/125\bZ$ & 0 & 0\\
&$G(5^5,28)$ & $(\bZ/5\bZ)^{\oplus 2}\oplus\bZ/625\bZ$
& $=$ & $(\bZ/5\bZ)^{\oplus 2}\oplus\bZ/625\bZ$ & 0 & 0\\
$\Phi_1$&$G(5^5,39)$ & $(\bZ/5\bZ)^{\oplus 4}\oplus(\bZ/25\bZ)^{\oplus 3}$
& $(\bZ/5\bZ)^{\oplus 3}\oplus(\bZ/25\bZ)^{\oplus 3}$
& $(\bZ/5\bZ)^{\oplus 3}\oplus(\bZ/25\bZ)^{\oplus 3}$ &
$\bZ/5\bZ$ & 0\\
&$G(5^5,58)$ & $(\bZ/5\bZ)^{\oplus 6}\oplus\bZ/125\bZ$
& $(\bZ/5\bZ)^{\oplus 5}\oplus\bZ/125\bZ$
& $(\bZ/5\bZ)^{\oplus 5}\oplus\bZ/125\bZ$ &
$\bZ/5\bZ$ & 0\\
&$G(5^5,71)$ & $(\bZ/5\bZ)^{\oplus 13}\oplus\bZ/25\bZ$
& $(\bZ/5\bZ)^{\oplus 9}\oplus\bZ/25\bZ$
& $(\bZ/5\bZ)^{\oplus 9}\oplus\bZ/25\bZ$ &
$(\bZ/5\bZ)^{\oplus 4}$ & 0\\
&$G(5^5,77)$ & $(\bZ/5\bZ)^{\oplus 25}$
& $(\bZ/5\bZ)^{\oplus 15}$
& $(\bZ/5\bZ)^{\oplus 15}$ &
$(\bZ/5\bZ)^{\oplus 10}$ & 0\\\hline
&$G(5^5,2)$ & $(\bZ/25\bZ)^{\oplus 4}$
& $\bZ/5\bZ\oplus (\bZ/25\bZ)^{\oplus 3}$
& $\bZ/5\bZ\oplus (\bZ/25\bZ)^{\oplus 3}$ &
$\bZ/5\bZ$ & 0\\
&$G(5^5,16)$ & $\bZ/5\bZ\oplus (\bZ/25\bZ)^{\oplus 2}$
& $=$
& $\bZ/5\bZ\oplus (\bZ/25\bZ)^{\oplus 2}$ &
0 & 0\\
&$G(5^5,17)$ & $(\bZ/5\bZ)^{\oplus 3}\oplus\bZ/125\bZ$
& $=$
& $(\bZ/5\bZ)^{\oplus 3}\oplus\bZ/125\bZ$ &
0 & 0\\
&$G(5^5,26)$ & $(\bZ/5\bZ)^{\oplus 2}\oplus\bZ/125\bZ$
& $=$
& $(\bZ/5\bZ)^{\oplus 2}\oplus\bZ/125\bZ$ &
0 & 0\\
&$G(5^5,29)$ & $\bZ/5\bZ\oplus\bZ/125\bZ$
& $=$
& $\bZ/5\bZ\oplus\bZ/125\bZ$ &
0 & 0\\
&$G(5^5,40)$ & $(\bZ/5\bZ)^{\oplus 8}\oplus\bZ/25\bZ$
& $(\bZ/5\bZ)^{\oplus 6}\oplus\bZ/25\bZ$
& $(\bZ/5\bZ)^{\oplus 6}\oplus\bZ/25\bZ$ &
$(\bZ/5\bZ)^{\oplus 2}$ & 0\\
&$G(5^5,41)$ & $(\bZ/5\bZ)^{\oplus 6}\oplus\bZ/25\bZ$
& $(\bZ/5\bZ)^{\oplus 5}\oplus\bZ/25\bZ$
& $(\bZ/5\bZ)^{\oplus 5}\oplus\bZ/25\bZ$ &
$\bZ/5\bZ$ & 0\\
$\Phi_2$
& $G(5^5,42)$ & $(\bZ/5\bZ)^{\oplus 3}\oplus(\bZ/25\bZ)^{\oplus 2}$
& $=$
& $(\bZ/5\bZ)^{\oplus 3}\oplus(\bZ/25\bZ)^{\oplus 2}$ &
0 & 0\\
&$G(5^5,43)$ & $(\bZ/5\bZ)^{\oplus 8}\oplus\bZ/25\bZ$
& $(\bZ/5\bZ)^{\oplus 6}\oplus\bZ/25\bZ$
& $(\bZ/5\bZ)^{\oplus 6}\oplus\bZ/25\bZ$ &
$(\bZ/5\bZ)^{\oplus 2}$ & 0\\
&$G(5^5,44)$ & $(\bZ/5\bZ)^{\oplus 4}\oplus\bZ/25\bZ$
& $=$
& $(\bZ/5\bZ)^{\oplus 4}\oplus\bZ/25\bZ$ &
0 & 0\\
&$G(5^5,59)$ & $(\bZ/5\bZ)^{\oplus 4}\oplus\bZ/25\bZ$
& $=$
& $(\bZ/5\bZ)^{\oplus 4}\oplus\bZ/25\bZ$ &
0 & 0\\
&$G(5^5,60)$ & $(\bZ/5\bZ)^{\oplus 3}\oplus\bZ/125\bZ$
& $=$
& $(\bZ/5\bZ)^{\oplus 3}\oplus\bZ/125\bZ$ &
0 & 0\\
&$G(5^5,72)$ & $(\bZ/5\bZ)^{\oplus 17}$
& $(\bZ/5\bZ)^{\oplus 11}$
& $(\bZ/5\bZ)^{\oplus 11}$ &
$(\bZ/5\bZ)^{\oplus 6}$ & 0\\
&$G(5^5,73)$ & $(\bZ/5\bZ)^{\oplus 11}$
& $(\bZ/5\bZ)^{\oplus 9}$
& $(\bZ/5\bZ)^{\oplus 9}$ &
$(\bZ/5\bZ)^{\oplus 2}$ & 0\\
&$G(5^5,74)$ & $(\bZ/5\bZ)^{\oplus 9}\oplus\bZ/25\bZ$
& $(\bZ/5\bZ)^{\oplus 7}\oplus\bZ/25\bZ$
& $(\bZ/5\bZ)^{\oplus 7}\oplus\bZ/25\bZ$ &
$(\bZ/5\bZ)^{\oplus 2}$ & 0\\\hline
&$G(5^5,18)$ & $(\bZ/5\bZ)^{\oplus 5}\oplus\bZ/25\bZ$
& $(\bZ/5\bZ)^{\oplus 3}\oplus\bZ/25\bZ$
& $(\bZ/5\bZ)^{\oplus 3}\oplus\bZ/25\bZ$ &
$(\bZ/5\bZ)^{\oplus 2}$ & 0\\
&$G(5^5,19)$ & $(\bZ/5\bZ)^{\oplus 2}\oplus(\bZ/25\bZ)^{\oplus 2}$
& $\bZ/5\bZ\oplus(\bZ/25\bZ)^{\oplus 2}$
& $\bZ/5\bZ\oplus(\bZ/25\bZ)^{\oplus 2}$ &
$\bZ/5\bZ$ & 0\\
&$G(5^5,20)$ & $(\bZ/5\bZ)^{\oplus 2}\oplus(\bZ/25\bZ)^{\oplus 2}$
& $\bZ/5\bZ\oplus(\bZ/25\bZ)^{\oplus 2}$
& $\bZ/5\bZ\oplus(\bZ/25\bZ)^{\oplus 2}$ &
$\bZ/5\bZ$ & 0\\
&$G(5^5,21)$ & $(\bZ/5\bZ)\oplus(\bZ/25\bZ)^{\oplus 2}$
& $(\bZ/5\bZ)^{\oplus 2}\oplus\bZ/25\bZ$
& $(\bZ/5\bZ)^{\oplus 2}\oplus\bZ/25\bZ$ &
$\bZ/5\bZ$ & 0\\
&$G(5^5,22)$ & $(\bZ/5\bZ)^{\oplus 4}\oplus\bZ/25\bZ$
& $(\bZ/5\bZ)^{\oplus 3}\oplus\bZ/25\bZ$
& $(\bZ/5\bZ)^{\oplus 3}\oplus\bZ/25\bZ$ &
$\bZ/5\bZ$ & 0\\
&$G(5^5,23)$ & $(\bZ/5\bZ)^{\oplus 3}\oplus\bZ/25\bZ$
& $(\bZ/5\bZ)^{\oplus 2}\oplus\bZ/25\bZ$
& $(\bZ/5\bZ)^{\oplus 2}\oplus\bZ/25\bZ$ &
$\bZ/5\bZ$ & 0\\
$\Phi_3$
&$G(5^5,24)$ & $\bZ/5\bZ\oplus\bZ/125\bZ$
& $=$
& $\bZ/5\bZ\oplus\bZ/125\bZ$ &
0 & 0\\
&$G(5^5,25)$ & $\bZ/5\bZ\oplus\bZ/125\bZ$
& $=$
& $\bZ/5\bZ\oplus\bZ/125\bZ$ &
0 & 0\\
&$G(5^5,61)$ & $(\bZ/5\bZ)^{\oplus 10}$
& $(\bZ/5\bZ)^{\oplus 7}$
& $(\bZ/5\bZ)^{\oplus 7}$ &
$(\bZ/5\bZ)^{\oplus 3}$ & 0\\
&$G(5^5,62)$ & $(\bZ/5\bZ)^{\oplus 6}\oplus\bZ/25\bZ$
& $(\bZ/5\bZ)^{\oplus 6}$
& $(\bZ/5\bZ)^{\oplus 6}$ &
$(\bZ/5\bZ)^{\oplus 2}$ & 0\\
&$G(5^5,63)$ & $(\bZ/5\bZ)^{\oplus 5}\oplus\bZ/25\bZ$
& $(\bZ/5\bZ)^{\oplus 4}\oplus\bZ/25\bZ$
& $(\bZ/5\bZ)^{\oplus 4}\oplus\bZ/25\bZ$ &
$\bZ/5\bZ$ & 0\\
&$G(5^5,64)$ & $(\bZ/5\bZ)^{\oplus 5}\oplus\bZ/25\bZ$
& $(\bZ/5\bZ)^{\oplus 4}\oplus\bZ/25\bZ$
& $(\bZ/5\bZ)^{\oplus 4}\oplus\bZ/25\bZ$ &
$\bZ/5\bZ$ & 0\\
&$G(5^5,65)$ & $(\bZ/5\bZ)^{\oplus 5}\oplus\bZ/25\bZ$
& $(\bZ/5\bZ)^{\oplus 4}\oplus\bZ/25\bZ$
& $(\bZ/5\bZ)^{\oplus 4}\oplus\bZ/25\bZ$ &
$\bZ/5\bZ$ & 0\\\hline
&$G(5^5,45)$ & $(\bZ/5\bZ)^{\oplus 12}$
& $(\bZ/5\bZ)^{\oplus 9}$
& $(\bZ/5\bZ)^{\oplus 9}$ &
$(\bZ/5\bZ)^{\oplus 3}$ & 0\\
&$G(5^5,46)$ & $(\bZ/5\bZ)^{\oplus 8}$
& $(\bZ/5\bZ)^{\oplus 6}$
& $(\bZ/5\bZ)^{\oplus 6}$ &
$(\bZ/5\bZ)^{\oplus 2}$ & 0\\
&$G(5^5,47)$ & $(\bZ/5\bZ)^{\oplus 7}$
& $(\bZ/5\bZ)^{\oplus 6}$
& $(\bZ/5\bZ)^{\oplus 6}$ &
$\bZ/5\bZ$ & 0\\
&$G(5^5,48)$ & $(\bZ/5\bZ)^{\oplus 7}$
& $(\bZ/5\bZ)^{\oplus 6}$
& $(\bZ/5\bZ)^{\oplus 6}$ &
$\bZ/5\bZ$ & 0\\
&$G(5^5,49)$ & $(\bZ/5\bZ)^{\oplus 5}$
& $(\bZ/5\bZ)^{\oplus 4}$
& $(\bZ/5\bZ)^{\oplus 4}$ &
$\bZ/5\bZ$ & 0\\
&$G(5^5,50)$ & $(\bZ/5\bZ)^{\oplus 4}\oplus\bZ/25\bZ$
& $(\bZ/5\bZ)^{\oplus 3}\oplus\bZ/25\bZ$
& $(\bZ/5\bZ)^{\oplus 3}\oplus\bZ/25\bZ$ &
$\bZ/5\bZ$ & 0\\
$\Phi_4$
&$G(5^5,51)$ & $(\bZ/5\bZ)^{\oplus 3}\oplus\bZ/25\bZ$
& $=$
& $(\bZ/5\bZ)^{\oplus 3}\oplus\bZ/25\bZ$ &
0 & 0\\
&$G(5^5,52)$ & $(\bZ/5\bZ)^{\oplus 4}$
& $=$
& $(\bZ/5\bZ)^{\oplus 4}$ &
0 & 0\\
&$G(5^5,53)$ & $(\bZ/5\bZ)^{\oplus 4}$
& $=$
& $(\bZ/5\bZ)^{\oplus 4}$ &
0 & 0\\
&$G(5^5,54)$ & $(\bZ/5\bZ)^{\oplus 3}\oplus\bZ/25\bZ$
& $=$
& $(\bZ/5\bZ)^{\oplus 3}\oplus\bZ/25\bZ$ &
0 & 0\\
&$G(5^5,55)$ & $(\bZ/5\bZ)^{\oplus 4}$
& $=$
& $(\bZ/5\bZ)^{\oplus 4}$ &
0 & 0\\
&$G(5^5,56)$ & $(\bZ/5\bZ)^{\oplus 4}$
& $=$
& $(\bZ/5\bZ)^{\oplus 4}$ &
0 & 0\\
&$G(5^5,57)$ & $(\bZ/5\bZ)^{\oplus 4}$
& $=$
& $(\bZ/5\bZ)^{\oplus 4}$ &
0 & 0\\\hline
\raisebox{-1.6ex}[0cm][0cm]{$\Phi_5$}
&$G(5^5,75)$ & $(\bZ/5\bZ)^{\oplus 15}$
& $(\bZ/5\bZ)^{\oplus 10}$ & $(\bZ/5\bZ)^{\oplus 10}$ & $(\bZ/5\bZ)^{\oplus  5}$ & 0\\
&$G(5^5,76)$ & $(\bZ/5\bZ)^{\oplus 10}$
& $=$ & $(\bZ/5\bZ)^{\oplus 10}$ & 0 & 0\\\hline
&$G(5^5,3)$ & $(\bZ/5\bZ)^{\oplus 7}$
& $(\bZ/5\bZ)^{\oplus 7}$ & $(\bZ/5\bZ)^{\oplus 6}$ & $\bZ/5\bZ$ & $\bZ/5\bZ$\\
&$G(5^5,4)$ & $(\bZ/5\bZ)^{\oplus 5}$
& $(\bZ/5\bZ)^{\oplus 5}$ & $(\bZ/5\bZ)^{\oplus 4}$ & $\bZ/5\bZ$ & $\bZ/5\bZ$\\
&$G(5^5,5)$ & $(\bZ/5\bZ)^{\oplus 5}$
& $(\bZ/5\bZ)^{\oplus 5}$ & $(\bZ/5\bZ)^{\oplus 4}$ & $\bZ/5\bZ$ & $\bZ/5\bZ$\\
&$G(5^5,6)$ & $(\bZ/5\bZ)^{\oplus 5}$
& $(\bZ/5\bZ)^{\oplus 5}$ & $(\bZ/5\bZ)^{\oplus 4}$ & $\bZ/5\bZ$ & $\bZ/5\bZ$\\
&$G(5^5,7)$ & $(\bZ/5\bZ)^{\oplus 3}\oplus\bZ/25\bZ$
& $(\bZ/5\bZ)^{\oplus 3}\oplus\bZ/25\bZ$ & $(\bZ/5\bZ)^{\oplus 2}\oplus\bZ/25\bZ$ & $\bZ/5\bZ$ & $\bZ/5\bZ$\\
\raisebox{-1.6ex}[0cm][0cm]{$\Phi_6$}
&$G(5^5,8)$ & $(\bZ/5\bZ)^{\oplus 4}$
& $(\bZ/5\bZ)^{\oplus 4}$ & $(\bZ/5\bZ)^{\oplus 3}$ & $\bZ/5\bZ$ & $\bZ/5\bZ$\\
&$G(5^5,9)$ & $(\bZ/5\bZ)^{\oplus 4}$
& $(\bZ/5\bZ)^{\oplus 4}$ & $(\bZ/5\bZ)^{\oplus 3}$ & $\bZ/5\bZ$ & $\bZ/5\bZ$\\
&$G(5^5,10)$ & $(\bZ/5\bZ)^{\oplus 3}\oplus\bZ/25\bZ$
& $(\bZ/5\bZ)^{\oplus 3}\oplus\bZ/25\bZ$ & $(\bZ/5\bZ)^{\oplus 2}\oplus\bZ/25\bZ$ & $\bZ/5\bZ$ & $\bZ/5\bZ$\\
&$G(5^5,11)$ & $(\bZ/5\bZ)^{\oplus 4}$
& $(\bZ/5\bZ)^{\oplus 4}$ & $(\bZ/5\bZ)^{\oplus 3}$ & $\bZ/5\bZ$ & $\bZ/5\bZ$\\
&$G(5^5,12)$ & $(\bZ/5\bZ)^{\oplus 4}$
& $(\bZ/5\bZ)^{\oplus 4}$ & $(\bZ/5\bZ)^{\oplus 3}$ & $\bZ/5\bZ$ & $\bZ/5\bZ$\\
&$G(5^5,13)$ & $(\bZ/5\bZ)^{\oplus 4}$
& $(\bZ/5\bZ)^{\oplus 4}$ & $(\bZ/5\bZ)^{\oplus 3}$ & $\bZ/5\bZ$ & $\bZ/5\bZ$\\
&$G(5^5,14)$ & $(\bZ/5\bZ)^{\oplus 4}$
& $(\bZ/5\bZ)^{\oplus 4}$ & $(\bZ/5\bZ)^{\oplus 3}$ & $\bZ/5\bZ$ & $\bZ/5\bZ$\\\hline
&$G(5^5,66)$ & $(\bZ/5\bZ)^{\oplus 10}$
& $(\bZ/5\bZ)^{\oplus 7}$ & $(\bZ/5\bZ)^{\oplus 6}$ & $(\bZ/5\bZ)^{\oplus 4}$ & $\bZ/5\bZ$\\
&$G(5^5,67)$ & $(\bZ/5\bZ)^{\oplus 6}$
& $(\bZ/5\bZ)^{\oplus 6}$ & $(\bZ/5\bZ)^{\oplus 5}$ & $\bZ/5\bZ$ & $\bZ/5\bZ$\\
$\Phi_7$&$G(5^5,68)$ & $(\bZ/5\bZ)^{\oplus 6}$
& $(\bZ/5\bZ)^{\oplus 6}$ & $(\bZ/5\bZ)^{\oplus 5}$ & $\bZ/5\bZ$ & $\bZ/5\bZ$\\
&$G(5^5,69)$ & $(\bZ/5\bZ)^{\oplus 6}$
& $(\bZ/5\bZ)^{\oplus 6}$ & $(\bZ/5\bZ)^{\oplus 5}$ & $\bZ/5\bZ$ & $\bZ/5\bZ$\\
&$G(5^5,70)$ & $(\bZ/5\bZ)^{\oplus 6}$
& $(\bZ/5\bZ)^{\oplus 6}$ & $(\bZ/5\bZ)^{\oplus 5}$ & $\bZ/5\bZ$ & $\bZ/5\bZ$\\\hline
$\Phi_8$
&$G(5^5,27)$ & $\bZ/5\bZ\oplus \bZ/25\bZ$
& $=$
& $\bZ/5\bZ\oplus \bZ/25\bZ$ &
0 & 0\\\hline
&$G(5^5,30)$ & $(\bZ/5\bZ)^{\oplus 6}$
& $(\bZ/5\bZ)^{\oplus 5}$ & $(\bZ/5\bZ)^{\oplus 5}$ & $\bZ/5\bZ$ & 0\\
$\Phi_9$
&$G(5^5,31)$ & $(\bZ/5\bZ)^{\oplus 4}$
& $(\bZ/5\bZ)^{\oplus 3}$ & $(\bZ/5\bZ)^{\oplus 3}$ & $\bZ/5\bZ$ & 0\\
&$G(5^5,32)$ & $(\bZ/5\bZ)^{\oplus 4}$
& $(\bZ/5\bZ)^{\oplus 3}$ & $(\bZ/5\bZ)^{\oplus 3}$ & $\bZ/5\bZ$ & 0\\\hline
&$G(5^5,33)$ & $(\bZ/5\bZ)^{\oplus 6}$
& $(\bZ/5\bZ)^{\oplus 5}$ & $(\bZ/5\bZ)^{\oplus 4}$ & $(\bZ/5\bZ)^{\oplus 2}$ & $\bZ/5\bZ$\\
&$G(5^5,34)$ & $(\bZ/5\bZ)^{\oplus 4}$
& $(\bZ/5\bZ)^{\oplus 3}$ & $(\bZ/5\bZ)^{\oplus 2}$ & $(\bZ/5\bZ)^{\oplus 2}$ & $\bZ/5\bZ$\\
\raisebox{-1.6ex}[0cm][0cm]{$\Phi_{10}$}
&$G(5^5,35)$ & $(\bZ/5\bZ)^{\oplus 4}$
& $(\bZ/5\bZ)^{\oplus 3}$ & $(\bZ/5\bZ)^{\oplus 2}$ & $(\bZ/5\bZ)^{\oplus 2}$ & $\bZ/5\bZ$\\
&$G(5^5,36)$ & $(\bZ/5\bZ)^{\oplus 4}$
& $(\bZ/5\bZ)^{\oplus 3}$ & $(\bZ/5\bZ)^{\oplus 2}$ & $(\bZ/5\bZ)^{\oplus 2}$ & $\bZ/5\bZ$\\
&$G(5^5,37)$ & $(\bZ/5\bZ)^{\oplus 4}$
& $(\bZ/5\bZ)^{\oplus 3}$ & $(\bZ/5\bZ)^{\oplus 2}$ & $(\bZ/5\bZ)^{\oplus 2}$ & $\bZ/5\bZ$\\
&$G(5^5,38)$ & $(\bZ/5\bZ)^{\oplus 3}$
& $(\bZ/5\bZ)^{\oplus 3}$ & $(\bZ/5\bZ)^{\oplus 2}$ & $\bZ/5\bZ$ & $\bZ/5\bZ$\\
\end{tabular}\\
\vspace*{2mm}
{\normalsize
Table $7$: $H^3_{\rm s}(G,\bQ/\bZ)$ and
$H^3_{\rm nr}(\bC(G),\bQ/\bZ)$ for groups $G$ of order $3125$
}
\end{table}
}


\section{Proof of Theorem \ref{t1.16}}\label{sePT2}

Let $G$ be a non-abelian simple group. First we recall some known results of the rationality problem of $\bC(G)$.

\begin{theorem}\label{t5.1}
{\rm (1) (Hajja \cite{Haj}, Hajja and Kang \cite[pages 530--533, Theorem 6]{HK})}
Let $k$ be any field. Then $k(A_n)$ is $k$-rational for $n=3,4$ where $A_n$ denotes the alternating group of degree $n$.\\
{\rm (2) (Maeda \cite{Mae})}
Let $k$ be any field. Then $k(A_5)$ is $k$-rational where $A_5$ is the alternating group of degree $5$.\\
{\rm (3) (Kemper \cite{Ke})}
Let $PSL_2(\bF_7)$ and $PSp_{4}(\bF_3)$ be the projective special linear group over the finite field $\bF_7$ and the projective symplectic group over the finite field $\bF_3$ respectively. Then $\bQ(\sqrt{-7})(PSL_2(\bF_7))$ is rational over $\bQ(\sqrt{-7})$, and $\bQ(\sqrt{-3})(PSp_4(\bF_3))$ is rational over $\bQ(\sqrt{-3})$.
\end{theorem}

As a result, we know that, for all $i \ge 2$, $H^i_{\rm nr}(\bC(G),\bQ/\bZ)=0$ if $G=A_5$, $PSL_2(\bF_7)$ and $PSp_{4}(\bF_3)$. We do not know whether $\bC(G)$ is $\bC$-rational if $G$ is a non-abelian simple group other than the above three groups. Note that the order of the group $PSL_2(\bF_7)$ is $168$, and the order of $PSp_{4}(\bF_3)$ is $25920$.

Now we consider the unramified cohomology groups of $\bC(G)$ where $G$ is non-abelian simple.

\begin{theorem}\label{t5.2}
Let $A_n$ be the alternating group of degree $n$ and $l$ be a prime number.\\
{\rm (1) (Bogomolov, Maciel and Petrov \cite{BMP})}
${\rm Br}_{\rm nr}(\bC(A_n))=0$ for any $n \ge 5$.\\
{\rm (2) (Kunyavskii \cite{Kuny})}
${\rm Br}_{\rm nr}(\bC(G))=0$ for any $G$ which is non-abelian simple $($resp. almost simple$)$.\\
{\rm (3) (Bogomolov and Petrov \cite[Theorem 1.2]{BP}; Kriz \cite[Theorem 2]{Kr})}
$\iota(H^i(A_n,\bZ/l\bZ))\cap H^i_{\rm nr}(\bC(A_n),\bZ/l\bZ)=0$ for any $i \ge 2$ and 
any $n \ge 5$
where $\iota: H^i(A_n,\bZ/l\bZ)\rightarrow H^i(\bC(A_n),\bZ/l\bZ)$ is
the inflation map.
In particular, $H^i_{\rm nr}(\bC(A_n),\bZ/l\bZ)=0$ for $i=2,3$.
\end{theorem}

Note that a group $G$ is called {\it almost simple} if there is a non-abelian simple group $H$ such that $H\leq G\leq {\rm Aut}(H)$. Kunyavskii's theorem is valid also for most of the quasisimple groups; see \cite{Kuny} for details.

We remark that there is a correction \cite{BB2} for \cite{BP}. The reader need not worry about the correction, because Sophi Kriz gave another proof for Theorem \ref{t5.2} (3) in \cite[Theorem 2]{Kr}, which will be reproduced later.

Now we start the proof of Theorem \ref{t1.16}.

The following lemma is reformulated from a result of Kriz \cite[Lemma 10]{Kr}.
For the convenience of the reader, we provide a complete proof of it.

\begin{lemma}\label{l5.4}
Let $G$ be a finite group.
For each prime divisor $p$ of $|G|$, choose a $p$-Sylow subgroup $G_p$.
Suppose that $X=\{x_1,\ldots,x_n\}$ is a finite set on which $G$ acts faithfully, i.e. $G$ is presented as a subgroup of $S_n$.
Define a rational function field $\bC(x_1,\ldots,x_n)$ with a $G$-action defined by $\sigma\cdot x_i=x_{\sigma(i)}$ for any $\sigma\in G$ and any $1\leq i\leq n$. Define $K:=\bC(x_1,\ldots,x_n)^G$, $K^{(p)}:=\bC(x_1,\ldots,x_n)^{G_p}$ and $\Gamma_K:={\rm Gal}(K^{\rm sep}/K)$.
If $M$ is a torsion $($discrete$)$ $\Gamma_K$-module on which $\Gamma_K$ acts trivially $($e.g. $M=\bQ/\bZ$ or $\bZ/n\bZ$ for some positive integer $n$$)$, then the restriction map
\begin{align}
{\rm res}:H_{\rm nr}^d(K/\bC,M)\rightarrow\bigoplus_{p\,|\,|G|}H_{\rm nr}^d(K^{(p)}/\bC,M)\label{eq1}
\end{align}
is injective for any $d\geq 2$.
Consequently, there is an injective map
\begin{align}
H_{\rm nr}^d(\bC(G)/\bC,M)\rightarrow\bigoplus_{p\,|\,|G|}H_{\rm nr}^d(\bC(G_p)/\bC,M)\label{eq2}
\end{align}
for any $d\geq 2$.
\end{lemma}
\begin{proof}
Step 1.
We claim that $\bC(G)$ and $K$ are stably isomorphic over $\bC$.
Consider the rational function field $\bC(x_i,x(g):1\leq i\leq n,g\in G)$.
Apply \cite[Theorem 1]{HK}.
We find that $\bC(x_i,x(g):1\leq i\leq n,g\in G)^G=\bC(x_i:1\leq i\leq n)^G(u_1,\ldots,u_s)=K(u_1,\ldots,u_s)$ with $s=|G|$.
By applying \cite[Theorem 1]{HK} again (but in a different order), we also find that $\bC(x_i,x(g):1\leq i\leq n,g\in G)^G=\bC(x(g):g\in G)^G(v_1,\ldots,v_n)=\bC(G)(v_1,\ldots,v_n)$.

Similarly, $\bC(G_p)$ and $K^{(p)}$ are stably isomorphic over $\bC$.

By Theorem \ref{t2.7}, we find that $H_{\rm nr}^d(\bC(G)/\bC,M)\simeq H_{\rm nr}^d(K/\bC,M)$ and $H_{\rm nr}^d(\bC(G_p)/\bC,M)\simeq H_{\rm nr}^d(K^{(p)}/\bC,M)$.
{}From the injectivity of (\ref{eq1}) we find an injective map of (\ref{eq2}).\\

Step 2.
We will show that (\ref{eq1}) is injective.

Note that $K^{(p)}$ is a finite extension of $K$; in fact, $[K^{(p)}:K]=[G:G_p]$.

Let $v$ be a rank-one discrete $\bC$-valuation of $K$ with residue field $\bk_v$, and let $v^\prime$ be an extension of $v$ to $K^{(p)}$ with residue field $\bk_{v^\prime}$.
Then we have the following commutative diagram
\begin{align*}
\xymatrix{
&H^d(K,M)\ar[d]_{\rm res} \ar[r]^{r_v} \ar@{}[dr]|\circlearrowleft
& \ar[d]_{e} H^{d-1}(\bk_v,M)  \\
& H^d(K^{(p)},M) \ar[r]^{r_{v^\prime}} &  H^{d-1}(\bk_{v^\prime},M).
}
\end{align*}
where the left-hand-side vertical map is the restriction map and
the right-hand-side vertical map is the multiplication by $e$ which is the ramification index of $v^\prime$ over $v$ (see \cite[page 19, Proposition 8.2]{GMS}).
Consider all the discrete $\bC$-valuations of $K$.
We obtain the commutative diagram
\begin{align*}
\xymatrix{
0 \ar[r] & H_{\rm nr}^d(K/\bC,M)\ar[d]_{\rm res} \ar[r] \ar@{}[dr]|\circlearrowleft
& \ar[d]_{\rm res} H^d(K,M)  \\
0 \ar[r] & H_{\rm nr}^d(K^{(p)}/\bC,M) \ar[r] &  H^d(K^{(p)},M).
}
\end{align*}
Similarly we obtain the commutative diagram for the corestriction map
\begin{align*}
\xymatrix{
0 \ar[r] & H_{\rm nr}^d(K^{(p)}/\bC,M)\ar[d]_{\rm cor} \ar[r] \ar@{}[dr]|\circlearrowleft
& \ar[d]_{\rm cor} H^d(K^{(p)},M)  \\
0 \ar[r] & H_{\rm nr}^d(K/\bC,M) \ar[r] & H^d(K,M).
}
\end{align*}
Note that the composite map $H_{\rm nr}^d(K/\bC,M)\xrightarrow{\rm res} H_{\rm nr}^d(K^{(p)}/\bC,M)\xrightarrow{\rm cor} H_{\rm nr}^d(K/\bC,M)$ is the multiplication by $[G:G_p]$.

If $[\alpha]\in H_{\rm nr}^d(K/\bC,M)$ is a cohomological class whose image in $H_{\rm nr}^d(K^{(p)}/\bC,M)$ is zero for all $p\,|\,|G|$, then $[G:G_p]\cdot[\alpha]=0$ for all $p\,|\,|G|$.
It follows that $[\alpha]=0$ in $H_{\rm nr}^d(K/\bC,M)$.
\end{proof}

\bigskip
{\it Proof of Theorem \ref{t1.16}.}-----------------

For each group $G$ in Theorem \ref{t1.16}, we will show that $\bC(G_p)$ is $\bC$-rational or stably $\bC$-rational for all $p\,|\,|G|$.
It follows that $H_{\rm nr}^d(\bC(G_p),M)=0$ for all $d\geq 2$.

Apply Lemma \ref{l5.4}.
We find that $H_{\rm nr}^d(\bC(G),M)=0$.

For a finite group $H$, it is known that $\bC(H)$ is $\bC$-rational if $H$ satisfies one of the following conditions: (i) $H$ is an abelian group (see Theorem \ref{t1.2}); (ii) $H$ contains an abelian normal subgroup of index $\le 22$ (see \cite{Haj1}, \cite{Ka1}); (iii) $H$ is a $p$-group of order $\leq p^4$ (see Theorem \ref{t1.4}).

We will use above three criteria to show that $\bC(G_p)$ is $\bC$-rational or stably $\bC$-rational, and finish the proof of Theorem \ref{t1.16}.\\

Case 1. $G$ is the Mathieu group $M_{11}$ or $M_{12}$.

By \cite[page 190]{Wi}, $|M_{11}|=2^4\cdot 3^2\cdot 5\cdot 11$, $|M_{12}|=2^6\cdot 3^3\cdot 5\cdot 11$.

Since the $p$-Sylow subgroup of $M_{11}$ is of order $\leq p^4$, it follows that $\bC(G_p)$ is $\bC$-rational.

Consider $M_{12}$.
Let $G_2$ be a $2$-Sylow subgroup of $M_{12}$.

By \cite[page 203, Table 5.1]{Wi}, $2^{1+4}S_3$ is a maximal subgroup of $M_{12}$.
It follows that $G_2$ contains an abelian normal subgroup of index $2$.\\

Case 2. $G$ is the Janko group $J_1$.

By \cite[page 267]{Wi}, $|J_1|=2^3\cdot 3\cdot 5\cdot 7\cdot 11\cdot 19$.
Done.\\

Case 3 (Kriz). $G$ is the alternating group $A_n$.

When $G=A_3$, apply Theorem \ref{t1.2}. When $G=A_4$, the group $A_4$ contains an abelian normal subgroup $V$ ($V$ is the Klein four group) such that $[A_4 : V]= 3$. Apply \cite{Haj1}, \cite{Ka1}.

{}From now on till the end of this case, we consider $G=A_n$ with $n\geq 5$.

For any $n\geq 5$, let $G_p$ be a $p$-Sylow subgroup of $A_n$ where $p$ is any prime number with $p\,|\,|A_n|$.
We will show that $\bC(x_1,\ldots,x_n)^{G_p}$ is $\bC$-rational.
If $p=2$, then $\bC(x_1,\ldots,x_n)^{G_2}$ is $\bC$-rational by \cite[Theorem 1]{Kr}.
If $p\geq 3$, then $G_p$ is also a $p$-Sylow subgroup of $S_n$.
Since $\bC(x_1,\ldots,x_n)^{G_p}$ is $\bC$-rational by Tsunogai's Theorem (see \cite[Theorem 1.4]{KWZ}), we are done.

Note that $\bC(x_1,\ldots,x_n)^{G_p}$ and $\bC(G_p)$ are stably $\bC$-isomorphic, because we may apply \cite[Theorem 1]{HK} as in Step 1 of Lemma \ref{l5.4}.
Thus $\bC(G_p)$ is stably $\bC$-rational.\\

Case 4. $G$ is the group $SL_2(\bF_q)$, $PSL_2(\bF_q)$ or $PGL_2(\bF_q)$.

Let $G_p$ be a $p$-Sylow subgroup of $G$.
We claim that
(i) if $p\geq 3$ is odd, then $G_p$ is abelian;
(ii) if $p=2$, then $G_2$ contains an abelian normal subgroup of index $2$.
Once we prove the above result for $SL_2(\bF_q)$,
it is obvious that the the same result is valid for $PSL_2(\bF_q)$.

{}From now on till the end of the proof of this theorem, we write $G=SL_2(\bF_q)$,
$H=PGL_2(\bF_q)$, $G_p$ and $H_p$ stand for $p$-Sylow subgroups of $G$ and $H$ respectively.
Remember that $|G|=|H|=q(q^2-1)$.\\

Step 1. $G$ contains an element $T_1$ of order $q+1$ and an element $T_2$ of order $q-1$.

In fact, choose $\alpha\in\bF_q^\times$ such that $\bF_q^\times=\langle\alpha\rangle$.
Define
\begin{align*}
T_2=\left(
\begin{array}{cc}
\alpha & 0 \\
0 & \alpha^{-1}
\end{array}
\right)\in G.
\end{align*}
On the other hand, choose $\beta\in\bF_{q^2}^\times$ with
$\bF_{q^2}^\times=\langle\beta\rangle$.
Since $[\bF_q(\beta):\bF_q]=2$, the minimum polynomial of $\beta$ over $\bF_q$ is of the form
\begin{align*}
f(X)=X^2-bX-a\in\bF_q[X].
\end{align*}
Note that $f(X)=(X-\beta)(X-\beta^q)$.
The rational normal form of $f(X)$ is
\begin{align*}
T_0=\left(
\begin{array}{cc}
0 & a \\
1 & b
\end{array}
\right)\in GL_2(\bF_q).
\end{align*}
The matrix $T_0$ is conjugate to the diagonal matrix ${\rm diag}(\beta,\beta^q)$ within $GL_2(\bF_{q^2})$.
Hence $T_0^{q-1}$ is conjugate to ${\rm diag}(\beta^{q-1},\beta^{q(q-1)})$.
Define $T_1:=T_0^{q-1}$.
Clearly $T_1\in G=SL_2(\bF_q)$.\\

Step 2. We will show that $G_p$ is either abelian or contains an abelian normal subgroups of index $2$.

If $p\,|\,q$ where $q=l^m$ for some prime number $l$, then $G_p$ is isomorphic to the subgroup
\begin{align*}
\left\{\left(
\begin{array}{cc}
1 & a \\
0 & 1
\end{array}
\right)\in G:a\in\bF_q\right\}
\end{align*}
which is an elementary abelian group of order $q$.

If $p$ $\not{|}$ $q$ and $p$ is odd, then $p\,|\,q+1$ or $p$ $\not{|}$ $q-1$, but not both.
By Step 1, $G_p$ is a cyclic group.

Thus there remains only the situation: $p$ $\not{|}$ $q$ and $p=2$.

If $q+1=2^r\cdot s$ where $r\geq 2$ and $2$ $\not{|}$ $s$,
then $2\,|\,q-1$ but $2^2$ $\not{|}$ $q-1$.
Let $T_1\in SL_2(\bF_q)$ be the element of order $q+1$ in Step 1.
Then $\langle T_1^s\rangle$ is conjugate to a subgroup of $G_2$ of index $2$.

If $q-1=2^{r^\prime}\cdot s^\prime$ where $r^\prime\geq 2$ and $2$ $\not{|}$ $s^\prime$, take $\langle T_2^{s^\prime}\rangle$.\\

Step 3. We will show that $H_p$ is either abelian or contains an abelian normal subgroup of index $2$.

Let $T_1$ and $T_2$ be the matrices defined in Step 1.
As an element in $H=PGL_2(\bF_q)$, the order of $T_1$ (resp. $T_2$) may be reduced.
In fact, it is not difficult to verify that
(i) if $q$ is odd, then $T_1$ (resp. $T_2$) is of order $(q+1)/2$
(resp. $q-1$) as an element in $H$;
(ii) If $q$ is even, then $T_1$ (resp. $T_2$) is of order $q+1$
(resp. $(q-1)/2$).
The remaining proof is not difficult and is omitted.\qed

\bigskip
We remark that the method of this theorem may be applied to other finite groups, which will be discussed elsewhere.

\bigskip
Before giving the second proof of Theorem \ref{t1.16} for $H_{\rm nr}^3(\bC(G)/\bC,\bQ/\bZ)=0$ (with $G=PSL_2(\bF_8), A_6, A_7$), recall the following theorem of Saltman.

\begin{theorem}[{Saltman \cite[Theorem 4.13]{Sa5}}]\label{t5.3}
If $G$ is not a $p$-group, then $N^3(G)=0$ where $N^3(G)$ is defined as $N^3(G):=H^3_{\rm n}(G,\bQ/\bZ)/(H^3_{\rm c}(G,\bQ/\bZ)+H^3_{\rm p}(G,\bQ/\bZ))$
where $H^3_{\rm c}(G,\bQ/\bZ)=\sum_{H\lneq G}{\rm Cores}^G_H(H^3_{\rm n}(H,\bQ/\bZ))$.
\end{theorem}

We find that, if $G$ is a $p$-group ($p$ is an odd prime number), then $N^3(G)=0$ (because $H^3_{\rm p}(G,\bQ/\bZ)=H^3_{\rm n}(G,\bQ/\bZ)$ by Theorem \ref{t2.14}). However, Saltman \cite[Theorem 4.14]{Sa5} showed that if $G$ is a non-abelian
$2$-group with a cyclic subgroup of index $2$,
then $N^3(G)\simeq \bZ/2\bZ$.\\

Because the computation for non-abelian simple groups need more
computational time, we modify the command {\tt H4pFromResolution(}$RG${\tt )}
to {\tt H4pFromResolution(}$RG${\tt :H1trivial)} as follows.

By Proposition \ref{p2.12} and Definition \ref{d2.13},
it is possible to find an exact sequence of $\bZ[G]$-modules
$0 \to \mu \to Q^\ast \to Q \to 0$  where $Q$ is a permutation $\bZ[G]$-lattice and
$Q^\ast$ is a $H^1$-trivial $\bZ[G]$-module,
i.e. $H^1(H,Q^\ast)=0$ for all subgroups $H\leq G$. Using this exact sequence, Peyre finds a formula for $H^3_{\rm p}(G,\bQ/\bZ)$ (see Definition \ref{d2.13} and Theorem \ref{t2.14}). Namely, in \cite[pages 196--197]{Pe2}, Peyre constructs an exact sequence $0 \to \mu \to Q^\ast \to Q \to 0$ where the permutation
$\bZ[G]$-lattice $Q$ is of the form
\begin{align*}
Q=\bigoplus_{H\leq G}\bZ[G/H]^{d(H)}
\end{align*}
with $d(H)=|H^1(H,\bQ/\bZ)|=|\,{\rm Hom}(H,\bQ/\bZ)\,|=|H/D(H)|=|H^{ab}|$
and he proves that
\begin{align*}
H_{\rm p}^3(G,\bQ/\bZ)
&={\rm Ker}\{H^3(G,\bQ/\bZ)\rightarrow H^3(G,Q^\ast)\}\\
&={\rm Image}\{H^2(G,Q)\xrightarrow{\delta}H^3(G,\bQ/\bZ)\}\\
&=\sum_{H\leq G\atop\alpha\in H^1(H,\bQ/\bZ)}
{\rm Image}\{H^2(G,\bZ[G/H])\xrightarrow{\delta_{\alpha}}H^3(G,\bQ/\bZ)\}\\
&=\sum_{H\leq G\atop\alpha\in H^1(H,\bQ/\bZ)}
{\rm Cores}^G_H({\rm Image}\{\delta_\alpha:H^2(H,\bZ)\rightarrow
H^3(H,\bQ/\bZ),\,\beta\mapsto\alpha\cup\beta\})\\
&=\sum_{H\leq G}{\rm Cores}^G_H({\rm Image}\{H^1(H,\bQ/\bZ)\otimes
H^1(H,\bQ/\bZ)\rightarrow H^3(H,\bQ/\bZ)\}).
\end{align*}
The third equality comes from Shapiro's lemma and
the fourth equality follows from the compatibility with the
cup products (see \cite[pages 196--197]{Pe2}).

To facilitate the computer computing, we will modify the exact sequence $0 \to \mu \to Q^\ast \to Q \to 0$ constructed by Peyre so that the $\bZ$-rank of $Q$ becomes small enough.

The key idea in the new function
{\tt H4pFromResolution(}$RG${\tt :H1trivial)}
with {\tt H1trivial} option
is to reduce the number of direct summands $\bZ[G/H]$ of $Q$. We begin by choosing a subgroup $H$ of $G$ and throwing away the direct summand $\bZ[G/H]^{d(H)}$ from $Q$ in Peyre's exact sequence $0 \to \mu \to Q^\ast \to Q \to 0$. Let $Q_0$ be such a permutation lattice so obtained:
\begin{align*}
Q_0=\bigoplus_{H\leq G}\bZ[G/H]^{d^{\prime}(H)}
\leq Q=\bigoplus_{H\leq G}\bZ[G/H]^{d(H)}
\end{align*}
with $d^{\prime}(H)=0$ for the subgroup $H$ we choose (while $d^{\prime}(H)=d(H)$ for other subgroups $H$). Whenever $Q_0$ is chosen, a $\bZ[G]$-module $Q_0^\ast$ is constructed by taking the pull-back diagram of $Q_0 \to Q$ with respect to the original exact sequence $0 \to \mu \to Q^\ast \to Q \to 0$. Thus we get an exact sequence
$0 \to \mu \to Q_0^\ast \to Q_0 \to 0$. We require that $Q_0$ is chosen so that that $Q_0^\ast$ is an $H^1$-trivial $\bZ[G]$-module.

To guarantee that the resulting $Q_0^\ast$ is $H^1$-trivial, it is necessary and sufficient that the map $\delta$ is surjective where $\delta$ is the connecting homomorphism:
\begin{align*}
Q_0^{H^{\prime}}=H^0(H^{\prime},Q_0)\xrightarrow{\delta}H^1(H^{\prime},\bQ/\bZ)\rightarrow
H^1(H^{\prime},Q_0^\ast)\rightarrow H^1(H^{\prime},Q_0)=0
\end{align*}
where $H^{\prime}$ runs over all subgroups of $G$. To check the surjectivity of the map $\delta$, we may apply the function
{\tt chooseHH1trivial(}$G${\tt)}.
In this way, we obtain a ``permissible" permutation lattice $Q_0$
whose $\bZ$-rank is smaller than that of $Q$.

The above procedure for obtaining $Q_0$ may be applied to any subgroup $H$ of $G$. If we start with a subgroup $H$ and end with a $Q_0^\ast$ which is not $H^1$-trivial, we give up this subgroup and try another one. Fortunately we may find a ``good" subgroup $H$ when we work on the groups in Theorem \ref{t1.16}.

{}From a ``good" subgroup $H$ and the permutation lattice $Q_0$ (so that $Q_0^\ast$ is $H^1$-trivial), we try another subgroup of $G$ and repeat the same process. Thus we find $Q_1$ and an exact sequence $0 \to \mu \to Q_1^\ast \to Q_1 \to 0$. The process is continued till we find the ``minimal" $Q_n$ where $n$ is some non-negative integer such that $Q_n^\ast$ is $H^1$-trivial. We emphasize that the computing time is rather short. The resulting exact sequence $0 \to \mu \to Q_n^\ast \to Q_n \to 0$ fits our computer computing perfectly.

\bigskip
{\it A new proof of some cases of Theorem \ref{t1.16}.} -----------------

The proof consists of three steps: Step 1 for the group $PSL_2(\bF_8)$, Step 2 for the group $A_6$, Step 3 for $A_7$.

\bigskip
Step 1. The case $G=PSL_2(\bF_8)$. Note that $|G|=504=2^3\cdot 3^2\cdot 7$.

Using the functions
{\tt H4pFromResolution(}$RG${\tt )} and
{\tt IsUnramifiedH3($RG$,$l$)}
as in Steps 1--4 of Section \ref{sePT},
we have
$H^4(G,\bZ)\simeq\bZ/126\bZ=\langle f_1\rangle$,
$H^4_{\rm p}(G,\bZ)\simeq\bZ/63\bZ=\langle f_1^2\rangle$
and $f_1\not\in H^4_{\rm nr}(G,\bZ)$.
Hence we see
$H^4_{\rm nr}(G,\bZ)=H^4_{\rm n}(G,\bZ)=H^4_{\rm p}(G,\bZ)\simeq\bZ/63\bZ$
and $H^3_{\rm nr}(\bC(G),\bQ/\bZ)\simeq H^4_{\rm nr}(\bC(G),\bZ)\simeq
H^4_{\rm nr}(G,\bZ)/H^4_{\rm n}(G,\bZ)=0$.
Here is the computer implementation.\\

\begin{verbatim}
gap> Read("H3nr.gap");
gap> RPSL28:=ResolutionFiniteGroup(PSL(2,8),5);
Resolution of length 5 in characteristic 0 for
Group([ (3,8,6,4,9,7,5), (1,2,3)(4,7,5)(6,9,8) ]) .
gap> CR_CocyclesAndCoboundaries(RPSL28,4,true).torsionCoefficients; # H^4(G,Z)
[ [ 126 ] ]
gap> H4pFromResolution(RPSL28:H1trivial);
4[ [ 8, 5 ], [ 9, 1 ], [ 18, 1 ], [ 56, 11 ] ]
1/4
...
4/4
[ [ 63 ],
[ [ 126 ],
  [ [ 2 ] ] ] ]
gap> IsUnramifiedH3(RPSL28,[1]);
1/1
[ [ 2, 2, 2 ], [ 0, 0, 1 ] ]
false
\end{verbatim}

\bigskip
Step 2. The case $G=A_6$. Note that $|A_6|=360=2^3\cdot 3^2\cdot 5$
and $A_6 \simeq PSL_2(\bF_9)$, although we do
not use this fact in the following proof.

Using
{\tt H4pFromResolution(}$RG${\tt )} and
{\tt IsUnramifiedH3($RG$,$l$)}
as in Steps 1--5 of Section \ref{sePT},
we have
$H^4(G,\bZ)\simeq\bZ/60\bZ=\langle f_1\rangle$,
$H^4_{\rm p}(G,\bZ)\simeq\bZ/30\bZ=\langle f_1^2\rangle$
and $f_1\in H^4_{\rm nr}(G,\bZ)$.
Hence we see
$H^4(G,\bZ)=H^4_{\rm nr}(G,\bZ)\simeq\bZ/60\bZ=\langle f_1\rangle$
and $H^4_{\rm p}(G,\bZ)\simeq\bZ/30\bZ=\langle f_1^2\rangle$.

In order to find $H^4_{\rm n}(G,\bZ)\simeq H^3_{\rm n}(G,\bQ/\bZ)$,
we evaluate $H^4_{\rm c}(G,\bZ)\simeq H^3_{\rm c}(G,\bQ/\bZ)$
where $H_{\rm n}^4(G,\bZ)$ (resp. $H_{\rm c}^4(G,\bZ)$)
is the image of $H_{\rm n}^3(G,\bQ/\bZ)$
(resp. $H_{\rm c}^3(G,\bQ/\bZ)$)
in the isomorphism $H^3(G,\bQ/\bZ) \to H^4(G,\bZ)$. Since $G$ is not a $p$-group, we find that $H^4_{\rm n}(G,\bZ)=H^4_{\rm c}(G,\bZ)+H^4_{\rm p}(G,\bZ)$
by Theorem \ref{t5.3}.

As in Step 2 in Section \ref{sePT}, we have
\begin{align*}
H^3_{\rm c}(G,\bQ/\bZ)=
\sum_{H\lneq G}{\rm Cores}^G_H(H^3_{\rm n}(H,\bQ/\bZ))=
\sum_{H\lneq G:\, {\rm up\, to\, conjugacy}
\atop {H^\prime < H:\,{\rm maximal}}}
{\rm Cores}_H^G(H^3_{\rm n}(H,\bQ/\bZ)).
\end{align*}
By GAP, we find exactly $5$ maximal subgroups $H_1,H_2,H_3,H_4,H_5$ of
$G$ up to conjugation:
\begin{align*}
H_1&=\langle (456),(12)(3465)\rangle\simeq S_4 ,\\
H_2&=\langle (136)(254),(12)(3456)\rangle\simeq S_4,\\
H_3&=\langle (13)(56),(1526)(34)\rangle\simeq (C_3\times C_3)\rtimes C_4,\\
H_4&=\langle (12)(34),(236)\rangle\simeq A_5,\\
H_5&=\langle (12)(34),(134)(256)\rangle\simeq A_5.
\end{align*}
Using
{\tt H4pFromResolution(}$RG${\tt )} and
{\tt IsUnramifiedH3($RG$,$l$)} again,
we obtain
\begin{align*}
&H^4(H_1,\bZ)=H^4_{\rm nr}(H_1,\bZ)\simeq\bZ/2\bZ\oplus\bZ/12\bZ=\langle g_1,g_2\rangle,\
H^4_{\rm p}(H_1,\bZ)\simeq\bZ/2\bZ\oplus\bZ/6\bZ=\langle g_1,g_2^2\rangle,\\
&H^4(H_2,\bZ)=H^4_{\rm nr}(H_2,\bZ)\simeq\bZ/2\bZ\oplus\bZ/12\bZ=\langle g_3,g_4\rangle,\
H^4_{\rm p}(H_2,\bZ)\simeq\bZ/2\bZ\oplus\bZ/6\bZ=\langle g_3,g_4^2\rangle,\\
&H^4(H_3,\bZ)=H^4_{\rm p}(H_3,\bZ)\simeq\bZ/12\bZ=\langle g_5\rangle,\\
&H^4(H_4,\bZ)=H^4_{\rm p}(H_4,\bZ)\simeq\bZ/30\bZ=\langle g_6\rangle,\\
&H^4(H_5,\bZ)=H^4_{\rm p}(H_5,\bZ)\simeq\bZ/30\bZ=\langle g_7\rangle.
\end{align*}
Note that, for $1\leq i\leq 5$,
\begin{align*}
H^4(H_i,\bZ)\geq H^4_{\rm nr}(H_i,\bZ)\geq H^4_{\rm n}(H_i,\bZ)
\geq\ &H^4_{\rm c}(H_i,\bZ)+H^4_{\rm p}(H_i,\bZ)\geq H^4_{\rm p}(H_i,\bZ),
H^4_{\rm c}(H_i,\bZ).
\end{align*}
Because $\bC(S_4)$ is $\bC$-rational, we know that
$H^3_{\rm nr}(\bC(S_4),\bQ/\bZ)=0$ and hence
$H^3_{\rm nr}(S_4,\bQ/\bZ)=H^3_{\rm n}(S_4,\bQ/\bZ)$.
It follows that for $i=1,2$,
$H^4(H_i,\bZ)=H^4_{\rm nr}(H_i,\bZ)=H^4_{\rm n}(H_i,\bZ)
=H^4_{\rm c}(H_i,\bZ)+H^4_{\rm p}(H_i,\bZ)\simeq\bZ/2\bZ\oplus\bZ/12\bZ$ and
$H^4_{\rm p}(H_i,\bZ)\simeq\bZ/2\bZ\oplus\bZ/6\bZ$.
In particular, we have
$[H^4_{\rm n}(H_i,\bZ):H^4_{\rm p}(H_i,\bZ)]=2$ for $i=1,2$.

By applying the function {\tt Cores(}$l,CRH,CRG${\tt )} which
is already used in the function {\tt H4pFromResolution(}$RG${\tt )},
we obtain that
\begin{align*}
&{\rm Cores}^G_{H_1}(g_1)=0,\
{\rm Cores}^G_{H_1}(g_2)=f_1^{45},\\
&{\rm Cores}^G_{H_2}(g_3)=0,\
{\rm Cores}^G_{H_2}(g_4)=f_1^{15},\\
&{\rm Cores}^G_{H_3}(g_5)=f_1^{50},\\
&{\rm Cores}^G_{H_4}(g_6)=f_1^{54},\\
&{\rm Cores}^G_{H_5}(g_7)=f_1^{54}.
\end{align*}
Recall that $H^4_{\rm nr}(G,\bZ)\simeq\bZ/60\bZ=\langle f_1\rangle$
and $H^4_{\rm p}(G,\bZ)\simeq\bZ/30\bZ=\langle f_1^2\rangle$.
Thus we get $H^4(G,\bZ)=H^4_{\rm nr}(G,\bZ)
=H^4_{\rm n}(G,\bZ)=H^4_{\rm c}(G,\bZ)+H^4_{\rm p}(G,\bZ)=
H^4_{\rm c}(G,\bZ)
\simeq\bZ/60\bZ=\langle f_1\rangle$
and hence we have $H^3_{\rm nr}(\bC(G),\bQ/\bZ)\simeq
H^4_{\rm nr}(\bC(G),\bZ)\simeq H^4_{\rm nr}(G,\bZ)/H^4_{\rm n}(G,\bZ)=0$.
Here is the command of computer implementation.\\

\begin{verbatim}
gap> Read("H3nr.gap");
gap> A6:=AlternatingGroup(6);
Alt( [ 1 .. 6 ] )
gap> RA6:=ResolutionFiniteGroup(A6,5);
Resolution of length 5 in characteristic 0 for Group([ (4,5,6), (1,2,3,4,5) ]) .
gap> CR_CocyclesAndCoboundaries(RA6,4,true).torsionCoefficients; # H^4(G,Z)
[ [ 60 ] ]
gap> H4pFromResolution(RA6:H1trivial);
7[ [ 5, 1 ], [ 9, 2 ], [ 12, 3 ], [ 12, 3 ], [ 24, 12 ], [ 24, 12 ], [ 36, 9 ] ]
1/7
...
7/7
[ [ 30 ],
[ [ 60 ],
  [ [ 2 ] ] ] ]
gap> IsUnramifiedH3(RA6,[1]:Subgroup);
1/1
[ [ 2 ], [ 0 ] ]
true

gap> MaxA6:=MaximalSubgroups(A6);;
gap> HH:=Filtered(List(ConjugacyClassesSubgroups2(A6),Representative),
> x->x in MaxA6);
[ Group([ (4,5,6), (1,2)(3,4,6,5) ]),
  Group([ (1,3,6)(2,5,4), (1,2)(3,4,5,6) ]),
  Group([ (1,3)(5,6), (1,5,2,6)(3,4) ]), Group([ (1,2)(3,4), (2,3,6) ]),
  Group([ (1,2)(3,4), (1,3,4)(2,5,6) ]) ]
gap> List(HH,StructureDescription);
[ "S4", "S4", "(C3 x C3) : C4", "A5", "A5" ]

gap> S4:=SymmetricGroup(4);
Sym( [ 1 .. 4 ] )
gap> RS4:=ResolutionFiniteGroup(S4,5);
Resolution of length 5 in characteristic 0 for Group([ (1,2), (1,2,3,4) ]) .
gap> H4pFromResolution(RS4:H1trivial);
5[ [ 4, 2 ], [ 4, 1 ], [ 8, 3 ], [ 12, 3 ], [ 24, 12 ] ]
1/5
...
5/5
[ [ 2, 6 ],
[ [ 2, 12 ],
  [ [ 1, 0 ], [ 0, 2 ] ] ] ]
gap> IsUnramifiedH3(RS4,[0,1]:Subgroup);
1/2
[ [ 2 ], [ 0 ] ]
2/2
[ [ 2 ], [ 0 ] ]
true

gap> RH3:=ResolutionFiniteGroup(HH[3],5);
Resolution of length 5 in characteristic 0 for Group([ (1,3)(5,6), (1,5,2,6)
(3,4) ]) .
gap> H4pFromResolution(RH3:H1trivial);
2[ [ 9, 2 ], [ 36, 9 ] ]
1/2
2/2
[ [ 12 ],
[ [ 12 ],
  [ [ 1 ] ] ] ]

gap> A5:=AlternatingGroup(5);
Alt( [ 1 .. 5 ] )
gap> RA5:=ResolutionFiniteGroup(A5,5);
Resolution of length 5 in characteristic 0 for Group([ (3,4,5), (1,2,3,4,
5) ]) .
gap> H4pFromResolution(RA5:H1trivial);
4[ [ 4, 2 ], [ 5, 1 ], [ 10, 1 ], [ 12, 3 ] ]
1/4
...
4/4
[ [ 30 ],
[ [ 30 ],
  [ [ 1 ] ] ] ]
gap> RHH:=List(HH,x->ResolutionFiniteSubgroup(RA6,x));
[ Resolution of length 5 in characteristic 0 for S4 .
    , Resolution of length 5 in characteristic 0 for S4 .
    , Resolution of length 5 in characteristic 0 for (C3 x C3) : C4 .
    , Resolution of length 5 in characteristic 0 for A5 .
    , Resolution of length 5 in characteristic 0 for A5 .
]
gap> CRA6:=CR_CocyclesAndCoboundaries(RA6,4,true);;

gap> CRHH:=List(RHH,x->CR_CocyclesAndCoboundaries(x,4,true));;
gap> List(CRHH,x->x.torsionCoefficients);
[ [ 2, 12 ], [ 2, 12 ], [ 12 ], [ 30 ], [ 30 ] ]
gap> Cores([1,0],CRHH[1],CRA6);
[ 0 ]
gap> Cores([0,1],CRHH[1],CRA6);
[ 45 ]
gap> Cores([1,0],CRHH[2],CRA6);
[ 0 ]
gap> Cores([0,1],CRHH[2],CRA6);
[ 15 ]
gap> Cores([1],CRHH[3],CRA6);
[ 50 ]
gap> Cores([1],CRHH[4],CRA6);
[ 54 ]
gap> Cores([1],CRHH[5],CRA6);
[ 54 ]
\end{verbatim}

\bigskip
Step 3. The case $G=A_7$. Note that $|G|=2520=2^3\cdot 3^2\cdot 5\cdot 7$.
We will give two proofs.
Method 1
is based on the result of $A_6$ in Step 2
and Plans's theorem \cite[Theorem 2]{Pl}.
Method 2 is similar to the proof of $H^3_{\rm nr}(\bC(A_6),\bQ/\bZ)$
in Step 2.\\

Method 1.
By the No-Name Lemma (see \cite[Theorem 1]{HK}) the field $\bC(A_n)$ is rational over $\bC(x_1, \cdots, x_n)^{A_n}$. Thus $\bC(A_6)$ and $\bC(A_7)$ are stably $\bC$-isomorphic by \cite[Theorem 2]{Pl}.
{}From Step 2 and Theorem \ref{t2.7} (1), we find that
$H^3_{\rm nr}(\bC(A_7),\bQ/\bZ)=H^3_{\rm nr}(\bC(A_6),\bQ/\bZ)=0$.\\

Method 2.
Using
{\tt H4pFromResolution(}$RG${\tt )} and
{\tt IsUnramifiedH3($RG$,$l$)}, we find that
$H^4(G,\bZ)\simeq\bZ/12\bZ=\langle f_1\rangle$,
$H^4_{\rm p}(G,\bZ)\simeq\bZ/6\bZ=\langle f_1^2\rangle$
and $f_1\in H^4_{\rm nr}(G,\bZ)$.
Hence we have
$H^4(G,\bZ)=H^4_{\rm nr}(G,\bZ)\simeq\bZ/12\bZ=\langle f_1\rangle$
and $H^4_{\rm p}(G,\bZ)\simeq\bZ/6\bZ=\langle f_1^2\rangle$.

As in Step 2, we see that $H^4_{\rm n}(G,\bZ)=H^4_{\rm c}(G,\bZ)+H^4_{\rm p}(G,\bZ)$ by Theorem \ref{t5.3} because $G$ is not a $p$-group
where $H_{\rm n}^4(G,\bZ)$ (resp. $H_{\rm c}^4(G,\bZ)$)
is the image of $H_{\rm n}^3(G,\bQ/\bZ)$
(resp. $H_{\rm c}^3(G,\bQ/\bZ)$)
in the isomorphism $H^3(G,\bQ/\bZ) \to H^4(G,\bZ)$,
and we have
\begin{align*}
H^4_{\rm c}(G,\bZ)=
\sum_{H\lneq G}{\rm Cores}^G_H(H^3_{\rm n}(H,\bQ/\bZ))=
\sum_{H\lneq G:\, {\rm up\, to\, conjugacy}
\atop {H^\prime < H:\,{\rm maximal}}}
{\rm Cores}_H^G(H^3_{\rm n}(H,\bQ/\bZ)).
\end{align*}
By GAP, there exist $5$ maximal subgroups $H_1,H_2,H_3,H_4,H_5$ of
$G$ up to conjugation:
\begin{align*}
H_1&=\langle (23)(57),(12)(4567),(23)(56)\rangle\simeq
(C_3\times A_4)\rtimes C_2,\\
H_2&=\langle (12)(37),(2654)(37)\rangle\simeq S_5,\\
H_3&=\langle (14)(23),(246)(357)\rangle\simeq PSL_3(\bF_2)\simeq PSL_2(\bF_7),\\
H_4&=\langle (13)(27),(157)(346)\rangle\simeq PSL_3(\bF_2)\simeq PSL_2(\bF_7),\\
H_5&=\langle (24)(35),(2654)(37)\rangle\simeq A_6.
\end{align*}
Using
{\tt H4pFromResolution(}$RG${\tt )} and
{\tt IsUnramifiedH3($RG$,$l$)} again,
we get
\begin{align*}
&H^4(H_1,\bZ)=H^4_{\rm nr}(H_1,\bZ)\simeq\bZ/3\bZ\oplus\bZ/6\bZ\oplus\bZ/12\bZ=\langle g_1,g_2,g_3
\rangle,\
H^4_{\rm p}(H_1,\bZ)\simeq\bZ/3\bZ\oplus(\bZ/6\bZ)^{\oplus 2}=\langle g_1,g_2,g_3^2\rangle,\\
&H^4(H_2,\bZ)=H^4_{\rm nr}(H_2,\bZ)\simeq\bZ/2\bZ\oplus\bZ/12\bZ=\langle g_4,g_5\rangle,\
H^4_{\rm p}(H_2,\bZ)\simeq\bZ/2\bZ\oplus\bZ/6\bZ=\langle g_4,g_5^2\rangle,\\
&H^4(H_3,\bZ)=H^4_{\rm nr}(H_3,\bZ)\simeq\bZ/12\bZ=\langle g_6\rangle,\
H^4_{\rm p}(H_3,\bZ)\simeq\bZ/6\bZ=\langle g_6^2\rangle,\\
&H^4(H_4,\bZ)=H^4_{\rm nr}(H_4,\bZ)\simeq\bZ/12\bZ=\langle g_7\rangle,\
H^4_{\rm p}(H_4,\bZ)\simeq\bZ/6\bZ=\langle g_7^2\rangle,\\
&H^4(H_5,\bZ)=H^4_{\rm nr}(H_5,\bZ)\simeq\bZ/60\bZ=\langle g_8\rangle,\
H^4_{\rm p}(H_5,\bZ)\simeq\bZ/30\bZ=\langle g_8^2\rangle.
\end{align*}
The result for $H_5\simeq A_6$ can be obtained by Step 2.
By the definition, we have for $1\leq i\leq 5$,
\begin{align*}
H^4(H_i,\bZ)\geq H^4_{\rm nr}(H_i,\bZ)\geq H^4_{\rm n}(H_i,\bZ)
\geq H^4_{\rm c}(H_i,\bZ)+H^4_{\rm p}(H_i,\bZ)\geq H^4_{\rm p}(H_i,\bZ),
H^4_{\rm c}(H_i,\bZ).
\end{align*}
Because $\bC(S_5)$ and $\bC(PSL_3(\bF_2))$ is $\bC$-rational
(note that $PSL_3(\bF_2)\simeq PSL_2(\bF_7)$, see Theorem \ref{t5.1} (3)),
$H^3_{\rm nr}(H_i,\bQ/\bZ)=H^3_{\rm n}(H_i,\bQ/\bZ)$ for $i=2,3,4$.
By Step 2, we already know that
$H^4_{\rm nr}(H_5,\bZ)=H^4_{\rm n}(H_5,\bZ)$.

For 
$H_1$, by the same method as in Step 2, we obtain
$H^4_{\rm nr}(H_1,\bZ)=H^4_{\rm n}(H_1,\bZ)
=H^4_{\rm c}(H_1,\bZ)+H^4_{\rm p}(H_1,\bZ)$ and
$[H^4_{\rm n}(H_1,\bZ):H^4_{\rm p}(H_1,\bZ)]=2$
(we also see $H^4_{\rm n}(H_1,\bZ)=H^4_{\rm c}(H_1,\bZ)$).
Note that we get 
$H^3_{\rm nr}(\bC(H_1),\bQ/\bZ)=0$
for the first maximal subgroup $H_1$ with $|H_1|=72$
although the $\bC$-rationality of $\bC(H_1)$ is unclear.

In conclusion, for $1\leq i\leq 5$, we get
$H^4(H_i,\bZ)=H^4_{\rm nr}(H_i,\bZ)=H^4_{\rm n}(H_i,\bZ)=H^4_{\rm c}(H_i,\bZ)+H^4_{\rm p}(H_i,\bZ)$
and $[H^4_{\rm n}(H_i,\bZ):H^4_{\rm p}(H_i,\bZ)]=2$.

By applying the function {\tt Cores(}$l,CRH,CRG${\tt )}, we get
\begin{align*}
&{\rm Cores}^G_{H_1}(g_1)=f_1^4,\
{\rm Cores}^G_{H_1}(g_2)=f_1^2,\
{\rm Cores}^G_{H_1}(g_3)=f_1^9,\\
&{\rm Cores}^G_{H_2}(g_4)=f_1^6,\
{\rm Cores}^G_{H_2}(g_5)=f_1^9,\\
&{\rm Cores}^G_{H_3}(g_6)=f_1^3,\\
&{\rm Cores}^G_{H_4}(g_7)=f_1^9,\\
&{\rm Cores}^G_{H_5}(g_8)=f_1.
\end{align*}
Recall that we have
$H^4(G,\bZ)=H^4_{\rm nr}(G,\bZ)\simeq\bZ/12\bZ=\langle f_1\rangle$
and $H^4_{\rm p}(G,\bZ)\simeq\bZ/6\bZ=\langle f_1^2\rangle$.
Thus we get $H^4(G,\bZ)=H^4_{\rm nr}(G,\bZ)
=H^4_{\rm n}(G,\bZ)=H^4_{\rm c}(G,\bZ)+H^4_{\rm p}(G,\bZ)
=H^4_{\rm c}(G,\bZ)
\simeq\bZ/12\bZ=\langle f_1\rangle$.
Note that for $H_5\simeq A_6$, we already have
$H^4_{\rm n}(G,\bZ)={\rm Cores}^G_{H_5}(H^4_{\rm n}(H_5,\bZ))\simeq
\bZ/12\bZ$ with $H^4_{\rm n}(H_5,\bZ)\simeq\bZ/60\bZ$ (see also Step 2).
Therefore we have $H^3_{\rm nr}(\bC(G),\bQ/\bZ)\simeq
H^4_{\rm nr}(\bC(G),\bZ)\simeq H^4_{\rm nr}(G,\bZ)/H^4_{\rm n}(G,\bZ)=0$.
Here is the computer implementation.\\

\begin{verbatim}
gap> Read("H3nr.gap");
gap> A7:=AlternatingGroup(7);
Alt( [ 1 .. 7 ] )
gap> RA7:=ResolutionFiniteGroup(A7,5);
Resolution of length 5 in characteristic 0 for Group([ (5,6,7), (1,2,3,4,5,6,7) ]) .
gap> CR_CocyclesAndCoboundaries(RA7,4,true).torsionCoefficients; # H^4(G,Z)
[ [ 12 ] ]

gap> H4pFromResolution(RA7:H1trivial);
6[ [ 5, 1 ], [ 7, 1 ], [ 24, 12 ], [ 36, 11 ], [ 36, 9 ], [ 120, 34 ] ]
1/6
...
6/6
[ [ 6 ],
[ [ 12 ],
  [ [ 2 ] ] ] ]
gap> IsUnramifiedH3(RA7,[1]:Subgroup);
1/3
[ [ 2 ], [ 0 ] ]
...
3/3
[ [ 2 ], [ 0 ] ]
true

gap> MaxA7:=MaximalSubgroups(A7);;
gap> HH:=Filtered(List(ConjugacyClassesSubgroups2(A7),Representative),
> x->x in MaxA7);
[ Group([ (2,3)(5,7), (1,2)(4,5,6,7), (2,3)(5,6) ]),
  Group([ (1,2)(3,7), (2,6,5,4)(3,7) ]),
  Group([ (1,4)(2,3), (2,4,6)(3,5,7) ]),
  Group([ (1,3)(2,7), (1,5,7)(3,4,6) ]),
  Group([ (2,4)(3,5), (2,6,5,4)(3,7) ]) ]
gap> List(HH,StructureDescription);
[ "(C3 x A4) : C2", "S5", "PSL(3,2)", "PSL(3,2)", "A6" ]

gap> RH1:=ResolutionFiniteGroup(HH[1],5);
Resolution of length 5 in characteristic 0 for
Group([ (2,3)(5,6), (2,3)(5,7), (1,2)(4,5,6,7) ]) .
gap> H4pFromResolution(RH1:H1trivial);
5[ [ 12, 4 ], [ 12, 1 ], [ 24, 8 ], [ 36, 11 ], [ 72, 43 ] ]
1/5
...
5/5
[ [ 3, 6, 6 ], [ [ 3, 6, 12 ], [ [ 1, 0, 0 ], [ 0, 1, 0 ], [ 0, 0, 2 ] ] ] ]
gap> IsUnramifiedH3(RH1,[0,0,1]:Subgroup);
1/4
[ [ 2 ], [ 0 ] ]
...
4/4
[ [ 2 ], [ 0 ] ]
true
gap> MaxH1:=MaximalSubgroups(HH[1]);;
gap> H1H:=Filtered(List(ConjugacyClassesSubgroups2(HH[1]),Representative),
> x->x in MaxH1);
[ Group([ (2,3)(6,7), (1,3,2), (5,6,7) ]),
  Group([ (4,6)(5,7), (4,5)(6,7), (2,3)(6,7), (1,3,2) ]),
  Group([ (4,6)(5,7), (4,5)(6,7), (2,3)(6,7), (5,6,7) ]),
  Group([ (4,6)(5,7), (4,5)(6,7), (2,3)(6,7), (1,3,2)(5,6,7) ]),
  Group([ (4,6)(5,7), (4,5)(6,7), (2,3)(6,7), (1,3,2)(5,7,6) ]),
  Group([ (4,6)(5,7), (4,5)(6,7), (1,3,2), (5,6,7) ]) ]
gap> List(H1H,StructureDescription);
[ "(C3 x C3) : C2", "(C6 x C2) : C2", "S4", "S4", "S4", "C3 x A4" ]
gap> RH1H:=List(H1H,x->ResolutionFiniteSubgroup(RH1,x));
[ Resolution of length 5 in characteristic 0 for (C3 x C3) : C2 .
    , Resolution of length 5 in characteristic 0 for (C6 x C2) : C2 .
    , Resolution of length 5 in characteristic 0 for S4 .
    , Resolution of length 5 in characteristic 0 for S4 .
    , Resolution of length 5 in characteristic 0 for S4 .
    , Resolution of length 5 in characteristic 0 for C3 x A4 .
]
gap> CRH1H:=List(RH1H,x->CR_CocyclesAndCoboundaries(x,4,true));;
gap> CRH1:=CR_CocyclesAndCoboundaries(RH1,4,true);;
gap> CRH1.torsionCoefficients;
[ 3, 6, 12 ]
gap> List(CRH1H,x->x.torsionCoefficients);
[ [ 3, 3, 6 ], [ 2, 2, 12 ], [ 2, 12 ], [ 2, 12 ], [ 2, 12 ], [ 3, 3, 6 ] ]
gap> H4pFromResolution(RH1H[1],CRH1H[1]:H1trivial);
2[ [ 9, 2 ], [ 18, 4 ] ]
1/2
2/2
[ [ 3, 3, 6 ],
[ [ 3, 3, 6 ],
  [ [ 1, 0, 0 ], [ 0, 1, 0 ], [ 0, 0, 1 ] ] ] ]
gap> RH1H2:=ResolutionFiniteGroup(H1H[2],5);
Resolution of length 5 in characteristic 0 for
Group([ (4,5)(6,7), (4,6)(5,7), (2,3)(6,7), (1,3,2) ]) .
gap> H4pFromResolution(RH1H2:H1trivial);
4[ [ 12, 5 ], [ 12, 4 ], [ 12, 1 ], [ 24, 8 ] ]
1/4
...
4/4
[ [ 2, 2, 6 ],
[ [ 2, 2, 12 ],
  [ [ 1, 0, 0 ], [ 0, 1, 0 ], [ 0, 0, 2 ] ] ] ]
gap> IsUnramifiedH3(RH1H2,[0,0,1]:Subgroup);
1/4
[ [ 2 ], [ 0 ] ]
...
4/4
[ [ 2 ], [ 0 ] ]
true
gap> H4pFromResolution(RH1H[3],CRH1H[3]:H1trivial); # RH1H[3]=RH1H[4]=RH1H[5]
5[ [ 4, 2 ], [ 4, 1 ], [ 8, 3 ], [ 12, 3 ], [ 24, 12 ] ]
1/5
...
5/5
[ [ 2, 6 ],
[ [ 2, 12 ],
  [ [ 1, 0 ], [ 0, 2 ] ] ] ]
gap> IsUnramifiedH3(RH1H[3],[0,1]:Subgroup); # RH1H[3]=RH1H[4]=RH1H[5]
1/2
[ [ 2 ], [ 0 ] ]
2/2
[ [ 2 ], [ 0 ] ]
true
gap> H4pFromResolution(RH1H[6],CRH1H[6]:H1trivial);
2[ [ 12, 5 ], [ 36, 11 ] ]
1/2
2/2
[ [ 3, 3, 6 ],
[ [ 3, 3, 6 ],
  [ [ 1, 0, 0 ], [ 0, 1, 0 ], [ 0, 0, 1 ] ] ] ]
gap> List(CRH1H,x->x.torsionCoefficients);
[ [ 3, 3, 6 ], [ 2, 2, 12 ], [ 2, 12 ], [ 2, 12 ], [ 2, 12 ], [ 3, 3, 6 ] ]
gap> Cores([1,0,0],CRH1H[1],CRH1);
[ 0, 2, 4 ]
gap> Cores([0,1,0],CRH1H[1],CRH1);
[ 0, 2, 0 ]
gap> Cores([0,0,1],CRH1H[1],CRH1);
[ 1, 2, 0 ]
gap> Cores([1,0,0],CRH1H[2],CRH1);
[ 0, 3, 6 ]
gap> Cores([0,1,0],CRH1H[2],CRH1);
[ 0, 3, 6 ]
gap> Cores([0,0,1],CRH1H[2],CRH1);
[ 0, 0, 3 ]
gap> Cores([1,0],CRH1H[3],CRH1);
[ 0, 3, 6 ]
gap> Cores([0,1],CRH1H[3],CRH1);
[ 0, 3, 3 ]
gap> Cores([1,0],CRH1H[4],CRH1);
[ 0, 3, 6 ]
gap> Cores([0,1],CRH1H[4],CRH1);
[ 0, 0, 3 ]
gap> Cores([1,0],CRH1H[5],CRH1);
[ 0, 3, 6 ]
gap> Cores([0,1],CRH1H[5],CRH1);
[ 0, 0, 3 ]
gap> Cores([1,0,0],CRH1H[6],CRH1);
[ 1, 4, 8 ]
gap> Cores([0,1,0],CRH1H[6],CRH1);
[ 2, 4, 4 ]
gap> Cores([0,0,1],CRH1H[6],CRH1);
[ 0, 0, 2 ]

gap> S5:=SymmetricGroup(5);
Sym( [ 1 .. 5 ] )
gap> RS5:=ResolutionFiniteGroup(S5,5);
Resolution of length 5 in characteristic 0 for
Group([ (1,2), (1,2,3,4,5) ]) .
gap> H4pFromResolution(RS5:H1trivial);
6[ [ 5, 1 ], [ 8, 3 ], [ 12, 3 ], [ 12, 4 ], [ 20, 3 ], [ 120, 34 ] ]
1/6
...
6/6
[ [ 2, 6 ],
[ [ 2, 12 ],
  [ [ 1, 0 ], [ 0, 2 ] ] ] ]
gap> IsUnramifiedH3(RS5,[0,1]:Subgroup);
1/2
[ [ 2 ], [ 0 ] ]
2/2
[ [ 2 ], [ 0 ] ]
true

gap> PSL32:=PSL(3,2);
Group([ (4,6)(5,7), (1,2,4)(3,6,5) ])
gap> RPSL32:=ResolutionFiniteGroup(PSL32,5);
Resolution of length 5 in characteristic 0 for
Group([ (4,6)(5,7), (1,2,4)(3,6,5) ]) .
gap> H4pFromResolution(RPSL32:H1trivial);
5[ [ 4, 1 ], [ 7, 1 ], [ 21, 1 ], [ 24, 12 ], [ 24, 12 ] ]
1/5
...
5/5
[ [ 6 ],
[ [ 12 ],
  [ [ 2 ] ] ] ]
gap> IsUnramifiedH3(RPSL32,[1]:Subgroup);
1/1
[ [ 2 ], [ 0 ] ]
true

gap> RHH:=List(HH,x->ResolutionFiniteSubgroup(RA7,x));
[ Resolution of length 5 in characteristic 0 for (C3 x A4) : C2 .
    , Resolution of length 5 in characteristic 0 for S5 .
    , Resolution of length 5 in characteristic 0 for PSL(3,2) .
    , Resolution of length 5 in characteristic 0 for PSL(3,2) .
    , Resolution of length 5 in characteristic 0 for A6 .
]
gap> CRA7:=CR_CocyclesAndCoboundaries(RA7,4,true);;
gap> CRHH:=List(RHH,x->CR_CocyclesAndCoboundaries(x,4,true));;
gap> List(CRHH,x->x.torsionCoefficients);
[ [ 3, 6, 12 ], [ 2, 12 ], [ 12 ], [ 12 ], [ 60 ] ]
gap> Cores([1,0,0],CRHH[1],CRA7);
[ 4 ]
gap> Cores([0,1,0],CRHH[1],CRA7);
[ 2 ]
gap> Cores([0,0,1],CRHH[1],CRA7);
[ 9 ]
gap> Cores([1,0],CRHH[2],CRA7);
[ 6 ]
gap> Cores([0,1],CRHH[2],CRA7);
[ 9 ]
gap> Cores([1],CRHH[3],CRA7);
[ 3 ]
gap> Cores([1],CRHH[4],CRA7);
[ 9 ]
gap> Cores([1],CRHH[5],CRA7);
[ 1 ]
\end{verbatim}

\section{Algorithm for computing $H_{\rm nr}^3(\bC(G),\bQ/\bZ)$ and $H_{\rm s}^3(G,\bQ/\bZ)$}\label{seGAP}

We give some algorithms for computing
the unramified cohomology $H_{\rm nr}^3(\bC(G),\bQ/\bZ)$ of degree $3$
and the stable cohomology $H_{\rm s}^3(G,\bQ/\bZ)$ of degree $3$.

The functions which are provided in this section,
for example {\tt IsUnramifiedH3}, {\tt H4pFromResolution},
are available from
{\tt https://www.math.kyoto-u.ac.jp/\~{}yamasaki/Algorithm/UnramDeg3/}.\\

\begin{verbatim}
# The following algorithms need
# GAP version>=4.8.7 and GAP package HAP version>=1.11.15.

LoadPackage("HAP");

ConjugacyClassesSubgroups2:= function(G)
    Reset(GlobalMersenneTwister);
    Reset(GlobalRandomSource);
    return ConjugacyClassesSubgroups(G);
end;

BlockSum:= function(l,n)
  local k,ans,i;
  k:=Length(l)/n;
  ans:=[];
  for i in [1..k] do
    Add(ans,Sum([(i-1)*n+1..i*n],x->l[x]));
  od;
  return ans;
end;

Cores := function(arg)
  local v,CRH,CRG,RH,RG,n,RGH,map,G,H,u,k;
  v:=arg[1];
  if Length(arg)=3 then
    CRH:=arg[2];
    CRG:=arg[3];
    k:=Length(CRH.cocyclesBasis[1])/Length(CRG.cocyclesBasis[1]);
    return CRG.cocycleToClass(BlockSum(CRH.classToCocycle(v),k));
  else
    RH:=arg[2];
    RG:=arg[3];
    n:=arg[4];
    G:=RG!.group;
    H:=RH!.group;
    if Length(arg)>5 then
      RGH:=arg[5];
      map:=arg[6];
    else
      if Length(arg)=5 then
        RGH:=arg[5];
      else
        RGH:=ResolutionFiniteSubgroup(RG,H);
      fi;
      map:=HomToIntegers(EquivariantChainMap(RGH,RH,IdentityMapping(H)));
    fi;
  fi;
  u:=Map(map)(v,n);
  k:=Order(G)/Order(H);
  return BlockSum(u,k);
end;

AbelianInvariantsSNF := function(G)
  local n,m,s,l;
  n:=AbelianInvariants(G);
  m:=DiagonalMat(n);
  s:=SmithNormalFormIntegerMat(m);
  return Filtered(DiagonalOfMat(s),x -> x>1);
end;

chooseH := function(G)
  local hh,hhc,hhd,hhdtf,n,ic,i,jc,j;
  hh:=[];
  hhc:=List(ConjugacyClassesSubgroups2(G),Elements);
  hhd:=List(hhc,x->List(x,y->[y,DerivedSubgroup(y)]));
  hhdtf:=List(hhd,x->x[1][1]<>x[1][2]);
  n:=Length(hhd);
  for ic in [n,n-1..2] do
    if hhdtf[ic]=true then
      i:=hhd[ic][1];
      for jc in [1..ic-1] do
        for j in hhd[jc] do
          if i[2]=j[2] and IsSubgroup(i[1],j[1]) then
            hhdtf[jc]:=false;
            break;
          fi;
        od;
      od;
    fi;
  od;
  hh:=List(Filtered([1..n],x->hhdtf[x]),y->hhd[y][1][1]);
  return hh;
end;

chooseHH1trivial := function(G)
  local hh,hhc,kd,hhctf,n,i,j,H,K,HGK,gk,orb,ker,o,im,k,hom;
  hh:=[];
  hhc:=List(ConjugacyClassesSubgroups2(G),Representative);
  kd:=List(hhc,x->[x,DerivedSubgroup(x)]);
  hhctf:=List(kd,x -> x[1]<>x[2]);
  n:=Length(hhc);
  for i in [n,n-1..2] do
    if hhctf[i]=true then
      H:=hhc[i];
      for j in [2..i-1] do
        if hhctf[j]=true then
          K:=hhc[j];
          HGK:=DoubleCosetsNC(G,H,K);
          ker:=kd[j][1];
          for orb in HGK do
            o:=RepresentativesContainedRightCosets(orb);
            o:=List(o,x->RightCoset(H,x));
            im:=[];
            gk:=GeneratorsOfGroup(ker);
            for k in gk do
              Add(im,Product(OrbitsDomain(Group(k),o,OnRight),
                x->(k^Length(x))^(Representative(x[1])^-1)));
            od;
            Apply(im,x->
              Image(NaturalHomomorphismByNormalSubgroup(H,kd[i][2]),x));
            hom:=GroupHomomorphismByImages(ker,H/kd[i][2],gk,im);
            ker:=Kernel(hom);
          od;
          kd[j][1]:=ker;
          if kd[j][1]=kd[j][2] then
            hhctf[j]:=false;
          fi;
        fi;
      od;
    fi;
  od;
  hh:=List(Filtered([1..n],x->hhctf[x]),y->hhc[y]);
  return hh;
end;

HnSubgroupBase := function(tor,im)
  local torbase,im1,v,base1,base2,ans;
  torbase:=DiagonalMat(tor);
  im1:=LatticeBasis(Concatenation(torbase,im));
  im1:=Filtered(im1,x->not x in torbase);
  ans:=[];
  base1:=torbase;
  for v in im1 do
    base2:=LatticeBasis(Concatenation(base1,[v]));
    if base2<>base1 then
      Add(ans,v);
      base1:=base2;
    fi;
  od;
  return ans;
end;

H4pFromResolution := function(arg)
  local G,RG,CRG4,hh,H,RH,CRH2,CRH4,n,RGH,map,idH,idlist,Reslist,
    im1,im2,im,I,i,j,tor,torbase,H4,H4gen;
  RG:=arg[1];
  G:=RG!.group;
  if Length(arg)>1 then
    CRG4:=arg[2];
  else
    CRG4:=CR_CocyclesAndCoboundaries(RG,4,true);
  fi;
  if ValueOption("H1trivial")=true or ValueOption("H1Trivial")=true then
    hh:=chooseHH1trivial(G);
  else
    hh:=chooseH(G);
  fi;
  im1:=[];
  Print(Length(hh));Print(List(hh,
    function(H)
      if ID_AVAILABLE(Order(H))<>fail then
        return IdSmallGroup(H);
      else
        return [Order(H),"?"];
      fi;
    end
  ));
  Print("\n");
  idlist:=[];
  Reslist:=[];
  for H in hh do
Print(Position(hh,H));Print("/");Print(Length(hh));Print("\n");
    if H=G then
      CRH2:=CR_CocyclesAndCoboundaries(RG,2,true);
      CRH4:=CRG4;
      n:=Length(CRH2.torsionCoefficients);
      if n>0 then
        I:=IdentityMat(n);
        im2:=List(UnorderedTuples(I,2),
          v->IntegralCupProduct(RG,v[1],v[2],2,2,CRH2,CRH2,CRH4));
      else
        im2:=[];
      fi;
    else
      if ID_AVAILABLE(Order(H))<>fail then
        idH:=IdSmallGroup(H);
        i:=Position(idlist,idH);
        if i<>fail then
          RH:=Reslist[i][1];
          im2:=Reslist[i][2];
        else
          if IsNilpotentGroup(H) then
            RH:=ResolutionNormalSeries(LowerCentralSeries(SmallGroup(idH[1],idH[2])),5);
          else
            RH:=ResolutionFiniteGroup(SmallGroup(idH[1],idH[2]),5);
          fi;
          CRH2:=CR_CocyclesAndCoboundaries(RH,2,true);
          CRH4:=CR_CocyclesAndCoboundaries(RH,4,true);
          n:=Length(CRH2.torsionCoefficients);
          if n>0 then
            I:=IdentityMat(n);
            im2:=List(UnorderedTuples(I,2),
              v->IntegralCupProduct(RH,v[1],v[2],2,2,CRH2,CRH2,CRH4));
            im2:=HnSubgroupBase(CRH4.torsionCoefficients,im2);
            Apply(im2,CRH4.classToCocycle);
          else
            im2:=[];
          fi;
          Add(idlist,idH);
          Add(Reslist,[RH,im2]);
        fi;
      else
        if IsNilpotentGroup(H) then
          RH:=ResolutionNormalSeries(LowerCentralSeries(H),5);
        else
          RH:=ResolutionFiniteGroup(H,5);
        fi;
        CRH2:=CR_CocyclesAndCoboundaries(RH,2,true);
        CRH4:=CR_CocyclesAndCoboundaries(RH,4,true);
        n:=Length(CRH2.torsionCoefficients);
        if n>0 then
          I:=IdentityMat(n);
          im2:=List(UnorderedTuples(I,2),
            v->IntegralCupProduct(RH,v[1],v[2],2,2,CRH2,CRH2,CRH4));
          im2:=HnSubgroupBase(CRH4.torsionCoefficients,im2);
          Apply(im2,CRH4.classToCocycle);
        else
          im2:=[];
        fi;
      fi;
      RGH:=ResolutionFiniteSubgroup(RG,H);
      map:=HomToIntegers(EquivariantChainMap(RGH,RH,IsomorphismGroups(H,RH!.group)));
      im2:=List(im2,i->CRG4.cocycleToClass(Cores(i,RH,RG,4,RGH,map)));
    fi;
    im1:=Concatenation(im1,im2);
  od;
  tor:=CRG4.torsionCoefficients;
  torbase:=DiagonalMat(tor);
  im1:=LatticeBasis(Concatenation(torbase,im1));
  im1:=LatticeBasis(Difference(im1,torbase));
  H4:=AbelianGroup(CRG4.torsionCoefficients);
  H4gen:=GeneratorsOfGroup(H4);
  im:=Group(List(im1,x->Product([1..Length(H4gen)],y->H4gen[y]^x[y])),Identity(H4));
  return [AbelianInvariantsSNF(im),[tor,im1]];
end;

mast := function(RG,RH,RI,RJ)
  local I,J,Ig,Jg,iso,HI,RHI,m,eqchmap;
  I:=RI!.group;
  J:=RJ!.group;
  Ig:=First(I,x->Order(x)=Order(I));
  Jg:=First(J,x->Order(x)=Order(J));
  iso:=GroupHomomorphismByImages(I,J,[Ig],[Jg]);
  RHI:=ResolutionFiniteDirectProduct(RH,RI);
  HI:=RHI!.group;
  m:=GroupHomomorphismByFunction(HI,RG!.group,
    x->Image(Projection(HI,1),x)*Image(iso,Image(Projection(HI,2),x)));
  eqchmap:=EquivariantChainMap(RHI,RG,m);
  return eqchmap;
end;

i2ast := function(RHI,RI)
  local i2,eqchmap;
  i2:=Embedding(RHI!.group,2);
  eqchmap:=EquivariantChainMap(RI,RHI,i2);
  return eqchmap;
end;

pr2ast := function(RI,RHI)
  local pr2,eqchmap;
  pr2:=Projection(RHI!.group,2);
  eqchmap:=EquivariantChainMap(RHI,RI,pr2);
  return eqchmap;
end;

sHI := function(xi,RHI,RI,n)
  local i2star,pr2star;
  i2star:=HomToIntegers(i2ast(RHI,RI));
  pr2star:=HomToIntegers(pr2ast(RI,RHI));
  return xi-Map(pr2star)(Map(i2star)(xi,n),n);
end;

HnZtoHnminus1QoverZfromResolution := function(arg)
  local v,R,n,p,M,Mp,l,i,j,w,basis,cols,u,c,ansp,ans;
  v:=arg[1];
  R:=arg[2];
  n:=arg[3];
  M:=[];
  for i in [1..R!.dimension(n)] do
    l:=List([1..R!.dimension(n-1)],x->0);
    w:=R!.boundary(n,i);
    for j in w do
      l[AbsInt(j[1])]:=l[AbsInt(j[1])]+SignInt(j[1]);
    od;
    Add(M,l);
  od;
  M:=TransposedMat(M);
  if Length(arg)>3 then
    p:=arg[4];
    Mp:=List(p,x->M[x]);
  else
    p:=[1..R!.dimension(n-1)];
    Mp:=M;
  fi;
  basis:=NormalFormIntMat(Mp,6);
  cols:=List([1..basis.rank],i->First([1..R!.dimension(n)],
    j->basis.normal[i][j]<>0));
  ansp:=List([1..Length(p)],x->0);
  u:=v;
  for i in [1..basis.rank] do
    c:=u[cols[i]]/basis.normal[i][cols[i]];
    ansp:=ansp+c*basis.rowtrans[i];
    u:=u-c*basis.normal[i];
  od;
  if Set(u)<>[0] then
    return fail;
  fi;
  ans:=List([1..R!.dimension(n-1)],x->0);
  for i in [1..Length(p)] do
    ans[p[i]]:=ansp[i];
  od;
  return ans;
end;

Hnminus1QoverZtoHnZfromResolution := function(v,R,n)
  local M,l,i,j,w;
  M:=[];
  for i in [1..R!.dimension(n)] do
    l:=List([1..R!.dimension(n-1)],x->0);
    w:=R!.boundary(n,i);
    for j in w do
      l[AbsInt(j[1])]:=l[AbsInt(j[1])]+SignInt(j[1]);
    od;
    Add(M,l);
  od;
  M:=TransposedMat(M);
  return v*M;
end;

HS := function(f,RHI,n)
  local fQoverZ,d1,d2;
  if n < 2 then
    Print(
     "ERROR: n must be at least 2. \n" );
    return fail;
  fi;
  d1:=RHI!.dimension(n-1)-RHI!.dimension(n-2);
  d2:=RHI!.dimension(n-1)-RHI!.dimension(n-3);
  fQoverZ:=HnZtoHnminus1QoverZfromResolution(f,RHI,n,[1..d2]);
  return fQoverZ{[d1+1..d2]};
end;

GeneratorOfCyclicGroup := function(C)
  return First(C,x->Order(x)=Order(C));
end;

chooseHI := function(G)
  local I,HI,HItf,n,i,j,ic,jc;
  I:=List(Filtered(ConjugacyClassesSubgroups2(G),
    x->IsCyclic(Representative(x))),Elements);
  HI:=List(I,x->List(x,y->[Centralizer(G,y),y]));
  HItf:=List(HI,x->true);
  n:=Length(HI);
  for ic in [n,n-1..2] do
    if HItf[ic]=true then
      i:=HI[ic][1];
      if GroupCohomology(i[1],3)=[] then
        HItf[ic]:=false;
      fi;
      for jc in [1..ic-1] do
        for j in HI[jc] do
          if i[1]=j[1] and IsSubgroup(i[2],j[2]) then
            HItf[jc]:=false;
            break;
          fi;
        od;
      od;
    fi;
  od;
  HI:=List(Filtered([1..n],x->HItf[x]),y->HI[y][1]);
  Apply(HI,x->[x[1],Group(GeneratorOfCyclicGroup(x[2]))]);
  HI:=Filtered(HI,x->Order(x[2])>1);
  return HI;
end;

chooseHISubgroup := function(G)
  local I,HI,HItf,n,i,j,ic,jc,idH,Hs3trivial27and125and343,Hs3trivial81,
    Hs3trivial625and2401,Hs3trivial243,Hs3trivial3125;
  Hs3trivial27and125and343:=[1..4];
  Hs3trivial81:=Concatenation([[1..6],[8..10],[13,14]]);
  Hs3trivial625and2401:=Concatenation([[1..6],[9,10],[13,14]]);
  Hs3trivial243:=Concatenation([[1],[3..12],[19..27],[33],[35],[43..47],[49,50],[66]]);
  Hs3trivial3125:=Concatenation([[1],[15..17],[24..29],[42],[44],[51..57],[59,60],[76]]);
  I:=List(Filtered(ConjugacyClassesSubgroups2(G),
    x->IsCyclic(Representative(x))),Elements);
  HI:=List(I,x->List(x,y->[Centralizer(G,y),y]));
  HItf:=List(HI,x->true);
  n:=Length(HI);
  for ic in [n,n-1..2] do
    if HItf[ic]=true then
      i:=HI[ic][1];
      if GroupCohomology(i[1],3)=[] then
        HItf[ic]:=false;
      fi;
      if ID_AVAILABLE(Order(i[1]))<>fail then
        idH:=IdSmallGroup(i[1]);
        if idH[1] in [9,25,49] then
          HItf[ic]:=false;
        elif idH[1] in [27,125,343] and idH[2] in Hs3trivial27and125and343 then
          HItf[ic]:=false;
        elif idH[1]=81 and idH[2] in Hs3trivial81 then
          HItf[ic]:=false;
        elif idH[1] in [625,2401] and idH[2] in Hs3trivial625and2401 then
          HItf[ic]:=false;
        elif idH[1]=243 and idH[2] in Hs3trivial243 then
          HItf[ic]:=false;
        elif idH[1]=3125 and idH[2] in Hs3trivial3125 then
          HItf[ic]:=false;
        fi;
      fi;
      for jc in [1..ic-1] do
        for j in HI[jc] do
          if i[1]=j[1] and IsSubgroup(i[2],j[2]) then
            HItf[jc]:=false;
            break;
          fi;
        od;
      od;
    fi;
  od;
  HI:=List(Filtered([1..n],x->HItf[x]),y->HI[y][1]);
  Apply(HI,x->[x[1],Group(GeneratorOfCyclicGroup(x[2]))]);
  HI:=Filtered(HI,x->Order(x[2])>1);
  return HI;
end;

IsUnramifiedH3 := function(RG,v)
  local G,CRG4,u,hi,k,H,I,RH,RI,RJ,RHI,m,mstar,mu,smu,HSmu,ZHSmu,CRH3,c;
  G:=RG!.group;
  if Length(v)=RG!.dimension(4) then
    u:=v;
  else
    CRG4:=CR_CocyclesAndCoboundaries(RG,4,true);
    u:=CRG4!.classToCocycle(v);
  fi;
  if ValueOption("subgroup")=true or ValueOption("Subgroup")=true then
    hi:=chooseHISubgroup(G);
  else
    hi:=chooseHI(G);
  fi;
  for k in hi do
    H:=k[1];
    I:=k[2];
    if H=G then
      RH:=RG;
    elif IsNilpotentGroup(H) then
      RH:=ResolutionNormalSeries(LowerCentralSeries(H),5);
    else
      RH:=ResolutionFiniteGroup(H,5);
    fi;
    RI:=ResolutionAbelianGroup([Order(I)],5);
    RJ:=ResolutionAbelianGroup(I,5);
    m:=mast(RG,RH,RI,RJ);
    RHI:=m!.source;
    mstar:=HomToIntegers(m);
    mu:=Map(mstar)(u,4);
    smu:=sHI(mu,RHI,RI,4);
    Print(Position(hi,k));Print("/");Print(Length(hi));Print("\n");
    HSmu:=HS(smu,RHI,4);
    ZHSmu:=Hnminus1QoverZtoHnZfromResolution(HSmu,RH,3);
    CRH3:=CR_CocyclesAndCoboundaries(RH,3,true);
    c:=[CRH3!.torsionCoefficients,CRH3!.cocycleToClass(ZHSmu)];
Print(c);Print("\n");
    if c[2]=fail then
      return fail;
    elif LatticeBasis(c)<>[c[1]] then
      return false;
    fi;
  od;
  return true;
end;
\end{verbatim}

%
%


\end{document}